\documentclass[11pt, a4paper]{amsart}
\usepackage{amsmath,amssymb,amscd,amsthm, url, amsfonts}
\usepackage[all]{xy}
\usepackage{enumerate}
\usepackage{fullpage}
\usepackage[dvipdfmx]{graphicx}
\usepackage{wrapfig}
\usepackage{mathtools}
\usepackage{ascmac}
\usepackage{tikz}
\usepackage{setspace}
\usepackage[dvipdfmx]{hyperref}
\usepackage{mathrsfs}
\usepackage[mathscr]{euscript}

\usepackage{aurical}
\usepackage[T1]{fontenc}
\theoremstyle{plain}
\numberwithin{equation}{section}
\newtheorem{thm}{Theorem}[section]
\newtheorem{prop}[thm]{Proposition}
\newtheorem{prop-dfn}[thm]{Proposition-Definition}
\newtheorem{cor}[thm]{Corollary}
\newtheorem{lem}[thm]{Lemma}
\newtheorem{claim}[thm]{Claim}
\newtheorem{prob}[thm]{\textcolor{black}{Problem}}

\theoremstyle{definition}
\newtheorem{dfn}[thm]{Definition}

\newtheorem{rmk}[thm]{Remark}

\newcommand{\noopsort}[1]{}

\def\rank{\mathop{\mathrm{rank}}\nolimits}
\def\dim{\mathop{\mathrm{dim}}\nolimits}
\def\Im{\mathop{\mathrm{Im}}\nolimits}
\def\Ker{\mathop{\mathrm{Ker}}\nolimits}

\def\Hom{\mathop{\mathrm{Hom}}\nolimits}

\def\fib{\mathop{\mathsf{fib}}\nolimits}
\def\cof{\mathop{\mathsf{cof}}\nolimits}
\def\ker{\mathop{\mathsf{ker}}\nolimits}
\def\cok{\mathop{\mathsf{cok}}\nolimits}
\def\im{\mathop{\mathsf{im}}\nolimits}
\def\<{{\langle}}
\def\>{{\rangle}}
\def\Mu{\mathit M_{\beta, \omega}^{Z_{1}, Z_{2}}}

\def\+{\mathop{\oplus}\nolimits}

\def\1{\mathop{\mathrm{id}}\nolimits}
\def\Spec{\mathop{\mathrm{Spec}}\nolimits}

\def\R{\mathop{\mathfrak{Re}}\nolimits}
\def\I{\mathop{\mathfrak{Im}}\nolimits}

\newcommand{\m}[1]{\mathop{(\mathrm{Mor{#1}})}}

\newcommand{\gl}[2]{{\mathsf{gl}\left({#1},  {#2}\right)}}
\newcommand{\tilt}[3]{{{#1}_{\beta, \omega}^{{#2}, {#3}}}}
\newcommand{\til}[3]{{{#1}_{{#2}, {#3}}^{\sigma_{1}, \sigma_{2}}}}

\newcommand{\met}[2]{{\mathsf{d}\left({#1},  {#2}\right)}}
\newcommand{\HNF}[3]{{
\xymatrix{
0	\ar[r]	&	{#2}_1	\ar[r]\ar[d]	&	{#2}_2	\ar[r]\ar[d]	&	\cdots\ar[r]	&	{#2}_{n-1}	\ar[r]\ar[d]	&	{#2}_n={#1}\ar[d]	\\
			&	{#3}_1\ar@{-->}[ul]	&	{#3}_2\ar@{-->}[ul] &					&	{#3}_{n-1}\ar@{-->}[ul]		&	{#3}_n\ar@{-->}[ul]	
}
}}

\newcommand{\ho}[1]{\mathrm{h}({#1})}   
\newcommand{\sod}[2]{{\left\langle {#1},{#2} \right\rangle}}

\newcommand{\Aut}[1]{{\mathrm{Aut}\,{#1}}}
\newcommand{\Stab}[1]{{\mathsf{Stab}\,{#1}}}

\newcommand{\Stabr}[1]{{\mathsf{Stab}^{\mathsf{r}}\,{#1}}}
\newcommand{\Stabf}[1]{{\mathsf{Stab}^{\mathsf{fl}}\,{#1}}}

\newcommand{\bb}[1]{{\mathbb{#1}}}
\newcommand{\mca}[1]{{\mathcal{#1}}}
\newcommand{\mr}[1]{{\mathrm{#1}}}
\newcommand{\ms}[1]{{\mathscr{#1}}}
\newcommand{\mb}[1]{{\mathbf{#1}}}

\title{On a deformation of gluing stability conditions}
\author[K. Kawatani]{Kotaro Kawatani}
\email{kawatanikotaro@gmail.com}
\address{Osaka Prefecture University, Gakuen-cho, Sakai-city, Osaka, Japan. Osaka /Yamato University, Katayama-cho, Suita-city, Osaka, Japan. Osaka }
\date{\today}


\begin{document}
\maketitle

\begin{abstract}
On a triangulated category $\mathbf D$ equipped with a semiorthogonal decomposition $\mathbf D=\langle{\mathbf D_{1}},{\mathbf D_{2}}\rangle$, 
Collins and Polishchuk develop a gluing construction of stability condition on $\mathbf D$. 
The gluing construction gives a stability condition on $\mathbf D$ from these on $\mathbf D_{1}$ and $\mathbf D_{2}$. 
We study a deformation of gluing stability conditions on for a nice semiorthogonal decomposition. 
As a consequence, we construct a continuous family of stability conditions by showing a deformation property introduced by Bridgeland's original paper. 
Here the deformation property is weaker than the support property which is the standard solution for the continuousness. 
After proving the continuousness of the family, we show that each stability condition in the family satisfies the support property via specialization. 
More precisely we find a stability condition with support property at the boundary of the family. 
Finally applying these results, we study the space of stability conditions on the category of morphisms in a triangulated category. 
\end{abstract}

\section{Introduction}

Motivated by Douglas's work for $\Pi$-stability, Bridgeland \cite{MR2373143} introduced the notion of stability conditions on a triangulated category. 
Let $\Stab{\mb D}$ be the set of locally finite stability conditions on a triangulated category $\mb D$. 
By virtue of \cite{MR2373143}, $\Stab{\mb D}$ is not only a set, but also a topological space. 
Moreover each connected component of $\Stab{\mb D}$ is a complex manifold if the rank of Grothendieck group $K_{0}(\mb D)$ is finite. 
Thus $\Stab{\mb D}$ is well understood locally and geometric properties of $\Stab{\mb D}$ is expected to reflect categorical properties of $\mb D$ (see also \cite{MR3592689}, \cite{MR2376815}).

On the other hand, 
global description of $\Stab{\mb D}$ is difficult.  
For instance the connectedness, which is the most fundamental topological property, is widely open. 
If $\mb D$ is the triangulated category of $A_{n}$-singularity, 
then Ishii, Ueda and Uehara showed that $\Stab {\mb D}$ is connected (the details are in \cite{MR2629510}). 
However for the other Kleinian singularities, the connectedness is still open. 
To discuss other cases, suppose that $\mb D$ is the bounded derived category $\mb D^{b}(X)$ of coherent sheaves on a smooth projective variety $X$. 
If $\dim X=1$, then $\Stab{\mb D^{b}(X)}$ is non-empty and connected by \cite{MR2373143} and \cite{MR2335991}. 
If $\dim X=2$, then $\Stab{\mb D^{b}(X)}$ is non-empty by \cite{MR2998828} but the connectedness is open. 
If $\dim X=3$, then the non-emptiness of $\Stab{\mb D^{b}(X)}$ follows from the generalized Bogomolov inequality proposed in  \cite{MR3121850} but the connectedness is widely open.

The global description of $\Stab{\mb D}$ is difficult but the space $\Stab{\mb D}$ is expected to be contractible (unless it is empty), that is, the homotopy type of $\Stab{\mb D}$ is the simplest in non-empty topological spaces. 
Motivated by the expectation, 
the author studied the space of stability conditions on the category of morphisms in a triangulated category $\mb D$ in the previous paper \cite{morphismstability}. 
Unfortunately the category of morphisms in a triangulated category is not triangulated. 
To solve this, let us suppose that the triangulated category $\mb D$ is the homotopy category $\ho{\ms C}$ of a stable infinity category $\ms C$. 
Then the infinity category $\ms C^{\Delta^{1}}$ of morphisms in the infinity category $\ms C$ is stable by \cite{higheralgebra}, and 
the homotopy category $\ho{\ms C^{\Delta^{1}}}$ is triangulated. 
Thus $\ho{\ms C^{\Delta^{1}}}$ is a reasonable candidate of a triangulated category of morphisms in $\mb D=\ho{\ms C}$. 

Although not directly related to the theme of this article, we would like to introduce some interesting aspects of the category of morphisms. 
Suppose that $\ms C$ is the derived (infinity) category of the affine scheme $\Spec \mb k$ of a field $\mb k$. 
Then the category $\ho{\ms C^{\Delta^{1}}}$ is equivalent to the derived category of the representation of $A_{2}$-quiver by \cite[Corollary 6.2]{morphismstability}. 
Hence the category $\ho{\ms C^{\Delta^{1}}}$ for a stable infinity category $\ms C$ can be regarded as a generalization of the derived category of the quiver representations. 
In addition, recalling that the objects in the derived category $\mb D^{b}(X)$ of a projective variety represent (topological) D-branes, 
one can expect the category $\ho{\ms C^{\Delta^{1}}}$ relates to open strings in string theory (for instance see \cite{MR2567952}).

Our basic motivation is to understand a relation between $\Stab{\ho{\ms C}}$ and $\Stab{\ho{\ms C^{\Delta^{1}}}}$. 
Before recalling our previous results, let us recall basic notation. 
Since any object $f \in \ms C^{\Delta^{1}}$ is a morphism in $\ms C$, 
the object $f$ is written as $[f \colon x \to y]$ for objects $x$ and $y $ in $\ms C$. 
Under these notation, we have three natural functors (as infinity categories) between $\ms C$ and $\ms C^{\Delta^{1}}$
\begin{equation}\label{eq:threefunctors}
\xymatrix{
\ms	C	\ar[r]|(.4)s&\ar@/_8pt/[l]|{d_0}\ar@/^8pt/[l]|{d_1}	\ms C^{\Delta^1}
}
; d_0 \dashv s \dashv d_1, 
\end{equation}
where $d_{0}$ sends the object $[f\colon x \to y] \in \ms C^{\Delta^{1}}$ to $y \in \ms C$, 
$d_{1}$ sends $f \in \ms C^{\Delta^{1}}$ to $x \in \ms C$, 
and $s$ sends an object $z \in \ms C$ to the object $\1_{z} \in \ms C^{\Delta^{1}}$ which represents the identity morphism $[\1 \colon z \to z]$ in $\ms C$. 
These functors give adjoint pairs of functors.  
Precisely $d_{0}$ is a left adjoint of $s$, and $d_{1}$ is a right adjoint of $s$. 

In \cite{morphismstability}, we have shown the following:

\begin{thm}[{\cite[Theorem 1.2]{morphismstability}}]
\label{thm1-1}
Let $\ms  C$ be a stable infinity category with $\rank K_{0}(\ho{\ms C})< \infty$.  
\begin{enumerate}
\item \label{main1-1}If there exists a reasonable stability condition on $\ho{\ms C}$ then there exists a reasonable stability condition on $\ho{\ms C^{\Delta^{1}}}$. 
Moreover both functors $d_0$ and $d_1$ induce continuous and injective maps $d_0^*$ and $d_1^*$ from the space $\Stabr {\ho {\ms C}}$ of reasonable stability conditions on $\ho{\ms C}$ to that of $\ho {\ms C^{\Delta^{1}}}$: 
\[
d_0^*, d_1^* \colon \xymatrix{ \Stabr {\ho{\ms C}}	\ar@<0.6ex>[r] \ar@<-0.6ex>[r] 	& \Stabr {\ho{\ms C^{\Delta^{1}}}}	}. 
\]
\item \label{main1-3}Both images $\Im d_0^*$ and $\Im d_1^*$ are closed in $\Stabr {\ho{\ms C^{\Delta^{1}}}}$ and do not intersect each other. 
\item \label{main1-4}A stability condition $\sigma$ is full if and only if $d_0^* \sigma$ (or $d_1^* \sigma$) is full. 
\end{enumerate}
\end{thm}
The category $\ho{\ms C^{\Delta^{1}}}$ has two semiorthogonal decompositions $\sod{\ho{\ms C}}{\ho{\ms C_{/0}}}$ and $\sod{\ho{\ms C_{0/}}}{\ho{\ms C}}$ (see also (\ref{eq:0-}) and (\ref{eq:-0})) corresponds to the adjoint pairs $d_{0} \dashv s$ and $s \dashv d_{1}$. 
Collins-Polishchuk \cite{MR2721656} constructed a ``gluing" stability condition from a semiorthogonal decomposition of a triangulated category. 
A reasonable stability condition introduced by \cite{MR2721656}  is necessary for the gluing construction. 
The construction of continuous maps $d_{1}$ and $d_{0}$ is based on the gluing construction. 
Hence we had to focus on reasonable stability conditions. 
Since a full stability condition is reasonable, $\Stab{\ho{\ms C^{\Delta^{1}}}}$ is non-empty when 
there exists a full stability condition on $\ho{\ms C}$ by Theorem \ref{thm1-1}. 

In this article, we would like to understand the relation between $\Stab{\ho{\ms C}}$ and $\Stab{\ho{\ms C^{\Delta^{1}}}}$ deeper. 
Several problems might be considered for the sake and the following is one of these: 
\begin{prob}\label{prob:kihon}
Let $\ms C$ be a stable infinity category and $\sigma$ a stability condition on $\ho{\ms C}$. 
Is $d_{0}^{*}\sigma$ path connected to $d_{1}^{*}\sigma$?
\end{prob}

Once recalling that the space of stability conditions is expected to be contractible, 
one can guess that both $\Stab{\ho{\ms C}}$ and $\Stab{\ho{\ms C^{\Delta^{1}}}}$ are contractible, in particular, they are homotopy equivalent to each other. 
In fact, if $\ho{\ms C}$ is the derived category of an affine Noetherian scheme, $\Stab{\ho{\ms C}}$ and $\Stab{\ho{\ms C^{\Delta^{1}}}}$ are homotopy equivalent by the author \cite{kawatani2020stability}. 
However it seems hard to show the homotopy equivalence in general, 
because of the difficulty of the global description. 


Now if the answer to Problem \ref{prob:kihon} is negative, $\Stab{\ho{\ms C}}$ and $\Stab{\ho{\ms C^{\Delta^{1}}}}$ might not be homotopy equivalent. 
Recalling that there is no example of $\mb D$ whose spaces of full stability conditions is non-connected, 
it is natural to expect to an affirmative answer to Problem \ref{prob:kihon} at least for a full stability condition on $\ho{\ms C}$. 
Moreover, by the previous paper \cite{morphismstability}, the answer is positive when $\ho{\ms C}$ is the derived category $\mb D^{b}(\bb P^{1})$ of the projective line $\bb P^{1}$ over a field $\mb k$. 
The aim of this article is giving a positive answer to Problem \ref{prob:kihon}. 

\begin{thm}[=Theorem \ref{thm:pathconnection}]\label{thm:main}
Let $\ms C$ be a stable infinity category with $\rank K_{0}(\ho{\ms C}) < \infty$. 
If a stability condition $\sigma \in \Stab{\ho{\ms C}} $ is rational, then 
$d_{0}^{*}\sigma$ is path connected to $d_{1}^{*}\sigma$. 
In particular, 
the restriction  of $d_{0}^{*}$ and $d_{1}^{*}$ to the space $\Stabf{\ho{\ms C}}$ of full stability conditions on $\ho{\ms C}$ gives the same map at the level of connected components $\pi_{0}(\Stabf{\ho{\ms C}}) \to \pi_{0}(\Stabf{\ho{\ms C^{\Delta^{1}}}})$. 
\end{thm}

To prove Theorem \ref{thm:main}, it is necessary to study an explicit deformation of gluing stability conditions on $\mb D=\sod{\mb D_{1}}{\mb D_{2}}$.  
Similarly to the case of K3 surfaces (cf. \cite{MR2376815}), we construct a complex one dimensional family of stability conditions which will be denoted by 
$\mca S (\epsilon_{1}, \epsilon_{2})= \{	\til{\Sigma}{\beta}{\omega} \mid  (\beta, \omega) \in \mca H^{+}(\epsilon_{1}) \cap \mca H^{-}(\epsilon_{2}) \}$ 
where 
$\sigma_{i} \in \Stab{\mb D_{i}}$ with nice gluing property (cf. Definition \ref{dfn:gluingproperty}).

The construction is divided into mainly four steps: 
\renewcommand{\theenumi}{Step \arabic{enumi}}
\begin{enumerate}
\item An observation of a common property of two semiorthogonal decomposition $\ho{\ms C^{\Delta^{1}}}=\sod{\ho{\ms C}}{\ho{\ms C_{/0}}}$ and $\ho{\ms C^{\Delta^{1}}}=\sod{\ho{\ms C_{0/}}}{\ho{\ms C}}$. 
\item An introduction of a stability on the gluing heart which is an analogy of the slope stability of vector bundles on K3 surfaces (or more generally smooth projective variety $X$ with $\dim X>1$). 
\item Construction of a continuous family of stability conditions on $\mb D=\sod{\mb D_{1}}{\mb D_{2}}$. 
\item Specialization of the semiorthogonal decompositions. 
\end{enumerate}
\renewcommand{\theenumi}{\arabic{enumi}}

We first introduce some conditions for semiorthogonal decompositions (cf. Definition \ref{dfn:gluingproperty}). 
These conditions are motivated by two semiorthogonal decompositions of $\ho{\ms C^{\Delta^{1}}}$

Recall that the tilting construction on K3 surfaces is based on the slope stability of vector bundles (cf. \cite{MR2376815}). 
Similarly to the case of K3 surfaces, we introduce a stability on the heart of a gluing stability condition $\gl{\sigma_{1}}{\sigma_{2}}$ on a triangulated category $\mb D$ equipped with the semiorthogonal decomposition $\sod{\mb D_{1}}{\mb D_{2}}$ discussed in Step 1 above. 
Then such a slope like stability satisfies the support property which is an analogy of the support property of the slope stability. 

Then, we introduce a torsion pair on the heart of a gluing stability condition. 
It is not hard to construct $\til{\Sigma}{\beta}{\omega}$ for rational numbers $(\beta, \omega)\in \bb Q \times \bb Q_{>0}$. 
The hardest part is to show the continuousness of the family. 
As mentioned in \cite[Appendix]{MR3573975}, if the stability condition $\til{\Sigma}{\beta}{\omega}$ satisfies the support property one could deform locally and show the continuousness of the family. 
However the support property is very hard in our situation since the category $\ho{\ms C^{\Delta^{1}}}$ is not geometric but too general. 
In stead of it, we show a weaker property suggested in the original paper due to Bridgeland \cite{MR2373143} so that the stability condition can be deformed locally.  
After that we extend our construction to non rational numbers $(\beta, \omega) \in \bb R \times \bb R_{>0}$ by taking limits of rational ones.

Though we could not prove the support property of any stability condition $\til{\Sigma}{\beta}{\omega} \in \mca S(\epsilon_{1}, \epsilon_{2})$ directly, we will prove the support property via ``specialization''. 
More precisely, at the boundary of the family $\mca S(\epsilon_{1}, \epsilon_{2})$,  there is a stability condition $\til{\Sigma}{1}{0}$ satisfying the support property if $\sigma_{1}$ and $\sigma_{2}$ satisfy the support property (see also \S \ref{sc:support}). 
Then the support property holds for any one in $\mca S(\epsilon_{1}, \epsilon_{2})$ since the family is continuous.

Finally applying the semiorthogonal decomposition $\mb D=\sod{\mb D_{1}}{\mb D_{2}}$ to $\ho{\ms C^{\Delta^{1}}}=\sod{\ho{\ms C}}{\ho{\ms C_{/0}}}$ and $\ho{\ms C^{\Delta^{1}}}=\sod{\ho{\ms C_{0/}}}{\ho{\ms C}}$, we obtain a path which connect $d_{0}^{*} \sigma$ and $d_{1}^{*}\sigma$. 
At the end of introduction, we would like to answer a question that readers may have. 
Problem \ref{prob:kihon} might be very simple for readers. 
However we hope that the construction of the answer, Theorem \ref{thm:pathconnection}, could be interesting and 
that our construction will provide a new insight into deformation of gluing stability conditions in a future study.

\section{Preliminaries}

Let $\mb D$ be a triangulated category. 
We always assume that the rank of the Grothendieck group $K_{0}(\mb D)$ is finite. 
The $t$-structure on $\mb D$ is denoted by $(\mb D^{\leq 0}, \mb D^{\geq 0})$. 
Recall that the subcategory $\mb D^{\leq 0} \cap \mb D^{\geq 0}$ is the heart $\mca A$ of the $t$-structure. 
As usual, let us denote $\mb D^{\leq 0}[p]$ by $\mb D^{\leq -p}$ and denote $\mb D^{\geq 0}[p] $ by $\mb D^{\geq -p}$ 
for any integer $p$. 

Through this article we frequently identify the complex plane $\bb C$ with the $2$-dimensional Euclidian space $\bb R^{2}$ via the canonical isomorphism. 
The upper half-plane is denoted by $\mca H$: 
\begin{equation}
\mca H=	\{	(\beta, \omega) 	\mid	\beta \in \bb R, \omega \in \bb R_{>0}	\}. 
\end{equation}
The closure of $\mca H$ in $\bb C$ is denoted by $\overline{\mca H}$. 
For an arbitrary subset $U$ of $\mca H$, the set of rational points in $U$ is denoted by $U_{\bb Q}$: 
\[
U_{\bb Q}	=	\{	(\beta, \omega ) \in U	\mid \beta \in \bb Q, \omega \in \bb Q	\}. 
\]

Given a complex number $z \in \bb C$, the real part of $z$ is denoted by $\R z$ and the imaginary part of $z$ is denote by $\I z$. 
We define the argument $\arg z$ of a non-zero complex number $z$ so that \begin{equation}\arg z \in (-\pi, \pi]. \end{equation} 
The following proposition is very elemental but is a key ingredient to prove the desired deformation property in Section \ref{sc:deformation}. 

\begin{prop}\label{bunkai-improve}
Let $\theta$ be in the open interval $(0, \pi)$. 
Suppose complex numbers $z_{1}$ and $z_{2} \in \bb C$ satisfy $0 \leq \arg z_{1}-\arg z_{2} \leq \theta $. 
The following holds: 
\begin{equation}
\label{eq:20}
|z_{1}+z_{2}|	\geq \frac{\sqrt{2+2\cos \theta}}{2}(|z_{1}|+|z_{2}|). 
\end{equation}
\end{prop}

\begin{proof}
Put $\phi = \arg z_{1}-\arg z_{2}$. 
The following inequality holds for any $x \in \bb R$: 
\begin{equation}\label{eq:improve}
\sqrt{1+ x^{2}+2x \cos \phi } \geq \frac{\sqrt{2+2\cos \phi}}{2}(x+1). 
\end{equation}
Indeed, a very elemental calculation below gives the proof: 
\begin{align*}
{x^{2}+2x \cos \phi + 1} - \left( \frac{\sqrt{2+2\cos \phi}}{2}(x+1)	\right)^{2}
	&= \frac{1-\cos \phi}{2}(x-1)^{2} \geq 0 . 
\end{align*}
On the other hand, we easily see $|z_{1}+z_{2}|	=\sqrt{|z_{1}|^{2}	+	|z_{2}|^{2}	+2|z_{1}|\cdot |z_{2}| \cdot \cos \phi}$. 
Then the inequality (\ref{eq:improve}) implies 
\begin{align*}
|z_{1}+z_{2}|		&\geq |z_{1}| \cdot \frac{\sqrt{2+2\cos \phi}}{2}\left(	\frac{|z_{2}|}{|z_{1}|}+1	\right)	
					=\frac{\sqrt{2+2\cos \phi}}{2}(|z_{2}|+|z_{1}|)
\end{align*}
Since $f(\phi)=\cos \phi$ is monotonically decreasing for $0 \leq \phi <\pi$, 
this gives the proof. 
\end{proof}

\begin{cor}\label{bunkai-sup}
Let $\theta$ be in the open interval $(0, \pi)$. 
Then the supremum 
\[
\sup
\left\{
\frac{|z_{2}|}{|z_{1}+z_{2}|}
\middle|
z_{1}, z_{2} \in \bb C, 
z_{1}+z_{2} \neq 0, 0 \leq| \arg z_{1}-\arg z_{2}| \leq \theta
\right\}
\]
is finite. 
\end{cor}

\begin{proof}
By Proposition \ref{bunkai-improve}, 
the following holds: 
\begin{align*}
\frac{|z_{2}|}{|z_{1}+z_{2}|}	&	\leq 
\frac{2}{\sqrt{2+2\cos \theta}}
\cdot 
\left(
\frac{|z_{2}|}{|z_{1}|+|z_{2}|}
\right)
\leq 
\frac{2}{\sqrt{2+2\cos \theta}}. 
\end{align*}
The right hand side is bounded above by the assumption for $\theta$. 
\end{proof}

\subsection{Stability conditions}
This section is devoted to a review of basic notation and conventions for stability conditions. 
Let $\mb D$ be a triangulated category. 

\begin{dfn}[\cite{MR2373143}]
\label{dfn:stabilitycondition}
A \textit{stability condition} $\sigma = (Z, \mca P)$ on $\mb D$ is a collection of a group homomorphism $Z \colon K_{0}(\mb D)\to \bb C$ called the central charge 
and the collection of full subcategories $\mca P(\phi)$ of $\mb D $ for $\phi \in \bb R$ satisfying the following condition: 
\begin{enumerate}
\item for any $E \in \mca P(\phi)$, then $Z(E) \in \bb R_{>0}\cdot \exp(\sqrt{-1}\pi \phi)$, 
\item $E$ is  in $\mca P(\phi)$ if and only if $E[1]$ is in $\mca P(\phi+1)$,  
\item if $\phi > \psi$ then $\Hom_{\mb D}(E, F)=0$ for $E \in \mca P(\phi)$ and $F \in \mca P(\psi)$, 
\item any object $E \in \mb D$ has a finite sequence of distinguished triangles 
\begin{equation}\label{eq:hnf}
\HNF{E}{E}{A}, 
\end{equation}
where $A_{i} \in \mca P(\phi)\setminus\{0 \}$ with $\phi_{i}>\phi_{i+1}$. 
\end{enumerate}

A non-zero object $E \in \mca P(\phi)$ for some $\phi \in \bb R$ is said to be \textit{$\sigma$-semistable} with the phase $\phi$. 
The collection of full subcategories $\bigcup _{\phi \in \bb R}\mca P(\phi)  $ satisfying the axiom (2), (3) and (4) is called the slicing of $\mb D$. 
\end{dfn}

\begin{rmk}
By the definition, the zero object $0 \in \mb D$ is not $\sigma$-semistable. 
We note that $\mca P(\phi)$ is abelian. 
A $\sigma$-semistable object $A \in \mca P(\phi)$ is said to be $\sigma$-stable if there is no non-trivial subobject of $A$ in $\mca P(\phi)$. 

For any $E \in \mb D$, the filtration given by (\ref{eq:hnf}) is called a \textit{Harder-Narasimhan filtration} of $E$. 
The filtration is unique up to isomorphism. 
The first term $E_{1}$ of the filtration is said to be the \textit{maximal destabilizing subobject} of $E$, and the mapping cone $E_{n}/E_{n-1}$ of $E_{n-1} \to E_{n}$ is called the \textit{maximal destabilizing quotient} of $E$. 
Moreover the phase of the maximal destabilizing subobject (resp. quotient) of $E$ is denoted by $\phi_{\sigma}^{+}(E)$ (resp. $\phi_{\sigma}^{-}(E)$). 
\end{rmk}

Now we recall convention which is necessary in this article.

\begin{dfn}
Let $\sigma$ be a stability condition on $\mb D$. 
\begin{enumerate}
\item A stability condition $\sigma$ is said to be \textit{discrete} if the image of $Z \colon K_{0}(\mb D) \to \bb C$ is discrete in $\bb C$. 
A stability condition $\sigma$ is said to be \textit{rational} if the image of $Z \colon K_{0}(\mb D) \to \bb C$ is contained in rational coefficients complex numbers.
\item For an interval $I \subset \bb R$, the extension closure of $\{ \mca P(\phi) \mid \phi \in I \}$ is denoted by $\mca P(I)$. 
\item A stability condition $\sigma=(Z, \mca P) $ is said to be \textit{locally finite} if for any $\phi$ there exists an $\epsilon >0$ such that the quasi-abelian category $\mca P(\phi-\epsilon, \phi + \epsilon)$ is finite length. 
\item For any $E \in \mb D$, 
the \textit{mass} of $E$ which will be denoted by $m_{\sigma}(E)$ is the sum $\sum_{A_{i}} |Z(A_{i})|$ where $A_{i}$ runs semistable factors of the Harder-Narasimhan filtration of $E$. 
\item 
Let $\tau=(W, \mca Q)$ be a stability condition on $\mb D$. 
Set 
\[
\met{\mca P}{\mca Q} =\sup \{	|\phi_{\sigma}^{-}(E) -\phi_{\tau}^{-}(E)|,  |\phi_{\sigma}^{+}(E) -\phi_{\tau}^{+}(E)|	\mid E \in \mb D \setminus \{0\}	\}, 
\]
 and 
a norm on $\Hom(K_{0}(\mb D), \bb C)$ by 
\[
\|	U	\|_{\sigma}	=\sup\left\{	\frac{|U(E)|}{|Z(E)|}	\middle|	 E \in \bigcup \mca P(\phi), E\neq 0		\right\}. 
\]
We note that $\|	U	\|_{\sigma}	$ might be infinity. 
\item The stability condition $\sigma$ is said to be \textit{full} if the norm $\| U \| _{\sigma}$ is finite for any $U \in \Hom(K_{0}(\mb D), \bb C)$. 
\item The stability condition $\sigma $ is said to be \textit{reasonable} if the infimum 
\[
\inf	\left\{	|Z(E)|		\middle|		E \in	\bigcup_{\phi \in \bb R}	\mca P(\phi)	\right\}
\]
is positive. 
\item The set of locally finite stability condition on $\mb D$ is denoted by $\Stab{\mb D}$. 
In addition, the set of full (resp. reasonable) stability conditions is denoted by $\Stabf{\mb D}$ (resp. $\Stabr{\mb D}$). 
\item $\sigma$ satisfies the \textit{support property} if there exists a quadratic form $q$ on $K_{0}(\mb D) \otimes \bb R$ such that 
\begin{itemize}
\item any $\sigma$-semistable object $E \in \mb D$, $q([E]) \geq 0$, and 
\item $q$ is negative definite on $\Ker Z \subset K_{0}(\mb D) \otimes \bb R$. 
\end{itemize}
\end{enumerate}
\end{dfn}

The support property is important since it is necessary to discuss the continuousness of the family of stability conditions, or well-behaved wall crossing phenomena. 
The following gives some equivalent formulation of the support property. 

\begin{prop}[\cite{MR2852118}, \cite{kontsevich2008stability}, \cite{MR3573975}]
\label{prop:supportproperty}

Let $\sigma$ be a locally finite stability condition on $\mb D$. 
The following are equivalent: 
\begin{enumerate}
\item $\sigma$ is full. 
\item $\sigma$ satisfies the support property. 
\item Fix a norm on $K_{0} (\mb D) \otimes \bb R$. 
There exists a positive constant $C >0$ such that the inequality 
\[	\frac{\|	[E]\|}{|Z(E)|} \leq C	\]
holds for any $\sigma$-semistable object $E \in \mb D$. 
\end{enumerate}
\end{prop}

By Proposition \ref{prop:supportproperty}, one can deduce 
\begin{prop}
Let $\sigma=(Z, \mca P)$ be a stability condition on a triangulated category $\mb D$. 
If $\sigma$ is full, then $\sigma$ is reasonable. 
In addition, a reasonable stability condition is locally finite. 
\end{prop}

\begin{proof}
Let us fix a norm $\| * \| \colon K_{0}(\mb D) \otimes \bb R \to \bb R_{>0}$. 
If $\sigma$ is full, by Proposition \ref{prop:supportproperty}, 
there exists a constant $C' > 0 $ such that the inequality 
\begin{equation}
C' \|	E	\|	\leq |Z(E)| 
\end{equation}
holds for any $\sigma$-semistable object $E$. 
Since $\|	E\|$ is discrete, $\sigma$ is reasonable. 
The last part follows from \cite[Lemma 1.1]{MR2721656}. 
\end{proof}

\begin{rmk}\label{rmk:openclosed}
Thus we have a sequence of inclusion
\[
\Stabf{\mb D} \subset \Stabr{\mb D}	\subset \Stab {\mb D}. 
\]
One can also check that the support property is open and closed by the argument in \cite[Lemmma 6.2]{MR2373143}. 
Hence $\Stabf{\mb D}$ is open and closed in $\Stab{\mb D}$. 
This implies that any $\sigma \in \mathsf{C}_{0}$ is full if there exists a full stability condition $\sigma_{0}$ in a connected component $\mathsf C_{0}$ of $\Stab{\mb D}$. 
The same assertion for reasonable stability conditions holds by \cite[Proposition 1.2]{MR2721656}. 
\end{rmk}

\begin{rmk}
If a locally finite stability condition $\sigma$ satisfies the support property, then $\sigma$ is said to be a \textit{Bridgeland stability condition}. 
The set of Bridgeland stability condition on a triangulated category $\mb D$ is the same as $\Stabf {\mb D}$ by Proposition \ref{prop:supportproperty}. 
Recently the set of Bridgeland stability conditions is often denoted by $\Stab{\mb D}$ instead of $\Stabf{\mb D}$, but in this paper we denote it by $\Stabf{\mb D}$. 
\end{rmk}

The following is a summary of \cite{MR2373143}. 

\begin{thm}[\cite{MR2373143}]	\label{thm:Bridgeland}
Let $\mb D$ be a triangulated category and $\sigma=(Z, \mca P)$ a locally finite stability condition. 
\begin{enumerate}
\item Set a subset $B_{\epsilon}(\sigma)$ of $\Stab {\mb D}$ by 
\[
B_{\epsilon}(\sigma)=\{	 \tau=(W, \mca Q) \in \Stab{\mb D}  \mid 	\met{\mca P}{\mca Q} < \epsilon, \|	W-Z	\|_{\sigma}< \sin (\pi \epsilon)	\}. 
\]
Then there exists a topology on $\Stab {\mb D}$ such that $ \{	B_{\epsilon}(\sigma)  \mid \sigma \in \Stab{\mb D}, \epsilon \in (0,1/4)	\} $ forms a basis of the topology. 
\item 
For an $\epsilon \in (0, 1/8)$, if $W \colon K_{0}(\mb D) \to \bb C$ satisfies 
\begin{equation}\label{eq:smalldeformation}
\|	W-Z	\|_{\sigma} <\sin (\pi \epsilon)
\end{equation}
then there exists a unique stability condition $\tau $ whose central charge is $W$ and slicing $\mca Q$ satisfies $\met {\mca Q}{\mca P} < \epsilon$. 
\item Each connected components of $\Stab{\mb D}$ is a complex manifold. 
The tangent space $T_{\sigma}$ at $\sigma \in \Stab{\mb D}$ is isomorphic to the linear subspace 
\[
V_{\sigma}	=\{	U \in \Hom(K_{0}(\mb D), \bb C) \mid \| U \|_{\sigma} <\infty	\}. 
\]
\item $\Stab{\mb D}$ has the right action of the universal cover $\widetilde{\mr{GL}}_{2}^{+}(\bb R)$ of $\mr{GL}_{2}^{+}(\bb R)$ and 
the left action of the autoequivalence group $\Aut(\mb D)$ of $\mb D$. 
\end{enumerate}
\end{thm}

\begin{dfn}\label{dfn:deformation}
Let $\sigma=(Z, \mca P)$ be a locally finite stability condition on $\mb D$.  
A stability condition $\tau = (W, \mca Q)$ on $\mb D$ is said to be a \textit{deformation of $\sigma$} if 
$\tau$ is in $B_{\epsilon}(\sigma)$
for an $\epsilon \in (0, 1/8)$. 
\end{dfn}

\begin{rmk}
For a locally finite stability condition $\sigma=(Z, \mca P)$ and a group homomorphism $W \colon K_{0}(\mb D) \to \bb C$, 
if $\|	W-Z	 \|_{\sigma}$ is sufficiently small then there exists a unique deformation of $\sigma$ by Theorem \ref{thm:Bridgeland}. 
\end{rmk}

Finally we recall another definition of stability conditions which is helpful to construct $\sigma$ in an explicit way.

\begin{dfn}
Let $\mca A$ be the heart of a bounded $t$-structure $(\mb D^{\leq 0}, \mb D^{\geq 0})$ on $\mb D$ and 
let $Z$ be a group homomorphism $Z \colon K_{0}(\mca A) \to \bb C$. 
\begin{enumerate}
\item $Z$ is said to be a central charge on $\mca A$ if $Z$ satisfies $0 < \arg Z(E) \leq \pi$ for any non-zero object $E \in \mca A$.  
\item Given a central charge $Z$, an object $E$ is said to be $Z$-semistable if the inequality 
$\arg Z(F) \leq \arg Z(E)$ holds for any non-trivial subobject $F$ of $E$. 
\item Let $Z$ be a central charge on $\mca A$. 
The pair $(\mca A, Z)$ has the Harder-Narashimhan property if any non-zero object $E $ has a finite filtration of subobjects 
\[
0=E_{0}\subset E_{1} \subset E_{2} \subset \cdots \subset E_{n-1}\subset E_{n}=E
\]
such that 
\begin{itemize}
\item for each $i \in \{1, 2, \cdots, n\}$, $A_{i}=E_{i}/E_{i-1}$ is $Z$-semistable, and 
\item $\arg Z(A_{i})$ is decreasing. 
\end{itemize}
In addition, the pair $(\mca A, Z)$ with the Harder-Narasimhan property is called a \textit{stability condition on the heart $\mca A$}. 
\end{enumerate}
\end{dfn}

\begin{prop}[{\cite[Proposition 5.3]{MR2373143}}]
Let $ \mb D$ be a triangulated category. Then the following are equivalent:
\begin{enumerate}
\item To give a stability condition $\sigma =(Z,\mca P)$ on $ \mb D$.
\item To give a pair $(\mca A,Z)$ consisting of  the heart $\mca A$ of a bounded t-structure on $ \mb D$ and a central charge $Z$ on $\mca A$ which has the Harder-Narashimhan property. 
\end{enumerate}
\end{prop}

\begin{rmk}
Let us briefly recall that 
the heart $\mca A$ is given by $\mca P(0,1]$. 
Throughout the article, we call $\mca A=\mca P(0,1]$ the \textit{heart of a stability condition} $\sigma$ and fluently identify the pair $(Z, \mca P)$ with $(\mca A, Z)$. 
Namely we denote by $\mca P$ the slicing of the stability condition $(\mca A, Z)$ and by $\mca A$ the heart of the stability condition $(Z, \mca P)$. 
\end{rmk}

Given a stability condition $\sigma=(\mca A, Z)$ on $\mb D$, 
we obtain a torsion pair on $\mca A$ which is useful in this article: 

\begin{dfn}\label{dfn:sigma-torsion}
Let $\mb D$ be a triangulated category and $\sigma= (\mca A, Z) $ a stability condition on $\mb D$. 
An object $E \in \mca A$ is said to be \textit{$\sigma$-torsion} if $\I Z(E)=0$. 
An object $E' \in \mca A$ is said to be \textit{$\sigma$-free} if $E \in \mca P(0,1)$

The full subcategory of $\sigma$-torsion objects in $\mca A$ is denoted by $\mca T^{\sigma}$ and 
the full subcategory of $\sigma$-torsion free objects in $\mca A$ is denoted by $\mca F^{\sigma}$. 
\end{dfn}

\begin{lem}
Keep the notation as in Definition \ref{dfn:sigma-torsion}. 
Then the pair $(\mca T^{\sigma}, \mca F^{\sigma})$ is a torsion pair on $\mca A$. 
\end{lem}

\begin{proof}
The definition of stability conditions directly implies the assertion. 
\end{proof}

\subsection{Semiorthogonal decompositions and stability conditions}

Let $\mb D= \<\mb D_{1}, \mb D_{2}	\>$ be a semiorthogonal decomposition of $\mb D$. 
Namely we have pairs of adjoint functors
\[
\xymatrix{
\mb D_{2}	\ar@<0.8ex>[r]^{i_{2}}	&	\mb D	\ar@<0.8ex>[r]^{\tau_{1}^{L}}	\ar@<0.8ex>[l]^(.4){\tau_{2}^{R}} &	\mb D_{1}	\ar@<0.8ex>[l]^{i_{1}}
}; 
i_{2}\dashv	\tau_{2}^{R}, \tau_{1}^{L}	\dashv	i_{1} , 
\]
both $i_{2}$ and $i_{1}$ are fully faithful and $\tau_{1}^{L} \circ i_{2} = \tau_{2}^{R} \circ i_{1}=0$. 
If $i_{1} $ has the right adjoint $\tau_{1}^{R} \colon \mb D \to \mb D_{1}$, then the composite $\tau_{1}^{R} \circ i_{2} [1] \colon \mb D_{2} \to \mb D_{1}$ is said to be the \textit{gluing functor} of the semiorthogonal decomposition $\mb D= \<\mb D_{1}, \mb D_{2}	\>$ which will be denoted by $\Phi \colon \mb D_{2} \to \mb D_{1}$. 

We introduce some important conditions derived from the category of morphisms. 

\begin{dfn}\label{dfn:gluingproperty}
Let $\mb D= \<\mb D_{1}, \mb D_{2}	\>$ be a semiorthogonal decomposition of $\mb D$, and let $(\mb D_{i}^{\leq 0}, \mb D_{i}^{\geq 0})$ be the bounded $t$-structure on $\mb D_{i}$. 
A torsion pair on the heart $\mb D_{i}^{\leq 0}\cap \mb D_{i}^{\geq 0}$ is denoted by $(\mca T_{i}, \mca F_{i})$. 
\begin{description}
\item[$\m1$] $i_{1} \colon \mb D_{1} \to \mb D$ has the right adjoint $\tau_{1}^{R} \colon \mb D \to \mb D_{1}$. 
\item[$\m2$] In addition to $\m1$, there exist equivalences $\mb D_{i} \to \mb C$ such that the right adjoint $\tau_{1}^{R} \colon \mb D \to \mb D_{1}$  
gives the left adjoint $\tau_{2}^{L}$ of $i_{2} \colon \mb D_{2} \to \mb D$ via the equivalences $\mb D_{k} \to \mb C$. 
\item[$\m3$]
In addition to $\m1$, the gluing functor $\Phi =\tau_{1}^{R}\circ i_{2}[1] \colon \mb D_{2} \to \mb D_{1}$ is $t$-exact with respect to these $t$-structures $(\mb D_{i}^{\leq 0}, \mb D_{i}^{\geq 0})$. 
\item[$\m4$] In addition to $\m3$, the gluing functor $\Phi$ respects to free part of the torsion pairs, that is, $\Phi (\mca F_{2}) \subset \mca F_{1}$. 
\end{description}

Moreover, let $\sigma_{i}=(\mca A_{i}, Z_{i})$ be in $\Stab{\mb D_{i}}$. 
If the gluing functor is $t$-exact with respect to the $t$-structures of $\sigma _{i}$, 
$(\sigma_{1}, \sigma_{2})$ is said to \textit{satisfy the condition $\m3$}. 
Similarly, $(\sigma_{1}, \sigma_{2})$ is said to \textit{satisfy the condition $\m4$}, if 
the torsion pair $(\mca T^{\sigma_{i}}, \mca F^{\sigma_{i}})$ on the heart $\mca A_{i}$ satisfies the condition $\m4$. 
Finally let us introduce the following condition for $\sigma_{i}$: 
\begin{description}
\item[$\m5$] In addition to $\m4$ for $\sigma_{i}$, the equality $Z_{2}(E_{2})=Z_{1}(\Phi E_{2})$ holds for any $E_{2} \in \mca A_{2}$. 
\end{description}
\end{dfn}

\begin{rmk}
Suppose a semiorthogonal decomposition $\mb D= \sod{\mb D_{1}}{\mb D_{2}}$ satisfies the condition $\m2$. 
Then the gluing functor $\Phi$ is isomorphic to $[1]$ on $\mb C$. 
If $t$-structures $(\mb D_{i}^{\leq 0}, \mb D_{i}^{\geq 0})$ satisfies $\m3$, then the heart $\mca A_{2}$ on $\mb D_{2} \cong \mb C$ is $\mca A_{1}[-1]$ on $\mb D_{1}\cong \mb C$. 
\end{rmk}

\begin{lem}\label{lem:morexact}
Suppose that a semiorthogonal decomposition $\mb D =\sod{\mb D_{1}}{\mb D_{2}}$ satisfies $\m1$. 
The gluing functor $\mb D_{2} \to \mb D_{1}$ is denoted by $\Phi$. 
There exists a distinguished triangle in $\mb D_{1}$
\[
\xymatrix{
\Phi(\tau_{2}^{R}F)[-1]	\ar[r]	&	\tau_{1}^{R}F	\ar[r]	&	\tau_{1}^{L}F	\ar[r]	&	\Phi(\tau_{2}^{R}F)
}
\]
for any $F \in \mb D$. 
\end{lem}

\begin{proof}
Consider the semiorthogonal decomposition of $F \in \mb D$: 
\begin{equation}\label{eq:sodmor}
\xymatrix{
i_{2}\tau_{2}^{R}(F)		\ar[r]	&	F	\ar[r]	&	i_{1}\tau_{1}^{L}(F)	\ar[r]	&	i_{2}\tau_{2}^{R}(F)[1]. 
}
\end{equation}
Taking $\tau_{1}^{R}$ to the sequence (\ref{eq:sodmor}), 
we obtain the desired sequence. 
\end{proof}

\begin{lem}
\label{lem:morvanishing}
Suppose that a semiorthogonal decomposition $\mb D =\sod{\mb D_{1}}{\mb D_{2}}$ satisfies $\m1$. 
The gluing functor is denoted by $\Phi \colon \mb D_{2} \to \mb D_{1}$. 
Let $\varphi \colon E \to F$ be a morphism in $\mb D$ with $\tau_{1}^{L}\varphi=0$ and $\tau_{2}^{R}\varphi=0$. 
If $\Hom_{\mb D_{1}}(\tau_{1}^{L}E, \Phi(\tau_{2}^{R}F)[-1])=0$, the morphism $f$ is zero. 
\end{lem}

\begin{proof}
Taking the semiorthogonal decomposition of $E$ and $F \in \mb D$, 
we obtain the following diagram of exact sequences: 
\begin{equation}\label{eq:morgluing}
\xymatrix{
\Hom_{\mb D}(i_{1}\tau_{1}^{L}E, i_{2}\tau_{2}^{R}F)	\ar[d]	\\
\Hom_{\mb D}(E, i_{2}\tau_{2}^{R}F)	\ar[r]^-{\bar s}\ar[d]_-{\bar t_{2}}	&	\Hom_{\mb D}(E, F)	\ar[r]^-{t_{1}}\ar[d]_-{t_{2}}	&	\Hom_{\mb D}(E, i_{1}\tau_{1}^{L}F)\ar[d]	\\
\Hom_{\mb D}(i_{2}\tau_{2}^{R}E, i_{2}\tau_{2}^{R}F)	\ar[r]^-{s}	&	\Hom_{\mb D}(i_{2}\tau_{2}^{R}E, F)		\ar[r]	 &	\Hom(i_{2}\tau_{2}^{R}E, i_{1}\tau_{1}^{L}F). 
}
\end{equation}
Note that $\Hom_{\mb D}(i_{1}\tau_{1}^{L}E, i_{2}\tau_{2}^{R}F)	 \cong 
\Hom_{\mb D_{1}}(\tau_{1}^{L}E, \Phi(\tau_{2}^{R}F)[-1])$. 
Hence the morphism $\bar t_{2}$ is injective. 
The assumption $\tau_{1}^{L}\varphi =\tau_{2}^{R}\varphi=0$ imply $t_{1}(\varphi)=t_{2}(\varphi)=0$, and we see 
that there exists a $\psi \in \Hom_{\mb D}(E, i_{2}\tau_{2}^{R}F)$ such that $\bar s(\psi )=\varphi$. 
By the semiorthogonality, the morphism $s$ above is an isomorphism. 
Since the composite $s \circ \bar t_{2}$ is injective, $\psi $ has to be in the kernel $\ker \bar t_{2}$ which is trivial. 
Thus $\psi $ is zero and so is $\bar s(\psi)=\varphi$. 
\end{proof}

\begin{cor}\label{cor:keyvanishing}
Let $\ms C$ be a stable infinity category, and put 
$\mb D=\ho{\ms C^{\Delta^{1}}}$. 
For objects $[f \colon x \to y]$ and $[g \colon z \to w]$ in $\ho{\ms C^{\Delta^{1}}}$, 
suppose a morphism $\tau \colon f \to g$ in $\ho{\ms C^{\Delta^{1}}}$ satisfies $d_{1}\tau=d_{0}\tau=0$. 
If $\Hom_{\ho{\ms C}}(x,w[-1])=0$, then $\tau$ is a zero morphism. 
\end{cor}

\begin{proof}
Let $j_{!} \colon \ms C \to \ms C^{\Delta^{1}}$ be the right adjoint of $d_{1} \colon \ms C^{\Delta^{1}} \to \ms C$ and 
let $j_{*} \colon \ms C \to \ms C^{\Delta^{1}}$ be the left adjoint of $d_{0} \colon \ms C^{\Delta^{1}} \to \ms C$. 
Since the right adjoint $\tau_{1}^{R}$ of $j_{!}$ is taking the fiber of morphisms in $\ms C$, 
we have the diagram of the semiorthogonal decomposition with $\m1$:
\[
\xymatrix{
\ho{\ms C}	\ar@<0.8ex>[r]^-{j_{*}}	&	\ho{\ms C^{\Delta^{1}}}	\ar@<0.8ex>[r]^-{\tau_{1}^{L}=d_{1}}	\ar@<0.8ex>[l]^-{d_{0}}	\ar@<-2.4ex>[r]_-{\tau_{1}^{R}} &	\ho{\ms C}.	\ar@<0.8ex>[l]|-{j_{!}}
}
\]
Now clearly we see 
\begin{align*}
\Hom_{\ho{\ms C^{\Delta^{1}}}}(\tau_{1}^{L}f, \Phi(\tau_{2}^{R}g))&= \Hom_{\ho{\ms C}}(x, w[-1]). 
\end{align*}
Then Lemma \ref{lem:morvanishing} implies the desired assertion. 
\end{proof}

The following lemma is a revision of an argument in \cite[Lemma 2.1 and Theorem 3.6]{MR2721656} in terms of the gluing functor. 

\begin{lem}\label{lm:SODt-structure}
Let $\mb D= \<\mb D_{1}, \mb D_{2}	\>$ be a semiorthogonal decomposition with $\m1$ and let $(\mb D_{i}^{\leq 0}, \mb D_{i}^{\geq 0}) $ be a bounded $t$-structure on $\mb D_{i}$. 
Suppose that the $t$-structures $(\mb D_{i}^{\leq 0}, \mb D_{i}^{\geq 0})$ satisfies $\m3$. 
\begin{enumerate}
\item Then the pair of following full subcategories determines a $t$-structure on $\mb D$: 
\begin{align*}
\mb D^{\leq 0}	&=	\{	E \in \mb D	\mid \tau_{2}^{R}E \in \mb D_{2}^{\leq 0}\text{ and } \tau_{1}^{L}E \in \mb D_{1}^{\leq 0}\},  \text{ and}\\
\mb D^{\geq 0}	&=	\{	E \in \mb D	\mid \tau_{2}^{R}E \in \mb D_{2}^{\geq 0}\text{ and }\tau_{1}^{L}E \in \mb D_{1}^{\geq 0}\}. 
\end{align*}
\item Put $\mca A= \mb D^{\leq 0} \cap \mb D^{\geq 0}$. 
Suppose a torsion pair $(\mca T_{i}, \mca F_{i})$ on the heart $\mca A_{i}= \mb D_{i}^{\leq 0} \cap \mb D_{i}^{\geq 0}$ satisfies the condition $\m4$. 
Then the pair $(\mca T, \mca F)$ (below) gives a torsion pair on $\mca A$. 
\begin{align*}
\mca T	&= \{	E \in \mca A	\mid \tau_{2}^{R}E \in 	 \mca T_{2} \text { and }\tau_{1}^{L} E \in \mca T_{1}\}, \text{ and} \\
\mca F	&= \{	E \in \mca A	\mid \tau_{2}^{R}E \in 	 \mca F_{2}\text { and } \tau_{1}^{L} E \in \mca F_{1}\}. 
\end{align*}
\end{enumerate}
\end{lem}

\begin{proof}
Put $\mca A_{i}:=\mb D_{i}^{\leq 0} \cap \mb D_{i}^{\geq 0}$. 
To prove the first assertion, by \cite[Lemma 2.1]{MR2721656}, it is enough to show $\Hom_{\mb D}(i_{1}E_{1}, i_{2}E_{2}[p])=0$ for any $E_{i} \in \mca A_{i} $ ($i=1,2$) and $p \leq 0$. 
By the adjunction, we see 
\[
\Hom_{\mb D}(i_{1} E_{1} , i_{2}E_{2}[p])	\cong \Hom_{\mb D_{1}}(E_{1}, \tau_{1}^{R}i_{2}E_{2}[p])
=\Hom_{\mb D_{1}}(E_{1}, \Phi(E_{2})[p-1]). 
\]
Since the gluing functor is $t$-exact, we have $\Phi (E_{2}) \in \mca A_{1}$. 
Thus $\Hom_{\mb D_{1}}(i_{1}E_{1}, i_{2}E_{2}[p]) =0$. 

To prove the second assertion, put $E_{2} $ by $\tau_{2}^{R}E$ and $E_{1}$ by $\tau_{1}^{L}E$. 
let $T_{i}$ (resp. $F_{i}$) be the torsion part of $E_{i}$ (the free part of $E_{i}$). 
Now we claim $\Hom_{\mb D}(i_{1}T_{1}, i_{2}F_{2}[1])=0$. 
In fact, the adjunction and the condition $\m4$ imply the vanishing as follows: 
\[
\Hom_{\mb D}(i_{1}T_{1}, i_{2}F_{2}[1])=
\Hom_{\mb D_{1}}(T_{1}, \tau_{1}^{R}i_{2}F_{2}[1]) =
\Hom_{\mb D_{1}}(T_{1}, \Phi (F_{2})) =0. 
\]
Then the vanishing $\Hom_{\mb D}(i_{1}T_{1}, i_{2}F_{2}[1])=0$ implies a morphism $i_{1}T_{1} \to i_{2}T_{2}[1]$ and $i_{1}F_{1} \to i_{2}F_{2}[1]$ which commute the following diagram: 
\begin{equation*}
\xymatrix{
i_{1}T_{1}	\ar[r]\ar[d]	&	i_{1}E_{1}	\ar[r]\ar[d]	&	i_{1}F_{1}\ar[d]	\\	
i_{2}T_{2}[1]	\ar[r]	&	i_{2}E_{2}[1]	\ar[r]	&	i_{2}F_{2}[1]. 
}
\end{equation*}
The middle vertical arrow represents the object $E$. 
Then the $3\times 3$ lemma in triangulated categories gives the distinguished triangle 
\[
\xymatrix{
T	\ar[r]	&	E	\ar[r]	&	F
}
\]
where $T$ (resp. $F$) is the mapping cone of $i_{1}T_{1} \to i_{2}T_{2}[1]$ (resp. $i_{1}F_{1} \to i_{2}F_{2}[1]$). 
Finally one can show the vanishing $\Hom(T, F)=0$ by the vanishing $\Hom(i_{1}T_{1}, i_{2}F_{2})=0$ and digram chasing. 
\end{proof}

\begin{dfn}
Keep the notation as in Lemma \ref{lm:SODt-structure}. 
We denote by $\gl{\mca A_{1}}{\mca A_{2}}$ the heart of the $t$-structure defined in Lemma \ref{lm:SODt-structure}. 
\end{dfn}

\begin{lem}\label{lem:morsub}
Let $\mb D=\sod{\mb D_{1}}{\mb D_{2}}$ be a semiorthogonal decomposition with $\m2$. 
Suppose that $t$-structures $(\mb D_{i}^{\leq 0}, \mb D_{i}^{\geq 0})$ on $\mb D_{i}$ satisfies $\m3$. 
The heart of the $t$-structure is denoted by $\mca A_{i}$. 

Then a subobject $F$ of $\tau_{1}^{L}E$ of $E \in \gl{\mca A_{1}}{\mca A_{2}}$ in $\mca A_{1}$ determines a subobject $\tilde{F}$ of $E$ in $\gl{\mca A_{1}}{\mca A_{2}}$. 
\end{lem}

\begin{proof}Set $E_{1}=\tau_{1}^{L}E$ and $E_{2}=\tau_{2}^{R}E$. 

Let $g_{E}$ be the morphism $E_{1} \to \Phi (E_{2})$ which represents the extension class $[E] \colon i_{1}E_{1} \to i_{2}E_{2}[1]$ via the adjunction $\Hom_{\mb D}(i_{1}E_{1}, i_{2}E_{2}[1]) \cong \Hom_{\mb D_{1}}(E_{1}, \Phi(E_{2}))$. 
We denote by $f \colon F \to \Phi (E_{2})$ the restriction of $g_{E}$ to $F$. 
Under the equivalence $\mb D_{1} \sim \mb C$, 
Then we have the following commutative diagram in $\mb C$:
\[
\xymatrix{
E_{1}	\ar[r]^-{g_{E}}	&	\Phi (E_{2})	\\
F	\ar[u]^{\iota}\ar[ur]|{f}\ar[r]	& \im f , \ar[u]	
}
\]
where $\im f$ is the image of $f$. 
Note that $\im f \cong \Phi (\im f[-1])$. 
By  the ismorphism 
$\Hom_{\mb C}(F, \im f) \cong \Hom_{\mb D}(i_{1}F, i_{2}(\im f))$
, we obtain the following diagram in $\mb D$: 
\[
\xymatrix{
i_{1}E_{1}	\ar[r]^-{[E]}	&	(i_{2}E_{2})[1]	\\
i_{1}F	\ar[u]^{i_{1}\iota}\ar[r]	& i_{2}(\im f) .\ar[u]	
}
\]
Then define $\tilde F[1]$ by the mapping cone of $i_{1}F \to i_{2}(\im f)$. 
Then there exists a morphism $j \colon F \to E$ such that $\tau_{1}^{L}j=\iota$. 
Since $\tau_{2}^{R}j$ is the inclusion, $F$ gives a subobject of $E$ by the $3 \times 3$ lemma in triangulated categories. 
\end{proof}

\begin{rmk}\label{rmk:morsub}
By the construction of $\tilde F$, $\tau_{2}^{R}\tilde F$ is isomorphic to $(\im f)[-1]$. 
\end{rmk}

\section{Slope on the gluing heart}

The following construction of stability conditions on $\mb D$ was motivated from \cite[Theorem 3.6]{MR2721656} and the previous work \cite{morphismstability}. 
\begin{prop}
Let $\mb D=\sod{\mb D_{1}}{\mb D_{2}}$ be a semiorthogonal decomposition with $\m1$ 
and let $\sigma_{i}=(\mca A_{i}, Z_{i})$ be a reasonable stability condition on $\mb D_{i}$. 
Suppose $(\sigma_{1}, \sigma_{2})$ satisfies $\m5$. 
Define $\gl{Z_{1}}{Z_{2}}  \colon K_{0}(\mb D) \to \bb C$ by   
\begin{align*}
\gl{Z_{1}}{Z_{2}}(E)	:=	Z_{1}(\tau_{1}^{L}E) + Z_{2}(\tau_{2}^{R}E). 
\end{align*}

If $\Phi(E_{2})$ is $\sigma_{1}$-semistable for any $\sigma_{2}$-semistable object $E_{2}  \in \mca A_{2}$,  
then the pair $(\gl{\mca A_{1}}{\mca A_{2}}, \gl{Z_{1}}{Z_{2}})$ which will be denoted by $\gl{\sigma_{1}}{\sigma_{2}}$ 
is a locally finite stability condition on $\mb D$. 
\end{prop}

\begin{proof}
Let $\mca P_{i}$ be the slicing of $\sigma_{i}$ for $i \in \{1,2\}$. 
Following \cite[Theorem 3.6]{MR2721656}, it is enough to show 
\begin{enumerate}
\item[(a)] $\Hom_{\mb D}(i_{1}E_{1}, i_{2}E_{2}[p])=0$ for any $E_{i }\in \mca A_{i}$ and any $p\leq 0$, and 
\item[(b)] there exists an $a$ in the interval $(0,1)$ such that $\Hom_{\mb D}(i_{1}F_{1}, i_{2}F_{2}[p])=0$ for any $F_{i} \in \mca P_{i}(a, a+1]$ and any $p \leq 0$. 
\end{enumerate}

Since $\Phi (\mca A_{2}) \subset \mca A_{1}$ by the condition $\m5$, we see
\begin{align*}
\Hom_{\mb D}(i_{1}E_{1}, i_{2}E_{2}[p]) &\cong \Hom_{\mb D_{1}}(E_{1}, \Phi(E_{2})[p-1])	=	0. 
\end{align*}
This gives the proof of (a). 

Now for any $F_{i} \in \mca P_{i}(a, a+1]$, there exists a canonical triangle in $\mb D_{i}$
\[ 
\xymatrix{
F_{i}^{\leq a}[1]	\ar[r]	&	F_{i}	\ar[r]	&	F_{i}^{>a}	\ar[r]	&	F_{i}^{\leq a}[2], 
}
\]
where $F_{i}^{\leq a } \in \mca P_{i}(0,a]$ and $F_{i}^{>a} \in \mca P_{i}(a, 1]$. 
The assertion (a) implies $\Hom_{\mb D}(i_{1}F_{i}^{\leq a}[1], F_{2}[p])=0$ and 
$\Hom_{\mb D}(i_{1}F_{1}^{>a}, i_{2}F_{2}^{>a}[p])=0$ for $p \leq 0$. 
Hence  it is enough to show 
$\Hom_{\mb D}(i_{1}F_{1}^{>a}, i_{2}F_{2}^{\leq a}[p+1])=0$ for any $p \leq 0$. 

By the assumption, we see $\Phi( \mca P_{2}(\phi)) \subset \mca P_{1}(\phi)$ for any $\phi \in \bb R$. 
Hence we have 
\begin{align*}
\Hom_{\mb D}(i_{1}F_{1}^{>a}, i_{2}F_{2}^{\leq a}[p+1])	 \cong 
\Hom_{\mb D_{1}}(F_{1}^{>a}, \Phi(F_{2}^{\leq a}[p]))=0, 
\end{align*}
which gives the proof. 
\end{proof}

\begin{rmk}
If the semiorthogonal decomposition $\mb D=\sod{\mb D_{1}}{\mb D_{2}}$ satisfies $\m2$, 
then $\Phi$ is an equivalence. 
Hence $\Phi(E_{2})$ is $\sigma_{1}$-semistable for any $\sigma_{2}$-semistable object $E_{2}$. 
\end{rmk}

Next we introduce a slope on the heart $\gl{\mca A_{1}}{\mca A_{2}}$ which is an analogy of the slope stability of coherent sheaves on projective surfaces. 
Similarly to the case of surfaces, a suitable ``freeness'' of objects in $\mb D=\sod{\mb D_{1}}{\mb D_{2}}$ is necessary.

\begin{prop-dfn}
Let $\mb D=\sod{\mb D_{1}}{\mb D_{2}}$ be a semiorthogonal decomposition and 
let $\sigma _{i}=(\mca A_{i}, Z_{i})$ be in $\Stab{\mb D_{i}}$. 
Suppose that $(\sigma_{1}, \sigma_{2})$ satisfies the condition $\m4$. 

Then the pair $(\mca T(\sigma_{1}, \sigma_{2}), \mca F(\sigma_{1}, \sigma _{2}))$ 
\begin{align*}
\mca T(\sigma_{1}, \sigma_{2})	&=	\{	E \in \gl{\mca A_{1}}{\mca A_{2}}	\mid	\tau_{1}^{L}E \in \mca T^{\sigma_{1}}, \tau_{2}^{R}E \in \mca T^{\sigma_{2}}	\}	, \text{ and }\\
\mca F(\sigma_{1}, \sigma_{2})	&=	\{	E \in \gl{\mca A_{1}}{\mca A_{2}}	\mid		\tau_{1}^{L}E \in \mca F^{\sigma_{1}}, \tau_{2}^{R}E \in \mca F^{\sigma_{2}}	\}	. 
\end{align*}
gives a torsion pair on the gluing $t$-structure $\gl{\mca A_{1}}{\mca A_{2}}$. 
We refer to the torsion pair as the gluing torsion pair by $(\sigma_{1},\sigma_{2})$. 

An object $E \in \gl{\mca A_{1}}{\mca A_{2}} $ is said to be $(\sigma_{1}, \sigma_{2})$-torsion (resp. $(\sigma_{1}, \sigma_{2})$-free) if $E$ is in $\mca T(\sigma_{1}, \sigma_{2})$ (resp. $\mca F(\sigma_{1}, \sigma_{2})$)
\end{prop-dfn}

\begin{proof}
The assertion follows from Lemma \ref{lm:SODt-structure}. 
\end{proof}

\begin{dfn}\label{dfn:fakeslope}
Let $\mb D=\< \mb D_{1}, \mb D_{2}\>$ be a semiorthogonal decomposition of $\mb D$ and 
let $\sigma _{i} =(\mca A_{i}, Z_{i}) \in \Stab{\mb D_{i}}$. 

Define $\Mu \colon K_{0}(\gl{\mca A_{1}}{\mca A_{2}}) \to  \bb C$ by 
\begin{equation}\label{eq:slope}
\Mu(E)	=	-\I Z_{1}(\tau_{1}^{L} E) + \omega \R Z_{2}(\tau_{2}^{R}E) -\beta \I Z_{2}(\tau_{2}^{R}E)
+\sqrt{-1}(\I Z_{1}(\tau_{1}^{L}E)	+	\I Z_{2} (\tau_{2}^{R}E)), 
\end{equation}
where $\beta \in \bb R$ and $\omega \in \bb R_{>0}$. 
\end{dfn}

\begin{rmk}
If $\I \Mu (E)=0$, then $\I \Mu (E) $ is non-negative and $\R \Mu (E)$ is non-positive. 
Moreover $\Mu(E)=0$ if and only if $\tau_{2}^{R}E=0$ and $\tau_{1}^{L}E$ is in the torsion part $\mca T^{\sigma_{1}}$ of $\sigma _{1}$. 
\end{rmk}

Now we wish to define a stability condition which is analogous to the slope stability on algebraic surfaces. 
The stability is necessary for the construction of a torsion pair on the heart $\gl{\mca A_{1}}{\mca A_{2}}$.

\begin{prop}
\label{prop:morslope}
Let $\mb D=\< \mb D_{1}, \mb D_{2}\>$ be a semiorthogonal decomposition of $\mb D$ with $\m1$. 
Choose stability conditions $\sigma _{i} =(\mca A_{i}, Z_{i}) \in \Stab{\mb D_{i}}$ such that $(\sigma_{1}, \sigma_{2})$ satisfies the condition $\m4$. 
Define a full sub category $\mca B$ of $\gl {\mca A_{1}}{\mca A_{2}}$ by 
\[
\mca B= \{	E	\in	\gl {\mca A_{1}}{\mca A_{2}}	\mid  \Mu(E)=0	\}\
\]
\begin{enumerate}
\item $\mca B$ is a Serre subcategory of $\gl{\mca A_{1}}{\mca A_{2}}$. 
\item Let $\mca C$ be the Serre quotient category $\gl{\mca A_{1}}{\mca A_{2}}/\mca B$. 
If both $\sigma_{1}$ and $\sigma _{2}$ are discrete, 
then the pair $\tau=(\mca C, \Mu)$ is a stability condition on the abelian category $\mca C$. 
\item In addition to (2), suppose that $E_{1}$ satisfies $\I Z_{1}(E_{1})>0$. Then $i_{1}(E_{1})$ is $\tau$-semistable with $\arg \Mu(i_{1}E_{1})=3\pi/4$. 
\item The stability condition $\tau$ defined in (2) satisfies the support property. 
\end{enumerate}
\end{prop}

\begin{proof}
Since $\Mu$ is a group homomorphism, $\mca B$ is closed under extensions and $\Mu$ satisfies 
\begin{itemize}
\item $\I \Mu (E) \geq 0$, and 
\item If $\I \Mu (E)=0$ then $\R (E) \leq 0$. 
\end{itemize}
Hence $\mca B$ is closed under subobjects and quotients.
This gives the proof of the assertion (1). 

Since $\sigma_{i}$ is discrete, 
the hearts $\mca A_{i}$ is a Noetherian abelian category by \cite[Lemma 3.4]{MR2721656}. 
Thus the heart $\gl{\mca A_{1}}{\mca A_{2}}$ is also Noetherian and so is $\gl{\mca A_{1}}{\mca A_{2}}/\mca B$. 
Since the set $\{\I Z(E) \mid E \in \gl{\mca A_{1}}{\mca A_{2}}\}$ is discrete in $\bb R_{\geq 0}$, 
the pair $(\mca C, \Mu)$ has the Harder-Narasimhan property by \cite[Lemma 3.4]{MR2721656} (or \cite{MR2376815}). 
This gives the proof of the assertion (2).

To prove the assertion (3), let $E_{1} $ be in $\mb D_{1}$. 
If necessary we may assume that $E_{1}$ is $\sigma_{1}$-free, since there exists a $\sigma_{1}$-free object 
$E' \in \gl{\mca A_{1}}{\mca A_{2}}$ such that $E'$ is isomorphic to $i_{1}E_{1}$ in $\mca C$.  
Note that any subobject $F$ of $i_{1}E_{1}$ in $\mca C$ gives a subobject $F' \subset i_{1}E_{1}$ in $\gl{\mca A_{1}}{\mca A_{2}}$ such that $F'$ is isomorphic to $F$ in $\mca C$. 
Since the essential image of $\mca A_{1}$ by $i_{1}$ is closed under subobject, 
we have $\arg \Mu(i_{1}E_{1}) =\arg \Mu (F)$ and this gives the proof of the assertion (3).

To prove the last assertion, it is enough to construct a quadratic form 
\begin{equation}\label{eq:support}
q	\colon K_{0}(\mca C) \otimes \bb R \to \bb R
\end{equation}
satisfying 
\begin{itemize}
\item the restriction of $q$ to the subspace ${\Ker \Mu}$ of $K_{0}(\mca C) \otimes \bb R$ is negative definite, and 
\item $q([E])\geq 0$ if the class $[E]$ is represented by a semistable object $E$. 
\end{itemize}
We show that the quadratic form 
\begin{equation}
q([E])= \I Z_{1}(\tau_{1}^{L}E) \cdot \I Z_{2}( \tau_{2}^{R}E )
\end{equation}
 satisfies the desired property. 
Note that the first condition is trivial since $\Ker \Mu$ is trivial by the construction of $\mca C$.

Take a $\tau$-semistable object $E \in \mca C$. 
Then we have 
\begin{equation*}
\xymatrix{
i_{2}\tau_{2}^{R}E	\ar[r]	&	E	\ar[r]	&	i_{1}\tau_{1}^{L}E	\ar[r]	&	i_{2}\tau_{2}^{R}E[1]. 
}
\end{equation*}
If $\I Z_{1}(i_{1}\tau_{1}^{L}E)=0$, then the inequality $q([E]) \geq 0$ holds. 
If $\I Z_{1}(i_{1}\tau_{1}^{L}E)>0$, then the assertion (3) implies
\begin{equation}\label{eq:keyeq}
0 < \arg \Mu (i_{2}\tau_{2}^{R}E)\leq	\arg \Mu (E)	\leq 	\arg \Mu (i_{1}\tau_{1}^{L}E) = \frac{3\pi}{4}. 
\end{equation}
Now let us canonically identify $\bb C$ as $\bb R^{2}$. 
Then the standard inner paring of $(1+\sqrt{-1}) \Mu (i_{2}\tau_{2}E)$ and $\Mu(i_{1}\tau_{1}^{L}E)$ on $\bb R^{2}$
is just 
\[
2 \I Z_{1}(\tau_{1}^{L}E) \cdot \I Z_{2}(\tau_{2}^{R}E). 
\]
By (\ref{eq:keyeq}) and the assertion (3) we have 
 \begin{equation}
 |\arg (1+\sqrt{-1}) \Mu (i_{2}\tau_{2}E)- \arg \Mu(i_{1}\tau_{1}^{L}E)| < \pi/2, 
 \end{equation}
 and this gives the proof. 
\end{proof}

\begin{rmk}\label{rmk:discrete}
In the assertion (2), we assume that the stability conditions $\sigma_{i}$ is discrete. 
If $\sigma_{i}$ is rational then $\sigma_{i}$ is discrete since $K_{0}(\mb D)$ is finitely generated. 
However, the existence of rational stability conditions is subtle. 
Of course, if $\sigma_{i}$ is full (equivalently satisfies the support property), then one can choose rational stability conditions by shrinking $\sigma_{i}$. 
\end{rmk}

\begin{dfn}\label{dfn:torsionpairs}
Let $\mb D=\sod{\mb D_{1}}{\mb D_{2}}$ be a semiorthogonal decomposition with $\m1$. 
Let $\sigma_{i}=(\mca A_{i}, Z_{i})$ be a discrete stability condition on $\mb D_{i}$. 
Suppose $(\sigma_{1}, \sigma_{2})$ satisfies $\m4$.

Let us define $\mu_{\beta, \omega }(E)$ for $E \in \gl{\mca A_{1}}{\mca A_{2}}$ by 
\[
\mu_{\beta, \omega }(E)	=
\frac{\I Z_{1}(\tau_{1}^{L}E) -\omega \R Z_{2}(\tau_{2}^{R}E)  + \beta \I Z_{2}(\tau_{2}^{R}E)}{\I Z_{1}(\tau_{1}^{L}E) + \I Z_{2}(\tau_{2}^{R}E)}. 
\]

Similarly to the case of slope stability on projective surfaces, 
$(\sigma_{1},\sigma_{2})$-free object $E$ is $\Mu$-semistable if and only if 
for any nontrivial subobject $F \subset E$ so that $\tau_{2}^{R}F$ is a proper subobject of $\tau_{1}^{L}E$, 
the inequality $\mu_{\beta, \omega }(F) \leq \mu_{\beta, \omega }(E)$ holds (cf. \cite{MR1450870}).

Define the following subcategories of $\gl{\mca A_{1}}{\mca A_{2}}$ by  
\begin{align*}
\tilt{\mca F}{\sigma_{1}}{\sigma_{2}}		&=	\left\{	E \in \gl{\mca A_{1}}{\mca A_{2}}	\mid \text{$E$ is $(\sigma_{1}, \sigma_{2})$-free with $\mu_{\beta, \omega }^{+}(E) \leq 0$}	\right\}	\text{, and }\\
\tilt{\mca T}{\sigma_{1}}{\sigma_{2}}	&=	\left\{	E \in \gl{\mca A_{1}}{\mca A_{2}}	\mid	E \text{ is $(\sigma_{1}, \sigma_{2})$-torsion or  $(\sigma_{1}, \sigma_{2})$-free part $E_{\mr{fr}}$ of $E$ has $\mu_{\beta, \omega}^{-}(E_{\mr{fr}})>0$}	\right\}. 
\end{align*}
The tilting heart of $\gl{\mca A_{1}}{\mca A_{2}}$ by the torsion pair $(\tilt{\mca T}{\sigma_{1}}{\sigma_{2}}	, \tilt{\mca F}{\sigma_{1}}{\sigma_{2}}	)$ is denoted by $\tilt{\mca A}{\sigma_{1}}{\sigma_{2}}$. 
\end{dfn}

\begin{lem}\label{lem:mormono}
Let $\mb D=\sod{\mb D_{1}}{\mb D_{2}}$ with the condition $\m1$. 
Suppose that rational stability conditions $\sigma_{i}=(\mca A_{i}, Z_{i}) \in \Stab{\mb D_{i}} $ satisfy the condition $\m4$. 

If a $(\sigma_{1}, \sigma_{2})$-free object $E \in \gl{\mca A_{1}}{\mca A_{2}}$ satisfies $\mu_{\beta, \omega}^{+}(E) <1$, 
then the canonical morphism $\tau_{1}^{L}E \to \Phi(\tau_{2}^{R}E)$ is a monomorphism 
in $\mca A_{1}$. 
\end{lem}

\begin{proof}
By Lemma \ref{lem:morexact} we have the distinguished triangle in $\mb D_{1}$: 
\begin{equation*}
\xymatrix{
\tau_{1}^{R}E	\ar[r]^-{f}	&	\tau_{1}^{L}E	\ar[r]^-{g}	&	\Phi(\tau_{2}^{R}E)		\ar[r]	&	\tau_{1}^{R}E[1]
}
\end{equation*}
Since $E$ is in $\gl{\mca A_{1}}{\mca A_{2}}$, 
both $\tau_{1}^{L}E$ and $\Phi(\tau_{2}^{R}E)$ is in $\mca A_{1}$.

Let $K $ be the kernel of the morphism $g$ above. 
Note that the morphism $ \iota  \colon K \to \tau_{1}^{L}E$ lifts to $\tau_{1}^{R}E$, that is, there exists a morphism $\bar \iota \colon K \to \tau_{1}^{R}E$ such that $f \circ \bar \iota =\iota$. 
If $K$ is non-zero, then we have a non-zero morphism in 
\[
\Hom_{\mb D_{1}}( K, \tau_{1}^{R}E) \cong \Hom_{\mb D}(i_{1}K , E). 
\]
By Proposition \ref{prop:morslope} (3), $i_{1}K$ is $\mu_{\beta, \omega}$-semistable with $\mu_{\beta, \omega}(i_{1}K)=1$. 
Hence $K$ has to be zero since both $i_{1}K$ and $E$ are $(\sigma_{1}, \sigma_{2})$-free. 
\end{proof}

\section{Rational coefficients stability conditions}

In the previous section, we introduced the tilting heart $\til{\mca A}{\beta}{\omega}$. 
The aim of this section is to show that the pair $(\tilt{\mca A}{\sigma_{1}}{\sigma_{2}}, \tilt{Z}{\sigma_{1}}{\sigma_{2}})$ 
gives a locally finite stability condition on $\mb D=\sod{\mb D_{1}}{\mb D_{2}}$ 
if the semiorthogonal decomposition satisfies $\m2$ and $(\sigma_{1}, \sigma_{2})$ satisfies $\m5$. 
These assumptions are necessary to make all the arguments work well.

\begin{dfn}\label{dfn:tiltingdeformation}
Let $\mb D=\sod{\mb D_{1}}{\mb D_{2}}$ with the condition $\m2$. 
Suppose that rational stability conditions $\sigma_{i}=(\mca A_{i}, Z_{i}) \in \Stab{\mb D_{i}} $ satisfy the condition $\m5$. 
Define $\tilt{Z}{\sigma_{1}}{\sigma_{2}} \colon K_{0}(\mb D) \to \bb C$ by
\begin{equation}
\tilt{Z}{\sigma_{1}}{\sigma_{2}} (E) = Z_{1}(\tau_{1}^{L}E) + (\beta -\sqrt{-1}\omega) Z_{2}(\tau_{2}^{R}E). 
\end{equation}
Moreover the pair $(\tilt{\mca A}{\sigma_{1}}{\sigma_{2}}, \tilt{Z}{\sigma_{1}}{\sigma_{2}})$ is denoted by $\til{\Sigma}{\beta}{\omega}$. 
\end{dfn}

\begin{prop}\label{prop:centralcharge}
Let $\mb D=\sod{\mb D_{1}}{\mb D_{2}}$ be a semiorthogonal decomposition with the condition $\m2$. 
Suppose that ratonal stability conditions $\sigma_{i}=(\mca A_{i}, Z_{i}) \in \Stab{\mb D_{i}}$ satisfy the condition $\m5$. 
Then $\tilt{Z}{\sigma_{1}}{\sigma_{2}}$ is a central charge on $\tilt{\mca A}{\sigma _{1}}{\sigma _{2}}$ for any $\beta \in \bb R$ and $\omega \in \bb R_{>0}$. 
\end{prop}

\begin{proof}
If an object $E \in \tilt{\mca T}{\sigma_{1}}{\sigma_{2}}$ is a $(\sigma_{1}, \sigma_{2})$-torsion object, then 
$\I Z_{1}(\tau_{1}^{L}E)= \I Z_{2}(\tau_{2}^{R}E)=0$. 
Hence we have 
$\I \tilt{Z}{\sigma_{1}}{\sigma_{2}}(E) = -\omega \R Z_{2}(\tau_{2}^{R}E) \geq 0$ and the equality holds if and only if $\tau_{2}^{R}E=0$. 
Thus $E$ is isomorphic to a $\sigma_{1}$-torsion object $E_{1} \in \mca T^{\sigma_{1}}$ and  we see 
$\R \tilt{Z}{\sigma_{1}}{\sigma_{2}} (E) = \R Z_{1}(E_{1}) \leq 0$ and the equality holds if and only if $E_{1}=0$. 
Moreover if the object $E \in \tilt{\mca T}{\sigma_{1}}{\sigma_{2}}$ is $(\sigma_{1}, \sigma_{2})$-free, 
the inequality $\mu_{\beta, \omega }(E) >0$ implies 
$\I \tilt{Z}{\sigma_{1}}{\sigma_{2}}(E)>0$.

Now take $E \in \tilt{\mca F}{\sigma_{1}}{\sigma_{2}}$. 
Without loss of generality, 
we may assume that the $(\sigma_{1}, \sigma_{2})$-free object $E$ is $\Mu$-semistable. 
Since $\I \tilt{Z}{\sigma_{1}}{\sigma_{2}}(E) \leq 0$ holds by the assumption $E \in \tilt{\mca F}{\sigma_{1}}{\sigma_{2}}$, we have to show   the following: 
\begin{itemize}
\item If $\I \tilt{Z}{\sigma_{1}}{\sigma_{2}}(E) =0$, then $\R \tilt{Z}{\sigma_{1}}{\sigma_{2}}(E) $ is positive. 
\end{itemize}

Let us suppose $\I \tilt{Z}{\sigma_{1}}{\sigma_{2}}(E) =0$ holds. 
By the inequality $\omega >0$,  we have  
\begin{equation}\label{eq:trivial}
 \R Z_{2}(\tau_{2}^{R}E) = \frac{1}{\omega}\cdot (\I Z_{1}(\tau_{1}^{L}E) + \beta \I Z_{2}(\tau_{2}^{R}E)). 
\end{equation}

Since the morphism $\tau_{1}^{L}E \to \Phi(\tau_{2}^{R}E)$ is a monomorphism in $\mca A_{1}$ by Lemma \ref{lem:mormono}, 
the condition $\m5$ implies 
\begin{equation}\label{eq:mono}
\I Z_{1}(\tau_{1}^{L}E)	\leq \I Z_{1} \left(\Phi(\tau_{2}^{R}E) \right) = \I Z_{2}(\tau_{2}^{R}E)
\end{equation}
Moreover 
there exists a subobject $\tilde F$ of $E$ in $\gl{\mca A_{1}}{\mca A_{2}}$ such that 
$\tau_{1}^{L}\tilde F \cong \tau_{1}^{L}E$ and $\Phi (\tau_{2}^{R}\tilde F) \cong \tau_{1}^{L}E$ by Lemma \ref{lem:morsub} and Remark \ref{rmk:morsub}. 
The $\Mu$-semistability of $E$ implies $\mu_{\beta, \omega }^{+}(\tilde F) \leq \mu_{\beta, \omega }(E) \leq 0$.

Now the condition $\m5$ implies 
\begin{align*}
\mu_{\beta, \omega }(\tilde F)	&=\frac{\I Z_{1}(\tau_{1}^{L}\tilde{F}) -\omega \R Z_{2}(\tau_{2}^{R}\tilde{F})  + \beta \I Z_{2}(\tau_{2}^{R}\tilde{F})}{\I Z_{1}(\tau_{1}^{L}\tilde{F}) + \I Z_{2}(\tau_{2}^{R}\tilde{F})}		\\
	&=\frac{\I Z_{1}(\tau_{1}^{L}\tilde{F}) -\omega \R Z_{1}(\Phi (\tau_{2}^{R}\tilde{F}))  + \beta \I Z_{1}(\Phi (\tau_{2}^{R}\tilde{F}))}{\I Z_{1}(\tau_{1}^{L}\tilde{F}) + \I Z_{1}(\Phi (\tau_{2}^{R}\tilde{F}))}	\\
	&=\frac{\I Z_{1}(\tau_{1}^{L}E) -\omega \R Z_{1}(\tau_{1}^{L}E)  + \beta \I Z_{1}(\tau_{1}^{L}E)}{2\I Z_{1}(\tau_{1}^{L}E)} \leq 0. 	
\end{align*}
Since $\I Z_{1}(\tau_{1}^{L}E)$ is positive, we see 
\begin{equation}\label{eq:sub}
\frac{1}{\omega }\left(\I Z_{1}(\tau_{1}^{L}E) + \beta \I Z_{1}(\tau_{1}^{L}E) \right)\leq  \R Z_{1}(\tau_{1}^{L}E). 
\end{equation}
Thus (\ref{eq:trivial}), (\ref{eq:mono}) and (\ref{eq:sub}) imply 
\begin{align*}
\R \tilt{Z}{\sigma_{1}}{\sigma_{2}}(E)	&=	\R Z_{1}(\tau_{1}^{L}E) +\beta \R Z_{2}(\tau_{2}^{R}E)+\omega \I Z_{2}(\tau_{2}^{R}E)	\\
	&=	\R Z_{1}(\tau_{1}^{L}E) +\frac{\beta}{\omega} (\I Z_{1}(\tau_{1}^{L}E) + \beta \I Z_{2}(\tau_{2}^{R}E))+\omega \I Z_{2}(\tau_{2}^{R}E)	\\
	&=	\R Z_{1}(\tau_{1}^{L}E) + \frac{\beta}{\omega} \I Z_{1}(\tau_{1}^{L}E) + \frac{\beta^{2}+\omega ^{2}}{\omega }\I Z_{2}(\tau_{2}^{R}E)	\\
	&\geq 	\R Z_{1}(\tau_{1}^{L}E) + \frac{\beta}{\omega} \I Z_{1}(\tau_{1}^{L}E) + \frac{\beta^{2}+\omega ^{2}}{\omega }\I Z_{1}(\tau_{1}^{L}E)	\\
	&\geq \frac{1}{\omega }\left(\I Z_{1}(\tau_{1}^{L}E) + \beta \I Z_{1}(\tau_{1}^{L}E) \right)+ \frac{\beta}{\omega} \I Z_{1}(\tau_{1}^{L}E) + \frac{\beta^{2}+\omega ^{2}}{\omega }\I Z_{1}(\tau_{1}^{L}E)	\\
	&= \frac{(1+\beta )^{2}+\omega ^{2}}{\omega } \cdot \I Z_{1}(\tau_{1}^{L}E)>0. 
\end{align*}
\end{proof}

\begin{rmk}
Suppose that a semiorthogonal decomposition $\mb D= \sod{\mb D_{1}}{\mb D_{2}}$ satisfies $\m2$. 
If stability conditions $\sigma _{i}=(\mca A_{i}, Z_{i}) \in \Stab{\mb D_{i}}$ satisfy $\m5$, 
then $\sigma_{2}$ is $\sigma_{1}[-1]$. 
\end{rmk}

\begin{prop}
Let $\mb D=\sod{\mb D_{1}}{\mb D_{2}}$ be a semiorthogonal decomposition with $\m2$. 
Suppose stability conditions $\sigma _{i}=(\mca A_{i}, Z_{i}) \in \Stab{\mb D_{i}}$ are rational and satisfy $\m5$.
If $\sigma_{1}$ and $\sigma_{2}$ are rational and both $\omega$ and $\beta$ are rational then 
the pair $\tilt{\Sigma}{\sigma_{1}}{\sigma_{2}}=(\tilt{\mca A}{\sigma_{1}}{\sigma_{2}}, \tilt{Z}{\sigma_{1}}{\sigma_{2}})$ is a locally finite stability condition. 
\end{prop}

\begin{proof}Since $\sigma_{i}$ is rational, $\sigma_{i}$ is discrete. 
In particular the heart $\mca A_{i}$ of $\sigma_{i}$ is Noetherian and so is 
$\gl{\mca A_{1}}{\mca A_{2}}$.

Suppose $E \in \tilt{\mca A}{\sigma_{1}}{\sigma _{2}}$ satisfies $\I \tilt{Z}{\sigma_{1}}{\sigma_{2}}(E)=0$. 
Due to \cite[Lemma 3.4]{MR2721656}, it is enough to show that any increasing filtration $E^{1}\subset E^{2} \subset \cdots \subset E^{n}\subset \cdots $ of subobject of $E \in \tilt{\mca A}{\sigma_{1}}{\sigma_{2}} $ terminates. 
Note that $E^{n}$ also satisfies $\I \tilt{Z}{\sigma_{1}}{\sigma_{2}}(E^{n})=0$. 

Taking the cohomology with respect to $\gl{\mca A_{1}}{\mca A_{2}}$, we have a sequence of monomorphisms
\[
H^{-1}(E^{n}) \subset H^{-1}(E^{n+1}) \subset H^{-1}(E). 
\]
We can assume $H^{-1}(E^{n})$ does not depend on $n$ since $\gl{\mca A_{1}}{\mca A_{2}}$ is Noetherian. 
Now we claim $H^{0}(E^{n})$ satisfies 
\begin{align}
\tau_{2}^{R}H^{0}(E^{n})&=0, \text{ and}	\label{eq:a}	\\
\tau_{1}^{L}H^{0}(E^{n}) &\in \mca T^{\sigma_{1}}. 	\label{eq:b}
\end{align}
Indeed, by the vanishing $\I \tilt{Z}{\sigma_{1}}{\sigma_{2}}(E^{n})=0$, 
we have $\I \tilt{Z}{\sigma_{1}}{\sigma_{2}} (H^{0}(E^{n}))=0$ which implies 
that $H^{0}(E^{n})$ is $(\sigma_{1}, \sigma_{2})$-torsion. 
The vanishing $\I \tilt{Z}{\sigma_{1}}{\sigma_{2}} (H^{0}(E^{n}))=0$ and the definition of $\til{Z}{\beta}{\omega}$ 
imply $\I Z_{1}(\tau_{1}^{L}H^{0}(E^{n}))=0$ and $\tau_{2}^{R}H^{0}(E^{n})=0$. 
This gives the proof of the claim.

 Let $F^{n}$ be the quotient $E^{n}/E^{n-1}$ in $\tilt{\mca A}{\sigma _{1}}{\sigma _{2}}$. 
 Since $H^{-1}(E^{n})$ is constant for $n$, we have a monomorphism in $\gl{\mca A_{1}}{\mca A_{2}}$
 \[
 \xymatrix{
 H^{-1}(F^{n})	\ar[r]	&	H^{0}(E^{n-1}). 
 }
 \] 
 Thus $\I \gl{Z_{1}}{Z_{2}}\left(H^{-1}(F^{n}) \right)=0$ holds by (\ref{eq:a}) and (\ref{eq:b}). 
On the other hand, if $H^{-1}(F^{n})$ is non-zero, then $\I \gl{Z_{1}}{Z_{2}}\left(H^{-1}(F^{n})\right)$ is positive by the $(\sigma_{1}, \sigma_{2})$-freeness of $H^{-1}(F)$. 
So $H^{-1}(F^{n})$ has to be $0$ and the canonical morphism 
$H^{0}(E^{n-1}) \to H^{0}(E^{n})$ is a monomorphism in $\gl{\mca A_{1}}{\mca A_{2}}$. 

To complete the proof, it is enough to show that $H^{0}(E^{n})$ is independent of $n$. 
By (\ref{eq:a}) and (\ref{eq:b}), we have $\R \tilt{Z}{\sigma_{1}}{\sigma_{2}} (H^{0}(E^{n}))= \R Z_{1}(\tau_{1}^{R}H^{0}(E^{n})) < 0$, 
it is enough to show that $ \R \tilt{Z}{\sigma_{1}}{\sigma_{2}} (H^{0}(E^{n}))$ is bounded below. 

Let $G^{n}$ be the quotient $E/E^{n}$ in $\tilt{\mca A}{\sigma_{1}}{\sigma_{2}}$. 
Then we have the following exact sequence in $\gl{\mca A_{1}}{\mca A_{2}}$: 
\begin{equation*}
\xymatrix{
0\ar[r]	&H^{-1}(E^{n})	\ar[r]	&	H^{-1}(E)	\ar[r]^-{\varphi_{n}}	&	H^{-1}(G^{n})	\ar[r]^-{\delta_{n}}	&	
H^{0}(E^{n})	\ar[r]^-{\psi_{n}}	&	H^{0}(E)	\ar[r]^-{\pi_{n}}	&	H^{0}(G^{n})\ar[r]	& 0
}
\end{equation*}
Since $H^{-1}(E^{n})= \ker  \varphi _{n}$ is independent of $n$, so is $\im \varphi _{n}=\ker  \delta _{n}$. 
Moreover we may assume that $\ker  \pi_{n}=\im \psi _{n}$ is independent of $n$ since $\gl{\mca A_{1}}{\mca A_{2}}$ is Noetherian. 
Then it is enough to show that $\R \tilt{Z}{\sigma_{1}}{\sigma_{2}}(\ker  \psi_{n})$ is bounded below since $H^{0}(E^{n})$ is an extension by $\im \psi _{n}$ and $\ker  \psi_{n}$: 
\[
\xymatrix{
0	\ar[r]	&	\ker  \psi_{n}	\ar[r]	&	H^{0}(E^{n})	\ar[r]	&	\im  \psi_{n}	\ar[r]	&	0. 
}
\]

Note that $\tau_{2}^{R} (\ker  \psi _{n}) =0$ since the functor $\tau_{2}^{R}$ is $t$-exact with respect to $\gl{\mca A_{1}}{\mca A_{2}}$ and $\mca A_{2}$. 
Consider the exact sequence in $\gl{\mca A_{1}}{\mca A_{2}}$: 
\begin{equation*}
\xymatrix{
0	\ar[r]	&	\ker  \delta_{n}	\ar[r]	&	H^{-1}(G^{n})	\ar[r]	&	\ker  \psi_{n}	\ar[r]	&	0. 
}
\end{equation*}
Then Lemma \ref{lem:morexact} implies the following diagram of distinguished triangle in $\mb D_{1}$: 
\begin{equation}
\xymatrix{
\tau_{1}^{R}(\ker  \delta_{n})	\ar[r]\ar[d]	&	\tau_{1}^{R}(H^{-1}(G^{n}))	\ar[r]\ar[d]	&	\tau_{1}^{R}(\ker  \psi_{n})	\ar[d]^-{\cong }	\\
\tau_{1}^{L}(\ker  \delta_{n})	\ar[r]\ar[d]_-{\nu}	&	\tau_{1}^{L}(H^{-1}(G^{n}))	\ar[r]\ar[d]_-{\epsilon}	&	\tau_{1}^{L}(\ker  \psi_{n})	\ar[d]	\\
\Phi(\ker  \delta_{n})	\ar[r]^-{\cong}			&	\Phi(H^{-1}(G^{n}))		\ar[r]		&	0
}
\end{equation}
The second and third row give short exact sequences in $\mca A_{1}$. 
Since $\ker  \delta_{n}$ and $H^{-1}(G^{n})$ are in $\tilt{\mca F}{\sigma_{1}}{\sigma_{2}}$, 
the morphism $\nu$ and $\epsilon$ are mono morphisms in $\mca A_{1}$ by Lemma \ref{lem:mormono}. 
Thus we see the following isomorphism
\[
\tau_{1}^{L}(\ker  \psi_{n}) \cong 
\frac{\tau_{1}^{L}(H^{-1}(G^{n}))}{\tau_{1}^{L}(\ker  \delta _{n})}
\cong 
\ker  \left(
\frac{\Phi(\ker  \delta_{n})}{\tau_{1}^{L}(\ker  \delta _{n})}	\to
\frac{\Phi(H^{-1}(G^{n}))}{\tau_{1}^{L}(H^{-1}(G^{n}))}
\right). 
\] 

The $t$-exactness of $\tau_{1}^{L}$ with respect to $\gl{\mca A_{1}}{\mca A_{2}}$ and $\mca A_{1}$ implies 
$\tau_{1}^{L}(\ker  \psi_{n}) \subset \tau_{1}^{L}(H^{0}(E^{n}))$. 
Since $H^{0}(E^{n})$ satisfies (\ref{eq:a}) and (\ref{eq:b}), 
$\tau_{1}^{L}(\ker  \psi_{n})$ is in $\mca T^{\sigma_{1}}$ and $\tau_{2}^{R}(\ker  \psi _{n})=0$. 
Moreover $\tau_{1}^{L}(\ker  \psi_{n})$ is a subobject of the torsion part $T$ of ${\Phi(\ker  \delta_{n})}/{\tau_{1}^{L}(\ker  \delta _{n})}$ with respect to the torsion pair $(\mca T^{\sigma_{1}}, \mca F^{\sigma_{1}})$. 

Thus  we see $\tilt{Z}{\sigma_{1}}{\sigma_{2}}(\ker  \psi _{n}) = Z_{1}(\tau_{1}^{L}(\ker  \psi_{n}))$ and 
\begin{equation}
\R Z_{1}(T)	\leq \R Z_{1}( \ker  \psi_{n}) = \R \tilt{Z}{\sigma_{1}}{\sigma_{2}}(\ker  \psi _{n}). 
\end{equation}
Since $\ker  \delta_{n}$ is independent of $n$, so is $\Phi(\ker  \delta_{n})/\tau_{1}^{L}(\ker  \delta_{n})$. 
Hence $\R \tilt{Z}{\sigma_{1}}{\sigma_{2}}(\ker  \psi _{n})$ is bonded below 
and the pair $\tilt{\Sigma}{\sigma_{1}}{\sigma_{2}}=(\tilt{\mca A}{\sigma_{1}}{\sigma _{2}}, \tilt{Z}{\sigma_{1}}{\sigma_{2}})$ is a stability condition on $\mb D$. 

Finally by the rationality of $\sigma_{i}$ and $(\beta, \omega)$, the stability condition $\til{\Sigma}{\beta}{\omega}$ is rational. 
Since $K_{0}(\mb D)$ is finitely generated, a rational stability condition is reasonable and in particular is locally finite. 
\end{proof}

\section{Deformation property}\label{sc:deformation}
Let $\mb D=\sod{\mb D_{1}}{\mb D_{2}}$ be a semiorthogonal decomposition with $\m2$. 
Fixing $\sigma_{i} \in \Stab{\mb D_{i}}$ with $\m5$, 
we have constructed a family $\{	\tilt{\Sigma}{\sigma_{1}}{\sigma_{2}}		\mid	\beta \in \bb Q, \omega \in \bb Q_{>0}	\}$ of locally finite stability conditions on $\mb D$. 
Now we wish to extend the family for real numbers and to show that the family is continuous for $\beta$ and $\omega$. 

Although one of standard solutions is the support property of $\til{\Sigma}{\beta}{\omega}$, 
we show that 
the central charge $\til{Z}{\beta '}{\omega '}$ for $(\beta', \omega ')$ satisfies a finiteness condition 
\begin{equation}
\|\til{Z}{\beta '}{\omega '}	\|	_{\til{\Sigma}{\beta}{\omega}} < \infty 
\end{equation}
which has been established in Bridgeland's deformation of stability conditions (see also Theorem \ref{thm:Bridgeland}). 
To show this, by the definition of $\til{Z}{\beta'}{\omega'}$, it is enough to show that 
the supremum 
\begin{equation}\label{eq:sup}
\sup	\left\{	\frac{| \til{Z}{\beta}{\omega}( \tau_{2}^{R} E) | }{|\til{Z}{\beta}{\omega}(E) |}	\middle|		E \text{ is $\til{\Sigma}{\beta}{\omega}$-semistable}		\right\}
\end{equation}
is finite. 
This will be proven in Proposition \ref{prop:fake}. 
In the proposition, we only prove (\ref{eq:sup}) for points $(\beta, \omega)$ in a contractible open set $\mca H^{+}(\epsilon_{1}) \cap \mca H^{-}(\epsilon_{2})$ of the upper half-plane $\mca H$. 
The restriction to $\mca H^{+}(\epsilon_{1}) \cap \mca H^{-}(\epsilon_{2})$ is necessary for Proposition \ref{prop:fake} by a technical reason.

A key ingredient of Proposition \ref{prop:fake} is Corollary \ref{bunkai-sup}. 
If the supremum of  $	|	\arg \til{Z}{\beta}{\omega} (\tau_{1}^{L}E) -	\arg \til{Z}{\beta}{\omega}(\tau_{2}^{R}E)	|$ for any semistable object $E$ was smaller than $\pi$, 
then the desired assertion followed from Corollary \ref{bunkai-sup}. 
However, we do not see whether the canonical decomposition 
\[
\xymatrix{
i_{2}\tau_{2}^{R}	E	\ar[r]	&	E	\ar[r]	&	i_{1}\tau_{1}^{L}E	\ar[r]	&	i_{2}\tau_{2}^{R}	E[1]
}
\]
satisfies the desired property or not since the decomposition is too ``rough''.

Thus a finer decomposition of the semistable object $E$ is necessary. 
Roughly ``finer decomposition''  reflect the following principle: 
\begin{itemize}
\item 
For a $(\sigma_{1}, \sigma_{2})$-free object $F \in \gl{\mca A_{1}}{\mca A_{2}}$, if the difference 
$\mu_{\beta, \omega}^{+}(F) - \mu_{\beta, \omega}^{-}(F)$ is sufficiently small, then 
the difference 
$	|	\arg \til{Z}{\beta}{\omega} (\tau_{1}^{L}F) -	\arg \til{Z}{\beta}{\omega}(\tau_{2}^{R}F)	|$ 
satisfies the assumption in Corollary \ref{bunkai-sup}. 
\end{itemize}
To complete Proposition \ref{prop:fake}, 
it is necessary to observe properties of $\mu_{\beta, \omega}$-semistable object. 
%

\begin{lem}\label{lem:mu>1}
Let $\mb D=\sod{\mb D_{1}}{\mb D_{2}}$ be a semiorthogonal decomposition with $\m2$. 
Choose rational stability conditions $\sigma _{i} =(\mca A_{i}, Z_{i}) \in \Stab{\mb D_{i}}$ such that $(\sigma_{1}, \sigma_{2})$ satisfies the condition $\m5$. 
Set a full subcategory of $\gl{\mca A_{1}}{\mca A_{2}}$ by  \begin{equation}
\mca T_{\beta, \omega}^{>1}  = \{E \in \gl{\mca A_{1}}{\mca A_{2}}	\mid \text{the $(\sigma_{1}, \sigma_{2})$-free part $F$ of $E \in \gl{\mca A_{1}}{\mca A_{2}}$ satisfies $\mu_{\beta, \omega}^{-}(F) >1$}	\}. 
\end{equation}

\begin{enumerate}
\item For any $E \in \mca T_{\beta, \omega}^{>1}$, the inequality $\arg \til{Z}{\beta}{\omega}(E)	\geq \arg (\beta-1+\sqrt{-1}\omega)(\beta- \sqrt{-1}\omega)$ holds. 
\item The supremum 
\begin{equation}
\sup
\left\{
\frac{|\til{Z}{\beta}{\omega}(\tau_{2}^{R}E)|}
{|\til{Z}{\beta}{\omega}(E)|}
\middle|
E \in \mca T_{\beta, \omega}^{>1}
\right\}
\end{equation}
is finite. 
\end{enumerate}
\end{lem}

\begin{proof}
Let $T$ be the $(\sigma_{1}, \sigma_{2})$-torsion part of $E$ and $F $ the quotient $E/T$ in $\gl{\mca A_{1}}{\mca A_{2}}$. 
The definition of $\til{Z}{\beta}{\omega}$ implies 
\begin{equation}\label{eq:21}
\arg (-\beta +\sqrt{-1}\omega)	\leq \arg \til{Z}{\beta}{\omega}(T)	\leq \pi. 
\end{equation}
The assumption $\mu_{\beta, \omega}^{-}(F)>1$ implies $\tau_{1}^{L}F=0$ by the assertion (3) in Proposition \ref{prop:morslope}. 
Hence we have 
\begin{equation*}
\mu_{\beta, \omega}(F)=\frac{-\omega \R Z_{2}(\tau_{2}^{R}F) + \beta Z_{2}(\tau_{2}^{R}F)}{\I Z_{2}(\tau_{2}^{R}F)}	>1
\end{equation*}
which implies $\I \left(	(\beta -1-\sqrt{-1}\omega)Z_{2}(\tau_{2}^{R}F) \right)>0$. 
Thus we have 
\begin{equation}\label{eq:22}
\arg (\beta -1+\sqrt{-1}\omega)(\beta -\sqrt{-1}\omega)	<	\arg  (\beta -\sqrt{-1}\omega)Z_{2}(\tau_{2}^{R}F)  
=\arg \til{Z}{\beta}{\omega}(F). 
\end{equation}
Since $\til{Z}{\beta}{\omega}(E)$ is nothing but 
$\til{Z}{\beta}{\omega}(T)+\til{Z}{\beta}{\omega}(F)$, 
the inequalities (\ref{eq:21}) and (\ref{eq:22}) give the proof of the assertion (1). 

Since $F$ satisfies $\tau_{1}^{L}F=0$, we have $\arg \til{Z}{\beta}{\omega}(\tau_{1}^{L}E)= \pi$ unless $\tau_{1}^{L}E$ is zero. 
Hence there exists a $\theta \in (0,\pi)$ such that 
the inequality 
\begin{equation}
0\leq 	|	\arg \til{Z}{\beta}{\omega}(\tau_{1}^{L}E)	-\arg \til{Z}{\beta}{\omega}(\tau_{2}^{R}E)		|	\leq \theta
\end{equation}
for any $E \in \mca T_{\beta, \omega}^{>1}$. 
Then Corollary \ref{bunkai-sup} gives the proof of (2). 
\end{proof}

\begin{lem}
\label{lem:comp-slope}
Let $\mb D=\sod{\mb D_{1}}{\mb D_{2}}$ be a semiorthogonal decomposition with $\m2$. 
Choose rational stability conditions $\sigma _{i} =(\mca A_{i}, Z_{i}) \in \Stab{\mb D_{i}}$ such that $(\sigma_{1}, \sigma_{2})$ satisfies the condition $\m5$. 
Take $\epsilon_{1}$ and $\epsilon_{2}$ so that $ \epsilon_{2}\leq \epsilon_{1} < 1$. 
Suppose a $(\sigma_{1}, \sigma_{2})$-free object $E$ satisfies $\mu_{\beta, \omega}^{+}(E)<1$. 

\begin{enumerate}
\item Suppose $\tau_{1}^{L}E \neq 0$. 
If $E$ satisfies $\mu_{\beta, \omega }^{+}(E) \leq \epsilon_{1}$, then 
$\arg Z_{1}(\tau_{1}^{L}E) \leq \arg (\beta +1-2\epsilon_{1}+\sqrt{-1}\omega)$ 
for any $(\beta, \omega ) \in \bb R \times \bb R_{>0}$. 
\item 
If $E $ satisfies $\mu_{\beta, \omega }^{+}(E) \leq \epsilon_{1}$
, then 
$\arg Z_{2}(\tau_{2}^{R}E) \leq \arg (\beta -\epsilon_{1}+\sqrt{-1}\omega)$ 
for any $(\beta, \omega ) \in \bb R \times \bb R_{>0}$. 

\item If $E $ satisfies $\epsilon_{2} \leq  \mu_{\beta, \omega }(E) $ then 
$\arg (\beta + 1 -2\epsilon_{2}+\sqrt{-1}\omega) \leq  \arg Z_{2}(\tau_{2}^{R}E)$. 
\end{enumerate}
\end{lem}

\begin{proof}
 We first note that $\tau_{2}^{R}E$ is non-zero and in particular $\I Z_{2}(\tau_{2}^{R}E) \neq 0$. 
 Otherwise, $E$ is in the essential image of $i_{1} \colon \mb D_{1} \to \mb D$ and hence $E$ is $\mu_{\beta, \omega}$-semistable with $\mu_{\beta, \omega}(E)=1$ by Proposition \ref{prop:morslope}. 
 This gives a contradiction.

Since $\mu_{\beta, \omega }^{+}(E) <1$, the morphism $\tau_{1}^{L}E \to \Phi(\tau_{2}^{R}E)$ is a monomorphism in $\mca A_{1}$ by Lemma \ref{lem:mormono}. 
Moreover Lemma \ref{lem:morsub} (and the proof) implies that there exists a subobject $F \subset E \in \gl{\mca A_{1}}{\mca A_{2}}$ such that $\tau_{1}^{L}F \cong \tau_{1}^{L}E$ and $\Phi (\tau_{2}^{R}F) \cong \tau_{1}^{L}E$. 
Since the subobject $F$ also satisfies $\mu_{\beta, \omega }^{+}(F) \leq \epsilon_{1}$ by the assumption for $E$, we see: 
\begin{align*}
\epsilon_{1}	&	\geq	\frac{\I Z_{1}(\tau_{1}^{L}F)-\omega \R Z_{2}(\tau_{2}^{R}F) +\beta \I Z_{2}(\tau_{2}^{R}F)}{\I Z_{1}(\tau_{1}^{L}F) + \I Z_{2}(\tau_{2}^{R}F)}	\\
			&=	\frac{\I Z_{1}(\tau_{1}^{L}F)-\omega \R Z_{1}(\Phi (\tau_{2}^{R}F)) +\beta \I Z_{1}(\Phi (\tau_{2}^{R}F))}{\I Z_{1}(\tau_{1}^{L}F) + \I Z_{1}(\Phi (\tau_{2}^{R}F))}	\\
			&=	\frac{-\omega \R Z_{1}(\tau_{1}^{L}E) +(\beta +1)\I Z_{1}(\tau_{1}^{L}E)}{2\I Z_{1}(\tau_{1}^{L}E) }. 	
			\end{align*}
The last inequality implies 
\begin{equation}\label{eq:left}
-\omega \R Z_{1}(\tau_{1}^{L}E) +(\beta +1-2\epsilon_{1})\I Z_{1}(\tau_{1}^{L}E) \leq 0. 
\end{equation}
Since the left hand side of (\ref{eq:left}) is just the imaginary part of  $(\beta +1-2\epsilon _{1} - \sqrt{-1}\omega)Z_{1}(\tau_{1}^{L}E)$, 
this gives the proof of the assertion (1).

Note that $i_{2}\tau_{2}^{R}E$ is a subobject of $E$ in $\gl{\mca A_{1}}{\mca A_{2}}$. 
Hence $i_{2}\tau_{2}^{R}E$ also satisfies $\mu_{\beta, \omega }^{+}(i_{2}\tau_{2}^{R}E) \leq \epsilon_{1}$. 
So we see 
\begin{align*}
\epsilon_{1}	&	\geq \mu_{\beta, \omega }(i_{2}\tau_{2}^{R}E)	 =	\frac{-\omega \R Z_{2}(\tau_{2}^{R}E) + \beta \I Z_{2}(\tau_{2}^{R}E)}{\I Z_{2}(\tau_{2}^{R}E)}. 
\end{align*}
Thus we have 
\begin{equation}
-\omega \R Z_{2}(\tau_{2}^{R}E) + (\beta -\epsilon_{1} )\I Z_{2}(\tau_{2}^{R}E) \leq 0
\end{equation}
which gives the proof of the assertion (2). 

We show the assertion (3). 
By the assumption $\epsilon_{2}	\leq \mu_{\beta, \omega }(E)$, we see 
\begin{align*}
\epsilon_{2}(\I Z_{1}(\tau_{1}^{L}E) + \I Z_{2}(\tau_{2}^{R}E)) 
&\leq 
\I Z_{1}(\tau_{1}^{L}E) - \omega \R Z_{2}(\tau_{2}^{R}E) + \beta \I Z_{2}(\tau_{2}^{R}E)
\\
\iff 
(\epsilon_{2}-1)\I Z_{1}(\tau_{1}^{L}E)	&\leq 	-\omega \R Z_{2}(\tau_{2}^{R}E) + (\beta - \epsilon_{2})\I Z_{2}(\tau_{2}^{R}E). 
\end{align*}
Lemma \ref{lem:mormono} and the condition $\m5$ implies $\I Z_{1}(\tau_{1}^{L}E) \leq \I Z_{1}(\Phi (\tau_{2}^{R}E))= \I Z_{2}(\tau_{2}^{R}E)$. 
Hence we obtain
\begin{align}
\label{eq:above}
(\epsilon_{2}-1)\I Z_{2}(\tau_{2}^{R}E)	&\leq 	-\omega \R Z_{2}(\tau_{2}^{R}E) + (\beta - \epsilon_{2})\I Z_{2}(\tau_{2}^{R}E)	\\
\iff 0 &\leq -\omega \R Z_{2}(\tau_{2}^{R}E) + (\beta - 2\epsilon_{2}+1) \I Z_{2}(\tau_{2}^{R}E). 
\end{align}
By the same argument for (1), we see 
$\arg (\beta -2\epsilon_{2}+1 + \sqrt{-1}\omega) \leq \arg Z_{2}(\tau_{2}^{R}E)$. 
\end{proof}

\begin{lem}
\label{lem:slope-mor}
Let $\mb D=\sod{\mb D_{1}}{\mb D_{2}}$ be a semiorthogonal decomposition with $\m2$. 
Choose stability conditions $\sigma _{i} =(\mca A_{i}, Z_{i}) \in \Stab{\mb D_{i}}$ such that $(\sigma_{1}, \sigma_{2})$ satisfies the condition $\m
5$. 
Take $\epsilon_{1}$ and $\epsilon_{2}$ which satisfy 
\begin{equation*}
0 \leq \epsilon_{1},  0 < 1-2\epsilon_{2}, \epsilon_{2}	\leq \epsilon_{1}<1,  \text{ and }	
\end{equation*}
\begin{equation}
0 < \omega ^{2}+(\beta + 1-2\epsilon_{2})^{2}+2(1-2\epsilon_{2})(\epsilon_{2}-\epsilon_{1}). \label{eq:domain1}
\end{equation}
Suppose that a $(\sigma_{1}, \sigma_{2})$-free object $E \in \gl{\mca A_{1}}{\mca A_{2}}$ satisfies
\begin{equation}\label{eq:condition}
\epsilon_{2}\leq \mu_{\beta, \omega }(E) \leq \mu_{\beta ,\omega}^{+}(E) \leq \epsilon_{1}. 
\end{equation}
\begin{enumerate}
\item 
There exists a $\theta \in (0, \pi)$ such that 
the argument of $\til{Z}{\beta}{\omega}(E)$ satisfies 
\begin{equation*}
0	<	\arg \til{Z}{\beta}{\omega}(E)-\arg (\beta +1-2\epsilon_{2}+\sqrt{-1}\omega)(\beta -\sqrt{-1}\omega) \leq \theta_{0}
\end{equation*}
for any $E$ with the condition (\ref{eq:condition}). 
\item 
The supremum
\begin{equation*}
\sup
\left\{
\frac
{|\til{Z}{\beta}{\omega}(\tau_{2}^{R}E)|}
{|\til{Z}{\beta}{\omega}(E)|}
\middle|
E \text{ is $(\sigma_{1}, \sigma_{2})$-free with }
\epsilon_{2}\leq \mu_{\beta, \omega }(E) \leq \mu_{\beta ,\omega}^{+}(E) \leq \epsilon_{1} 
\right\}. 
\end{equation*}
is finite. 

\end{enumerate}
\end{lem}

\begin{proof}
Set $\theta_{1}=\arg (\beta +1-2\epsilon_{1}+\sqrt{-1}\omega)$, 
$\theta_{2}=\arg (\beta + 1 -2\epsilon_{2}+\sqrt{-1}\omega)	(\beta-\sqrt{-1}\omega )$, 
and 
$\theta_{3}= \arg (\beta -\epsilon_{1}+\sqrt{-1}\omega)(\beta -\sqrt{-1}\omega)$. 
Then the conditions $\epsilon_{1}>0$ and $1-2\epsilon_{2}>0$ imply that  
$\theta_{1}$ is in the open interval $(0, \pi)$, $\theta_{2}$ is non-positive, and $\theta _{3}$ is in $[0, \pi)$. 

Lemma \ref{lem:comp-slope} implies the following: 
\begin{align}
0	&<\arg \til{Z}{\beta}{\omega}(\tau_{1}^{L}E)	\leq  	\theta_{1}, \text{ and }	\label{eq:361}	\\
\theta_{2}	&<	\arg \til{Z}{\beta}{\omega}(\tau_{2}^{R}E)	\leq \theta_{3}. \label{eq:362}
\end{align}

We claim that the assumption for $(\epsilon_{1}, \epsilon_{2})$ implies 
\begin{align}\label{eq:convex1}	
0	<	\theta_{1}-\theta_{2} < \pi. 
\end{align}
Note that the inequalities (\ref{eq:convex1}) is equivalent to 
\begin{equation}\label{eq:convex-arg}
\I	\frac{\beta +1-2\epsilon_{1}+\sqrt{-1}\omega}{(\beta +1-2\epsilon_{2}+\sqrt{-1}\omega )(\beta -\sqrt{-1}\omega)} > 0. 
\end{equation}
Since the left hand side of (\ref{eq:convex-arg}) is  
$\omega \cdot (\omega ^{2}+(\beta + 1-2\epsilon_{2})^{2}+2(1- 2\epsilon_{2})(\epsilon_{2}-\epsilon_{1}))$ up to the positive constant 
$| (\beta +1-2\epsilon_{2}+\sqrt{-1}\omega )(\beta -\sqrt{-1}\omega) |^{2}$, 
the inequalities (\ref{eq:convex1}) hold.

Now the inequalities 
\begin{equation}\label{eq:convex2}
0	<	\theta_{3}-\theta_{2}	<	\pi
\end{equation}
also hold as follows. 
Since $\I (\beta-\epsilon_{1}+\sqrt{-1}\omega)/(\beta +1 +2\epsilon_{2}+\sqrt{-1}\omega)$ is $\omega(1 +\epsilon_{1}-2\epsilon_{2})$ up to positive constant, the assumption for $\epsilon_{1}$ and $\epsilon_{2}$ imply 
\begin{align*}
1+ \epsilon_{1}-2\epsilon_{2}	&\geq  1+\epsilon_{1}-2\epsilon_{1}=1-\epsilon_{1}>0. 
\end{align*}
The inequalities above imply (\ref{eq:convex2}).

Set $\theta_{0}$ by $\max \{	\theta_{1}-\theta_{2}, \theta_{3}-\theta_{2}\}$. 
Since $\theta_{2}$ is negative, 
the arguments $\arg \til{Z}{\beta}{\omega}(\tau_{1}^{L}E)$ and $\arg \til{Z}{\beta}{\omega}(\tau_{2}^{R}E)$ are in the interval 
$(\theta_{2}, \theta_{0}+\theta_{2}]$. 
By inequalities (\ref{eq:convex1}) and (\ref{eq:convex2}), $\theta_{0}$ is in $(0,\pi)$. 
Moreover (\ref{eq:361}) and (\ref{eq:362}) implies 
\begin{equation*}
\theta_{2}	<	\arg \til{Z}{\beta}{\omega}(E)	\leq \theta_{0} + \theta_{2}. 
\end{equation*}
Then Corollary \ref{bunkai-sup} implies the assertion (2). 
\end{proof}

\begin{rmk}
If $\epsilon_{1}=0$, then the inequality (\ref{eq:domain1}) is nothing but 
\begin{equation}\label{eq:domain2}
\omega^{2}+\beta^{2}+(2\beta+1)(2\epsilon_{2}+1) \geq 0. 
\end{equation}
Thus $(\beta, \omega)$ satisfies the (\ref{eq:domain2}) if $2\beta+ 1 \geq 0$. 
\end{rmk}

\begin{prop}
\label{prop:argsup1}
Let $\mb D=\sod{\mb D_{1}}{\mb D_{2}}$ be a semiorthogonal decomposition with $\m2$. 
Choose rational stability conditions $\sigma _{i} =(\mca A_{i}, Z_{i}) \in \Stab{\mb D_{i}}$ such that $(\sigma_{1}, \sigma_{2})$ satisfies the condition $\m5$. 

For an $\epsilon_{1} \in (0,1/2) $, set an open subset $\mca H^{+}(\epsilon_{1})$ by
\begin{equation}\label{eq:masi}
\mca H^{+}(\epsilon_{1})=
\left\{
(\beta , \omega) \in \bb R \times \bb R_{>0}	\middle|	
\begin{split}
0 &< \omega^{2} + (\beta +1)^{2}-2\epsilon_{1} , \text{ and }\\
0&< \omega ^{2}+(\beta +1-2\epsilon_{1})^{2}+2(1-2\epsilon_{1})(\epsilon_{1}-1)
\end{split} 
\right\}.
\end{equation}

Set a full subcategory $\mca T_{\beta, \omega}^{<1}$ of $\gl{\mca A_{1}}{\mca A_{2}}$ by 
\begin{equation}
\mca T_{\beta, \omega}^{<1}	=\{	E \in \gl{\mca A_{1}}{\mca A_{2}}	\mid E
\text{ is $(\sigma_{1}, \sigma_{2})$-free with }0 < \mu_{\beta, \omega}^{-}(E) \leq \mu_{\beta, \omega }^{+}(E) <1
	\}. 
\end{equation}

If $(\beta, \omega ) \in \mca H^{+}(\epsilon_{1})$, 
\begin{enumerate}
\item then there exists a $\theta \in (0,\pi)$ such that the following holds: 
\[
\sup\{	\arg\til{Z}{\beta}{\omega}(E) \mid E \in \mca T_{\beta, \omega}^{<1}	\} \leq \theta. 
\]
\item The supremum 
\begin{equation}
\sup
\left\{
\frac{|\til{Z}{\beta}{\omega}(\tau_{2}^{R}E)|}{|\til{Z}{\beta}{\omega}(E)|}
\middle|
E
\in 
\mca T_{\beta, \omega}^{<1}
\right\}
\end{equation}
is finite. 
\end{enumerate}
\end{prop}

\begin{proof}
Taking the Harder-Narasimhan filtration with respect to $\mu_{\beta, \omega}$-semistability, we obtain a short exact sequence in $\gl{\mca A_{1}}{\mca A_{2}}$
\begin{equation}
\xymatrix{
0	\ar[r]	&	E^{+}	\ar[r]	&	E	\ar[r]	&	E^{-}	\ar[r]	&	0, 
}
\end{equation}
where $E^{+}$ and $E^{-}$ satisfy the following: 
\begin{itemize}
\item $E^{+}$  is $(\sigma_{1}, \sigma_{2})$-free with $\epsilon_{1} < \mu_{\beta, \omega}^{-}(E^{+})\leq \mu_{\beta, \omega}^{+}(E^{+}) <1$ and 
\item $E^{-}$ is $(\sigma_{1}, \sigma_{2})$-free with $0 < \mu_{\beta, \omega}^{-}(E^{-}) \leq \mu_{\beta, \omega }^{+}(E^{-}) \leq \epsilon_{1}$. 
\end{itemize}

Now we note that two inequalities in (\ref{eq:masi}) come from the inequality (\ref{eq:domain1}) in Lemma \ref{lem:slope-mor}. 
Applying Lemma \ref{lem:slope-mor}, there exist $\theta_{+}$ and $\theta_{-}$ in $(0,\pi)$ such that $\til{Z}{\beta}{\omega}(E^{i}) \leq \theta_{i}$ where $i \in \{+, -\}$. 
If set $\max\{\theta_{+}, \theta_{-}\}$ by $\theta$, 
then the argument of $\til{Z}{\beta}{\omega}(E)=\til{Z}{\beta}{\omega}(E^{+})+\til{Z}{\beta}{\omega}(E^{-})$ is in the interval $(0, \theta]$ and this gives the proof of the assertion (1) since $\theta$ is independent of the choice of $E \in \mca T_{\beta, \omega}^{<1}$. 

Both $\arg \til{Z}{\beta}{\omega}(E^{+})$ and $\arg \til{Z}{\beta}{\omega}(E^{-})$ are positive 
since $\mu_{\beta, \omega}(E^{+})$ and $\mu_{\beta, \omega}(E^{-})$ are positive. 
Moreover the inequality $|\arg \til{Z}{\beta}{\omega}(E^{+})-\til{Z}{\beta}{\omega}(E^{-})| \leq \theta$ holds. 
Then Proposition \ref{bunkai-improve} implies that there exists a positive constant $C>0$ such that the following holds: 
\begin{equation}
|\til{Z}{\beta}{\omega}(E)| \geq C|\til{Z}{\beta}{\omega}(E^{+})|+C|\til{Z}{\beta}{\omega}(E^{-})|. 
\end{equation}

By Lemma \ref{lem:comp-slope}, there exist $\theta_{0}$ and $\theta_{0}'$ in the interval $(0,\pi)$ such that 
the following hold:
\begin{align}
\label{eq:33}	0	&<	|\arg\til{Z}{\beta}{\omega}(\tau_{1}^{L}E^{+})-\arg\til{Z}{\beta}{\omega}(\tau_{2}^{R}E^{+})|	\leq \theta_{0},	\text{ and}\\
\label{eq:34}	0	&<	|\arg\til{Z}{\beta}{\omega}(\tau_{1}^{L}E^{-})-\arg\til{Z}{\beta}{\omega}(\tau_{2}^{R}E^{-})|	\leq \theta_{0}'. \end{align}

Hence we see 
\begin{align*}
\frac{|\til{Z}{\beta}{\omega}(\tau_{2}^{R}E)|}{|\til{Z}{\beta}{\omega}(E)|}	&
\leq 
\frac{1}{C}\frac{|\til{Z}{\beta}{\omega}(\tau_{2}^{R}E^{+})+\til{Z}{\beta}{\omega}(\tau_{2}^{R}E^{-})|}{ |\til{Z}{\beta}{\omega}(E^{+})|+|\til{Z}{\beta}{\omega}(E^{-})|}	\\
&\leq 
\frac{1}{C}\left(	\frac{|\til{Z}{\beta}{\omega}(\tau_{2}^{R}E^{+})|}{|\til{Z}{\beta}{\omega}(E^{+})|+|\til{Z}{\beta}{\omega}(E^{-})|}
+\frac{|\til{Z}{\beta}{\omega}(\tau_{2}^{R}E^{-})|}{|\til{Z}{\beta}{\omega}(E^{+})|+|\til{Z}{\beta}{\omega}(E^{-})|}	\right)	\\
&\leq 
\frac{1}{C}\left(	
\frac{|\til{Z}{\beta}{\omega}(\tau_{2}^{R}E^{+})|}{|\til{Z}{\beta}{\omega}(E^{+})|}
+
\frac{|\til{Z}{\beta}{\omega}(\tau_{2}^{R}E^{-})|}{|\til{Z}{\beta}{\omega}(E^{-})|}	
\right)	
\end{align*}
Corollary \ref{bunkai-sup} implies that 
both ${|\til{Z}{\beta}{\omega}(\tau_{2}^{R}E^{+})|}/{|\til{Z}{\beta}{\omega}(E^{+})|}$ and 
${|\til{Z}{\beta}{\omega}(\tau_{2}^{R}E^{-})|}/{|\til{Z}{\beta}{\omega}(E^{-})|}$ are bounded above by (\ref{eq:33}) and (\ref{eq:34}). 
Hence the supremum is bounded above. 
\end{proof}

\begin{rmk}
In the final section, we wish to deform the stability condition $\til{\Sigma}{\beta}{\omega}$ along a path from $(1,0) \in \overline{\mca H}$ to $(\cos 2\pi/3, \sin 2\pi/3) \in \mca H$. 
If set $\epsilon_{1}$ by $1/3$, then the set $\mca H^{+}(1/3)$ includes the desired path. 
\end{rmk}

\begin{cor}\label{cor:mu<-1/2}
Let $\mb D=\sod{\mb D_{1}}{\mb D_{2}}$ be a semiorthogonal decomposition with $\m2$. 
Choose rational stability conditions $\sigma _{i} =(\mca A_{i}, Z_{i}) \in \Stab{\mb D_{i}}$ such that $(\sigma_{1}, \sigma_{2})$ satisfies the condition $\m
5$. 
Take $\epsilon_{2} \leq 0$ and $(\beta, \omega)\in \bb Q \times \bb Q_{>0}$ so that 
\begin{equation}
2\beta +1 -2\epsilon_{2}>0.  
\end{equation}
Set $\mca  F_{\beta, \omega}^{\leq \epsilon_{2}}$ by 
\begin{equation*}
\mca F_{\beta, \omega}^{\leq \epsilon_{2}}		:=	
\{
F \in \gl{\mca A_{1}}{\mca A_{2}}	\mid	
F\text{ is $(\sigma_{1}, \sigma_{2})$-free with }\mu_{\beta, \omega}^{+}(F)\leq \epsilon_{2}
\}
\end{equation*}
and take $F  \in \mca F_{\beta, \omega }^{\leq \epsilon_{2}}$.

\begin{enumerate}
\item The argument of $\til{Z}{\beta}{\omega}(F)$ is greater than $\arg (\beta- \sqrt{-1}\omega)$. 
In particular the following holds for any $F \in \mca F_{\beta, \omega}^{\leq \epsilon_{2}}$: 
\begin{equation}
\arg (\beta -\sqrt{-1}\omega)	<	\til{Z}{\beta}{\omega}(F)	\leq 0
\end{equation}

\item 
The supremum 
\begin{equation*}
\sup
\left\{
\frac{|\til{Z}{\beta}{\omega}(\tau_{2}^{R}F)|}
{|\til{Z}{\beta}{\omega}(F)|}
\middle|
F \in \mca F_{\beta, \omega}^{\leq \epsilon_{2}}
\right\}
\end{equation*}
is finite. 
\end{enumerate}
\end{cor}

\begin{proof}
The imaginary part $\I \frac{\beta +1-2\epsilon_{2}+\sqrt{-1}\omega}{\beta -\sqrt{-1}\omega} $
is $\omega(2\beta + 1 -2\epsilon_{2})$ up to positive constant. 
Then 
The assumption $2\beta +1 -2\epsilon_{2}>0$ directly implies 
\begin{equation}\label{eq:330}
0 < \arg \frac{\beta +1 -2\epsilon_{2}+\sqrt{-1}\omega}{\beta -\sqrt{-1}\omega} <\pi. 
\end{equation}

By the definition of $\til{Z}{\beta}{\omega}$, we have $\arg (-\beta +\sqrt{-1}\omega) < \arg \til{Z}{\beta}{\omega}(\tau_{2}^{R}E)$. 
Lemma \ref{lem:comp-slope} implies 
\begin{equation}\label{eq:ver}
0	<	\arg \til{Z}{\beta}{\omega}(\tau_{1}^{L}F) \leq \arg (\beta +1 -2\epsilon_{2}+\sqrt{-1}\omega), \text{and}
\end{equation}
\begin{equation}\label{eq:ver2}
\arg (\beta-\sqrt{-1}\omega) < \arg \til{Z}{\beta}{\omega}(\tau_{2}^{R}F) \leq \arg (\beta -\epsilon_{2}+\sqrt{-1}\omega)(\beta -\sqrt{-1}\omega). 
\end{equation}
Since $\arg (\beta -\epsilon_{2}+\sqrt{-1}\omega)(\beta -\sqrt{-1}\omega)$ is negative, 
the inequalities (\ref{eq:330}), (\ref{eq:ver}), and (\ref{eq:ver2}) imply 
$\arg (-\beta+\sqrt{-1}\omega)	<	\til{Z}{\beta}{\omega}(F)$. 	
In addition the inequalities $\mu_{\beta, \omega}^{+}(F)\leq \epsilon_{2} < 0$ imply $\arg \til{Z}{\beta}{\omega}(F) \leq 0$, and thus we have 
\begin{equation}
\arg (\beta- \sqrt{-1}\omega)	<	\til{Z}{\beta}{\omega}(F) \leq 0. 
\end{equation}

Finally, by (\ref{eq:ver}) and (\ref{eq:ver2}), Corollary \ref{bunkai-sup} implies the assertion (2). 
\end{proof}

The following lemma might be technical for readers, but is necessary for the proof of Proposition \ref{prop:free-angle}

\begin{lem}\label{lem:kihon1}
Let $\mb D=\sod{\mb D_{1}}{\mb D_{2}}$ be a semiorthogonal decomposition with $\m2$. 
Choose stability conditions $\sigma _{i} =(\mca A_{i}, Z_{i}) \in \Stab{\mb D_{i}}$ such that $(\sigma_{1}, \sigma_{2})$ satisfies the condition $\m
5$. 
Take $(\beta, \omega)\in \bb Q \times \bb Q_{>0}$. 

Suppose that $E \in \gl{\mca A_{1}}{\mca A_{2}}$ satisfies the following: 
\begin{itemize}
\item $\tau_{1}^{L}E$ is $\sigma_{1}$-free, and 
\item the $(\sigma_{1}, \sigma_{2})$-free part $E_{\mr{fr}}$ of $E$ has $\mu_{\beta, \omega }^{-}(E_{\mr{fr}})	\geq 1$. 
\end{itemize}
The the following holds: 
\begin{enumerate}
\item The $\sigma_{2}$-free part $F$ of $\tau_{2}^{R}E$ has the property $\mu_{\beta, \omega}^{-}(i_{2}F) \geq 1$. 
\item $\arg \til{Z}{\beta}{\omega}(\tau_{2}^{R}E)	\geq \arg (\beta-1+\sqrt{-1}\omega)(\beta -\sqrt{-1}\omega)$. 
\end{enumerate}
\end{lem}

\begin{proof}
Let $E_{\mr{tor}}$ be the $(\sigma_{1}, \sigma_{2})$-torsion part of $E$. 
By the assumption we see $\tau_{1}^{L}E_{\mr{tor}}=0$ and 
\begin{equation}\label{eq:42}
\arg \til{Z}{\beta}{\omega}(E_{\mr{tor}})=\arg \til{Z}{\beta}{\omega}(\tau_{2}^{R}E_{\mr{tor}})=\arg (-\beta +\sqrt{-1}\omega). 
\end{equation}

Let $A$ be an arbitrary subobject of $\tau_{2}^{R}E$ such that the quotient $\tau_{2}^{R}E/A$ is $\sigma_{2}$-free. 
Then $i_{2}A$ gives a subobject of $E$. 
Hence the quotient $Q= E/i_{2}A$ is $(\sigma_{1}, \sigma_{2})$-free with $\mu_{\beta, \omega}^{-}(Q)\geq 1$. 
Then the inequalities $1 \leq \mu_{\beta, \omega }^{-}(Q) \leq \mu_{\beta, \omega }(Q)$ imply 
\begin{equation}\label{eq:43}
-\omega \R Z_{2}(\tau_{2}^{R}Q) + (\beta -1)\I Z_{2} (\tau_{2}^{R}Q) \geq 0, 
\end{equation}
which means $\arg Z_{2}(\tau_{2}^{R}Q) \geq \arg (\beta-1+\sqrt{-1}\omega)$. 
The inequality (\ref{eq:43}) also implies 
\begin{equation}\label{eq:44}
\frac{-\omega \R Z_{2}(\tau_{2}^{R}Q)+ \beta \I Z_{2}(\tau_{2}^{R}Q)}{\I Z_{2}(\tau_{2}^{R}Q)} \geq 1. 
\end{equation}
Since $A$ is arbitrary, the inequality (\ref{eq:44}) implies the assertion (1). 

Recall $\til{Z}{\beta}{\omega}(\tau_{2}^{R}E)=\til{Z}{\beta}{\omega}(\tau_{2}^{R}E_{\mr{tor}})+\til{Z}{\beta}{\omega}(i_{2}F)$ 
and $Q = E/i_{2}$. 
Applying the inequality (\ref{eq:43}) $A$ as $\sigma_{2}$-torsion part of $\tau_{2}^{R}E$, 
we have 
\begin{equation}\label{eq:45}
\arg \til{Z}{\beta}{\omega}(i_{2}F) = \arg (\beta-\sqrt{-1}\omega)Z_{2}(F)	\geq \arg (\beta -1+\sqrt{-1}\omega)(\beta -\sqrt{-1}\omega). 
\end{equation}
Then (\ref{eq:42}) and (\ref{eq:45}) imply the assertion (2). 
\end{proof}

\begin{prop}\label{prop:free-angle}
Let $\mb D=\sod{\mb D_{1}}{\mb D_{2}}$ be a semiorthogonal decomposition with $\m2$. 
Choose rational stability conditions $\sigma _{i} =(\mca A_{i}, Z_{i}) \in \Stab{\mb D_{i}}$ such that $(\sigma_{1}, \sigma_{2})$ satisfies the condition $\m
5$. 
Take $(\beta, \omega)\in  \mca H^{+}(\epsilon_{1})_{\bb Q}$ for an $\epsilon_{1} \in (0,1/2)$ (cf. (\ref{eq:masi})).

Then the supremum 
\begin{equation}
\sup
\left\{
\frac{|\til{Z}{\beta}{\omega}(\tau_{2}^{R}T)|}{|\til{Z}{\beta}{\omega}(T)|} 
\middle|
T \in \til{\mca T}{\beta}{\omega} \cap 	\til{\mca P}{\beta}{\omega}(\phi) \text{ where 
}  \pi \phi \in (0, \arg (-\beta+\sqrt{-1}\omega)]
\right\}
\end{equation}
is finite. 
\end{prop}

\begin{proof}
We first note that the $(\sigma_{1}, \sigma_{2})$-torsion part $E_{\mr{tor}}$ of any $E \in \til{\mca T}{\beta}{\omega}$ satisfies 
\[
\arg (-\beta +\sqrt{-1}\omega)	\leq \arg \til{Z}{\beta}{\omega}(E_{\mr{tor}})	\leq \pi
\]
by the definition of $\til{Z}{\beta}{\omega}$. 
Moreover the equality $\arg (-\beta +\sqrt{-1}\omega)=\til{Z}{\beta}{\omega}(E_{\mr{tor}})$ holds if and only if 
$\tau_{1}^{L}E_{\mr{tor}}=0$. 
Thus, for an object $T \in \til{\mca T}{\beta}{\omega} \cap 	\til{\mca P}{\beta}{\omega}(\phi)$ with $\pi\phi \in (0, \arg (-\beta+\sqrt{-1}\omega)]$, we may assume that $\tau_{1}^{L}T$ is $\sigma_{1}$-free 
since $\arg \til{Z}{\beta}{\omega}(T) \leq \arg (-\beta +\sqrt{-1}\omega)$.

Taking the Harder-Narasimhan filtration of $T$ with respect to $\mu_{\beta, \omega}$-semistability, 
we have the following short exact sequence in $\gl{\mca A_{1}}{\mca A_{2}}$ 
\begin{equation}\label{eq:41}
\xymatrix{
0	\ar[r]	&	T^{+}	\ar[r]	&	T	\ar[r]	&	T^{-}	\ar[r]	&	0, 
}
\end{equation}
where $T^{+}$ and $T^{-}$ satisfy 
\begin{itemize}
\item the $(\sigma_{1}, \sigma_{2})$-free part $T^{+}_{\mr{fr}}$ of $T^{+}$ has $\mu_{\beta, \omega}^{-}(T^{+}_{\mr{fr}}) \geq 1$, and 
\item $T^{-}$ is $(\sigma_{1}, \sigma_{2})$-free with $0 < \mu_{\beta,\omega}^{-}(T^{-})\leq \mu_{\beta, \omega }^{+}(T^{-}) <1 $. 
\end{itemize}

Note that $T^{-}$ also satisfies $\arg \til{Z}{\beta}{\omega}(T^{-})>0$. 
By Lemma \ref{prop:argsup1}, there exists a $\theta $ in $(0, \pi)$ such that the inequalities
\begin{equation}\label{eq:46}
0	<	\arg \til{Z}{\beta}{\omega}(T^{-})	\leq \theta 
\end{equation}
holds for any $T$.

Now the sequence (\ref{eq:41}) gives an exact sequence not only in $\til{\mca T}{\beta}{\omega}$ but also in $\til{\mca A}{\beta}{\omega}$. 
Thus $T^{+}$ is a subobject of $T$ in $\til{\mca A}{\beta}{\omega}$ and $T^{+}$ satisfies 
\begin{equation}\label{eq:47}
0<	\phi^{-} (T^{+})	\leq \phi^{+}(T^{+})	\leq 	\phi^{+}(T)\leq 	\arg (-\beta +\sqrt{-1}\omega)
\end{equation}
with respect to the stability condition $\til{\Sigma}{\beta}{\omega}$.

By Proposition \ref{bunkai-improve}, 
there exists a constant $C_{1}>0$ which is independent of $T$ such that the inequality below holds: 
\begin{equation}\label{eq:48}
|\til{Z}{\beta}{\omega}(T)|	\geq C_{1}|\til{Z}{\beta}{\omega}(T^{+})|+C_{1}|\til{Z}{\beta}{\omega}(T^{-})|. 
\end{equation}
 Then, the inequality (\ref{eq:48}) implies the following: 
\begin{align}
\frac{|\til{Z}{\beta}{\omega}(\tau_{2}^{R}T)|}{|\til{Z}{\beta}{\omega}(T)|} 
&	\leq 
\frac{1}{C_{1}}
\left(
\frac{|\til{Z}{\beta}{\omega}(\tau_{2}^{R}T)|}{|\til{Z}{\beta}{\omega}(T^{+})|+|\til{Z}{\beta}{\omega}(T^{-})|} 
\right)	\\
	&	\leq 
	\frac{1}{C_{1}}
\left(
\frac{|\til{Z}{\beta}{\omega}(\tau_{2}^{R}T^{+})|}{|\til{Z}{\beta}{\omega}(T^{+})|+|\til{Z}{\beta}{\omega}(T^{-})|} 
+\frac{|\til{Z}{\beta}{\omega}(\tau_{2}^{R}T^{-})|}{|\til{Z}{\beta}{\omega}(T^{+})|+|\til{Z}{\beta}{\omega}(T^{-})|} 
\right)	\\
	&	\leq 
	\frac{1}{C_{1}}
\left(
\frac{|\til{Z}{\beta}{\omega}(\tau_{2}^{R}T^{+})|}{|\til{Z}{\beta}{\omega}(T^{+})|} 
+
\frac{|\til{Z}{\beta}{\omega}(\tau_{2}^{R}T^{-})|}{|\til{Z}{\beta}{\omega}(T^{-})|} 
\right). 
\end{align}

 By Proposition \ref{prop:argsup1}, 
there exists a positive constant $M_{1}>0$ such that 
the following holds: 
\begin{equation}
\frac{|\til{Z}{\beta}{\omega}(\tau_{2}^{R}T^{-})|}{|\til{Z}{\beta}{\omega}(T^{-})|} 
\leq M_{1}
\end{equation}

 Taking semiorthogonal decomposition of $T^{+}$, there exists a distinguished triangle below: 
 \begin{equation}\label{eq:49}
 \xymatrix{
 \tau_{2}^{R}T^{+}	\ar[r]	&	T^{+}	\ar[r]	&	\tau_{1}^{L}T^{+}. 
 }
 \end{equation}
Since $\tau_{1}^{L}T_{+}$ is a subobject of $\tau_{1}^{L}T$ in $\mca A_{1}$, 
$\tau_{1}^{L}T_{+}$ is $\sigma_{1}$-free. 
Then, by Lemma \ref{lem:kihon1}, $\tau_{2}^{R}T$ is in $\til{\mca T}{\beta}{\omega}$ and the inequality 
\begin{equation}
\label{eq:410}
\theta_{1}
\leq \arg \til{Z}{\beta}{\omega}(\tau_{2}^{R}T^{+})	
\end{equation}
holds where $\theta_{1}=\arg (\beta -1+\sqrt{-1}\omega)(\beta -\sqrt{-1}\omega)$. 
In parituclar the distinguished triangle (\ref{eq:49})
gives a short exact sequence not only in $\gl{\mca A_{1}}{\mca A_{2}}$ but also in $\til{\mca A}{\beta}{\omega}$. 
Hence $\tau_{2}^{R}T^{+}$ is a subobject of $T$ in $\til{\mca A}{\beta}{\omega}$, and the inequalities
\begin{equation}
\label{eq:411}
\arg \til{Z}{\beta}{\omega}(\tau_{2}^{R}T^{+})	\leq \pi\phi \leq \theta_{2}:=\arg (-\beta +\sqrt{-1}\omega)
\end{equation}
holds.

Now the closed interval $[\theta_{1}, \theta_{2}]$ is contained in $(0,\pi]$. 
Since $\arg \til{Z}{\beta}{\omega}(\tau_{1}^{L}T^{+})$ is in the interval $(0,\pi]$, 
there exists a $\theta \in [0,\pi)$ such that the inequalities
\begin{equation*}
0	\leq |\arg \til{Z}{\beta}{\omega}(\tau_{1}^{L}T^{+}) -\arg \til{Z}{\beta}{\omega}(\tau_{2}^{R}T^{+})|	\leq \theta
\end{equation*}
hold for any $T$. 
By Corollary \ref{bunkai-sup}, 
there exists a positive constant $M_{2}$ which is independent of the choice of $T$ and satisfies the following: 
\begin{equation}
\frac
{|\til{Z}{\beta}{\omega}(\tau_{2}^{R}T^{+})|}
{|\til{Z}{\beta}{\omega}(T^{+})|}	
=
\frac
{|\til{Z}{\beta}{\omega}(\tau_{2}^{R}T^{+})|}
{|\til{Z}{\beta}{\omega}(\tau_{1}^{L}T^{+})+\til{Z}{\beta}{\omega}(\tau_{2}^{R}T^{+})|}	
\leq M_{2}. 
\end{equation}
Thus we have finished the proof. 
\end{proof}

\begin{prop}\label{prop:F-bound}
Let $\mb D=\sod{\mb D_{1}}{\mb D_{2}}$ be a semiorthogonal decomposition with $\m2$. 
Choose rational stability conditions $\sigma _{i} =(\mca A_{i}, Z_{i}) \in \Stab{\mb D_{i}}$ such that $(\sigma_{1}, \sigma_{2})$ satisfies the condition $\m
5$. 
For an $\epsilon_{2}<0$ and set a subset 
$\mca  H^{-}(\epsilon_{2})$ of $\mca H$ by 
\begin{equation}
\mca H^{-}(\epsilon_{2})	:=
\left\{
(\beta , \omega) \in \mca H 	\middle|
\begin{split}
0 &< 2\beta +1-2\epsilon_{2}, \text{ and}	\\
0	&<	\omega^{2}+(\beta+1-2\epsilon_{2})^{2}+2\epsilon_{2}(1-2\epsilon_{2})
\end{split}
\right\}. 
\end{equation}
Suppose that $(\beta, \omega ) \in \mca H^{-}(\epsilon_{2})$. 
Then the supremum 
\begin{equation}
\sup
\left\{
\frac{|\til{Z}{\beta}{\omega}(\tau_{2}^{R}F)|}{|\til{Z}{\beta}{\omega}(F)|}
\middle|
F \in \til{\mca F}{\beta}{\omega}
\right\}
\end{equation}
is finite. 
\end{prop}

\begin{proof}
Choose $F \in \til{\mca F}{\beta}{\omega}$. 
Taking the Harder-Narasimhan filtration of $F$ with respect to $\mu_{\beta, \omega}$-semistability, 
we have the following short exact sequence in $\gl{\mca A_{1}}{\mca A_{2}}$ 
\begin{equation*}
\xymatrix{
0	\ar[r]	&	F^{+}	\ar[r]	&	F	\ar[r]	&	F^{-}	\ar[r]	&	0, 
}
\end{equation*}
where $F^{+}$ and $F^{-}$ satisfy 
\begin{itemize}
\item $F^{+}$ is $(\sigma_{1}, \sigma_{2})$-free with $\epsilon_{2} < \mu_{\beta, \omega}^{-}(F^{+}) \leq \mu_{\beta, \omega}^{+}(F^{+})	\leq 0$, and 
\item $F^{-}$ is $(\sigma_{1}, \sigma_{2})$-free with $\mu_{\beta, \omega}^{+}(F^{-})\leq \epsilon_{2}$. 
\end{itemize}
Since $\mu_{\beta, \omega}^{+}(F^{+})\leq 0$, applying Lemma \ref{lem:slope-mor} to $F^{+}$, 
there exists a $\theta_{+} \in (-\pi,0)$ such that the inequalities hold: 
\begin{equation}\label{eq:f+}
\theta_{+}	<	\arg\til{Z}{\beta}{\omega}(F^{+})	\leq 0. 
\end{equation}
Moreover, applying Corollary \ref{cor:mu<-1/2} to $F^{-}$, there exists a $\theta_{-}$ in the open interval $(-\pi, 0)$ such that the following holds: 
\begin{equation}\label{eq:f-}
\theta_{-}\leq \til{Z}{\beta}{\omega}(F^{-})	\leq 0. 
\end{equation}
Thus one can use Proposition \ref{bunkai-improve} by (\ref{eq:f+}) and (\ref{eq:f-}) and there exists a constant $C_{1}>0$ such that 
the following holds: 
\begin{equation}\label{eq:425}
|\til{Z}{\beta}{\omega}(F)|	\geq C_{1} |\til{Z}{\beta}{\omega}(F^{+})|+C_{1} |\til{Z}{\beta}{\omega}(F^{-})|. 
\end{equation}

Similarly to the proof of Proposition \ref{prop:free-angle}, we see 
\begin{align*}
\frac
{|\til{Z}{\beta}{\omega}(\tau_{2}^{R}F)|}
{|\til{Z}{\beta}{\omega}(F)|}
&\leq 
\frac{1}{C_{1}}
\left(
\frac
{|\til{Z}{\beta}{\omega}(\tau_{2}^{R}F)|}
{|\til{Z}{\beta}{\omega}(F^{+})|+|\til{Z}{\beta}{\omega}(F^{-})|}
\right)	\\
	&\leq 
\frac{1}{C_{1}}
\left(
\frac
{|\til{Z}{\beta}{\omega}(\tau_{2}^{R}F^{+})|}
{|\til{Z}{\beta}{\omega}(F^{+})|+|\til{Z}{\beta}{\omega}(F^{-})|}
+
\frac
{|\til{Z}{\beta}{\omega}(\tau_{2}^{R}F^{-})|}
{|\til{Z}{\beta}{\omega}(F^{+})|+|\til{Z}{\beta}{\omega}(F^{-})|}
\right)	\\
	&\leq 
\frac{1}{C_{1}}
\left(
\frac
{|\til{Z}{\beta}{\omega}(\tau_{2}^{R}F^{+})|}
{|\til{Z}{\beta}{\omega}(F^{+})|}
+
\frac
{|\til{Z}{\beta}{\omega}(\tau_{2}^{R}F^{-})|}
{|\til{Z}{\beta}{\omega}(F^{-})|}
\right)	\\
\end{align*}

Now, by Lemma \ref{lem:slope-mor} and Corollary \ref{cor:mu<-1/2}, there exist 
upper bounds for ${|\til{Z}{\beta}{\omega}(\tau_{2}^{R}F^{+})|}/{|\til{Z}{\beta}{\omega}(F^{+})|}$ and 
${|\til{Z}{\beta}{\omega}(\tau_{2}^{R}F^{-})|}/{|\til{Z}{\beta}{\omega}(F^{-})|}$. 
 This gives the proof. 
 \end{proof}

\begin{prop}\label{prop:boundofT}
Let $\mb D=\sod{\mb D_{1}}{\mb D_{2}}$ be a semiorthogonal decomposition with $\m2$. 
Choose rational stability conditions $\sigma _{i} =(\mca A_{i}, Z_{i}) \in \Stab{\mb D_{i}}$ such that $(\sigma_{1}, \sigma_{2})$ satisfies the condition $\m
5$. 
Take $(\beta, \omega)\in \mca H^{+}(\epsilon_{1})$ for an $\epsilon_{1} \in (0,1/2)$. 

Set $\mca T_{\beta, \omega}'$ by 
\begin{equation}
\mca T_{\beta, \omega}'	
=
\til{\mca T}{\beta}{\omega}	\cap	
\til{\mca P}{\beta}{\omega}\left[\frac{1}{\pi}\arg (-\beta +\sqrt{-1}\omega), 1	\right]
\end{equation}
Then the supremum 
\begin{equation}
\sup
\left\{
\frac
{|\til{Z}{\beta}{\omega}(\tau_{2}^{R}T)|}
{|\til{Z}{\beta}{\omega}(T)|}
\middle|
T \in \mca T_{\beta, \omega}'	
\right\}
\end{equation}
is finite. 
\end{prop}

\begin{proof}
Taking the Harder-Narasimhan filtration of $T \in \mca {T}_{\beta, \omega}'$ with respect to $\mu_{\beta, \omega}$-semistability, 
there exists a short exact sequence in $\gl{\mca A_{1}}{\mca A_{2}}$
\[
\xymatrix{
0	\ar[r]	&	T^{>1}	\ar[r]	&	T	\ar[r]	&	T^{\leq 1}	\ar[r]	&	0
}
\]
such that 
\begin{itemize}
\item the $(\sigma_{1}, \sigma_{2})$-free part $T^{>1}_{\mr{fr}}$ of $T^{>1}$ satisfies $\mu_{\beta, \omega}^{-}(T^{>1}_{\mr{fr}}) >1$, and 
\item $T^{\leq 1}$ is $(\sigma_{1}, \sigma_{2})$-free with $0 < \mu_{\beta, \omega}^{-}(T^{\leq 1})\leq \mu_{\beta, \omega}^{+}(T^{\leq 1})\leq 1$. 
\end{itemize}
Then Lemma \ref{lem:mu>1} implies 
\begin{equation} 
\arg (\beta-1+\sqrt{-1}\omega)(\beta -\sqrt{-1}\omega)
\leq 
\arg \til{Z}{\beta}{\omega}(T^{>1}). 
\end{equation}
Moreover the canonical morphism $T \to T^{\leq 1}$ is an epimorphism not only in $\til{\mca T}{\beta}{\omega}$ but also in 
$\til{\mca A}{\beta}{\omega}$. 
Hence the quotient $T^{\leq 1}$ in $\til{\mca A}{\beta}{\omega}$ also satisfies 
\begin{equation}
\arg(-\beta +\sqrt{-1}\omega) \leq \til{Z}{\beta}{\omega}(T^{\leq 1}). 
\end{equation}
Thus, by Proposition \ref{bunkai-improve}, 
there exists a constant $C_{1}$ such that 
the following holds: 
\begin{equation}
|\til{Z}{\beta}{\omega}(T)|	\geq C_{1}\left(	|\til{Z}{\beta}{\omega}(T^{>1})|+|\til{Z}{\beta}{\omega}(T^{\leq 1})|	\right)
\end{equation}

Taking the Harder-Narasimhan filtration of $T^{\leq 1}$, we have the short exact sequence in $\gl{\mca A_{1}}{\mca A_{2}}$
\[
\xymatrix{
0	\ar[r]	&	T^{=1}	\ar[r]	&	T^{\leq 1}	\ar[r]	&	T^{<1}	\ar[r]	&	0, 
}
\]
where $T^{=1}$ and $T^{<1}$ are 
\begin{itemize}
\item $T^{=1}$ is $(\sigma_{1}, \sigma_{2})$-free and is $\mu_{\beta, \omega}$-semistable with $\mu_{\beta, \omega}(T^{=1})=1$, and 
\item $T^{<1}$ is $(\sigma_{1}, \sigma_{2})$-free with $\mu_{\beta, \omega}^{+}(T^{<1})<1$. 
\end{itemize}
Note that the inequality $\mu_{\beta , \omega}^{-}(T^{<1})>0$ also holds. 
Then the canonical morphism $T \to T^{<1}$ gives an epimorphism in $\til{\mca A}{\beta}{\omega}$. 
Thus $T^{<1}$ satisfies 
$\arg (-\beta +\sqrt{-1}\omega) < \arg \til{Z}{\beta}{\omega}(T^{<1})$. 
By Proposition \ref{prop:argsup1}, there exists a $\theta _{1} \in (0,\pi)$ such that the inequality  
$\arg \til{Z}{\beta}{\omega}(T^{<1})	\leq \theta_{1}$ holds. 
Thus we obtain 
\begin{equation}
 	\arg (-\beta +\sqrt{-1}\omega) < \arg \til{Z}{\beta}{\omega}(T^{<1})	 \leq \theta _{1}
\end{equation}
and we see that there exists a $\theta \in [0,\pi)$ such that $|\arg \til{Z}{\beta}{\omega}(T^{=1})-\arg \til{Z}{\beta}{\omega}(T^{<1})| < \theta$. 
By Proposition \ref{bunkai-improve}, 
there exists a positive constant $C_{2}$ such that the following holds: 
\begin{equation}
|\til{Z}{\beta}{\omega}(T^{\leq 1})| \geq C_{2}\left(	|\til{Z}{\beta}{\omega}(T^{=1})|+|\til{Z}{\beta}{\omega}(T^{<1})|	\right). 
\end{equation}
Thus we see 
\begin{align}
\notag \frac
{|\til{Z}{\beta}{\omega}(\tau_{2}^{R}T)|}
{|\til{Z}{\beta}{\omega}(T)|}	
&	\leq 
\frac{1}{C_{1}}
\left(
\frac{|\til{Z}{\beta}{\omega}(\tau_{2}^{R}T)|}
{|\til{Z}{\beta}{\omega}(T^{>1})|+|\til{Z}{\beta}{\omega}(T^{\leq 1})|}	
\right)	\\
\notag &	\leq 
\frac{1}{C_{1}}
\left(
\frac{|\til{Z}{\beta}{\omega}(\tau_{2}^{R}T^{>1})|+|\til{Z}{\beta}{\omega}(\tau_{2}^{R}T^{\leq 1})|}
{|\til{Z}{\beta}{\omega}(T^{>1})|+|\til{Z}{\beta}{\omega}(T^{\leq 1})|}	
\right)	\\
\notag &\leq 
\frac{1}{C_{1}}
\left(
\frac{|\til{Z}{\beta}{\omega}(\tau_{2}^{R}T^{>1})|}
{|\til{Z}{\beta}{\omega}(T^{>1})|}	
+
\frac{|\til{Z}{\beta}{\omega}(\tau_{2}^{R}T^{\leq 1})|}
{|\til{Z}{\beta}{\omega}(T^{\leq 1})|}	
\right)	
\\
\notag &\leq 
\frac{1}{C_{1}}
\left(
\frac{|\til{Z}{\beta}{\omega}(\tau_{2}^{R}T^{>1})|}
{|\til{Z}{\beta}{\omega}(T^{>1})|}	
+
\frac{|\til{Z}{\beta}{\omega}(\tau_{2}^{R}T^{\leq 1})|}
{C_{2}|\til{Z}{\beta}{\omega}(T^{=1})|+C_{2}|\til{Z}{\beta}{\omega}(T^{<1})|}	
\right)	
\\
\notag &\leq 
\frac{1}{C_{1}}
\left(
\frac{|\til{Z}{\beta}{\omega}(\tau_{2}^{R}T^{>1})|}
{|\til{Z}{\beta}{\omega}(T^{>1})|}	
+
\frac{|\til{Z}{\beta}{\omega}(\tau_{2}^{R}T^{=1})|+|\til{Z}{\beta}{\omega}(\tau_{2}^{R}T^{<1})|}
{C_{2}|\til{Z}{\beta}{\omega}(T^{=1})|+C_{2}|\til{Z}{\beta}{\omega}(T^{<1})|}	
\right)	
\\
\label{eq:last}&\leq 
\frac{1}{C_{1}}
\left(
\frac{|\til{Z}{\beta}{\omega}(\tau_{2}^{R}T^{>1})|}
{|\til{Z}{\beta}{\omega}(T^{>1})|}	
+
\frac
{|\til{Z}{\beta}{\omega}(\tau_{2}^{R}T^{=1})|}
{C_{2}|\til{Z}{\beta}{\omega}(T^{=1})|}	
+
\frac{|\til{Z}{\beta}{\omega}(\tau_{2}^{R}T^{<1})|}
{C_{2}|\til{Z}{\beta}{\omega}(T^{<1})|}	
\right)	
\end{align}

By Lemma \ref{lem:mu>1}, there exists an upper bound for  
${|\til{Z}{\beta}{\omega}(\tau_{2}^{R}T^{>1})|}/{|\til{Z}{\beta}{\omega}(T^{>1})|}	$. 
The assumption $\mu_{\beta, \omega}(T^{=1})=1$ implies $\arg \til{Z}{\beta}{\omega}(\tau_{2}^{R}T^{=1})=\arg (\beta-1+\sqrt{-1}\omega)(\beta- \sqrt{-1}\omega)$ and hence 
the difference 
$|\arg \til{Z}{\beta}{\omega}(\tau_{1}^{L}T^{=1}) -\arg \til{Z}{\beta}{\omega}(\tau_{2}^{R}T^{=1})|$
is strictly smaller than $\pi$. 
Hence there exists an upper bound for ${|\til{Z}{\beta}{\omega}(\tau_{2}^{R}T^{=1})|}/{C_{2}|\til{Z}{\beta}{\omega}(T^{=1})|}	$ by Corollary \ref{bunkai-sup}. 
Finally, by Lemma \ref{lem:slope-mor}, there exists an upper bound for 
${|\til{Z}{\beta}{\omega}(\tau_{2}^{R}T^{<1})|}/{|\til{Z}{\beta}{\omega}(T^{<1})|}	$. 
Hence (\ref{eq:last}) is bounded above. 
\end{proof}

\begin{lem}\label{lem:hamigaki}
Let $\mb D=\sod{\mb D_{1}}{\mb D_{2}}$ be a semiorthogonal decomposition with $\m2$. 
Choose rational stability conditions $\sigma _{i} =(\mca A_{i}, Z_{i}) \in \Stab{\mb D_{i}}$ such that $(\sigma_{1}, \sigma_{2})$ satisfies the condition $\m
5$. 
If $(\beta, \omega) $ is in $\mca H^{-}(\epsilon_{2})$
then the inequality 
\[
\arg (-\beta +\sqrt{-1}\omega) < \til{Z}{\beta}{\omega}(F[1])
\]
holds for any $F \in \til{\mca F}{\beta}{\omega}$. 
\end{lem}

\begin{proof}
Taking the Harder-Narashimhan filtration of $F \in \til{\mca F}{\beta}{\omega}$, 
we obtain the short exact sequence 
\begin{equation*}
\xymatrix{
0	\ar[r]	&	F^{+}	\ar[r]	&	F	\ar[r]	&	F^{-}	\ar[r]	&	0	, 
}
\end{equation*}
where $F^{+}$ and $F^{-}$ satisfy 
\begin{itemize}
\item $F^{+}$ is $(\sigma_{1}, \sigma_{2})$-free with $\epsilon_{2} < \mu_{\beta, \omega}^{-}(F^{+})\leq \mu_{\beta, \omega}^{+}(F^{+})\leq 0$, and 
\item $F^{-}$ is $(\sigma_{1}, \sigma_{2})$-free with $\mu_{\beta, \omega}^{+}(F^{-})\leq \epsilon_{2}$. 
\end{itemize}
In particular we have 
$\arg \til{Z}{\beta}{\omega}(F^{\pm}) \leq 0$.

Since  $\epsilon_{2}$ is negative, Lemma \ref{lem:slope-mor} implies 
\begin{equation}\label{eq:yohukashi}
\arg (\beta +1-2\epsilon_{2}+\sqrt{-1}\omega)(\beta -\sqrt{-1}\omega)	\leq \arg \til{Z}{\beta}{\omega}(F^{+}) \leq 0. 
\end{equation}
In addition Corollary \ref{cor:mu<-1/2} implies 
\begin{equation}\label{eq:yohukashi2}
\arg (\beta -\sqrt{-1}\omega)	< \arg \til{Z}{\beta}{\omega}(F^{-})\leq 0 . 
\end{equation}
Since the inequality $ \arg (\beta -\sqrt{-1}\omega) < \arg (\beta +1-2\epsilon_{2}+\sqrt{-1}\omega)(\beta -\sqrt{-1}\omega)$ clearly holds, 
we obtain
\begin{equation*}
\arg (\beta -\sqrt{-1}\omega)	<	\til{Z}{\beta}{\omega}(F)	\leq 0
\end{equation*}
which gives the proof. 
\end{proof}

\begin{prop}\label{prop:keyprop}
Let $\mb D=\sod{\mb D_{1}}{\mb D_{2}}$ be a semiorthogonal decomposition with $\m2$. 
Choose rational stability conditions $\sigma _{i} =(\mca A_{i}, Z_{i}) \in \Stab{\mb D_{i}}$ such that $(\sigma_{1}, \sigma_{2})$ satisfies the condition $\m
5$. 
Suppose that $(\beta, \omega) $ is in $ \mca H^{-}(\epsilon_{2})$, 

Let $E \in \til{\mca A}{\beta}{\omega}$ be $\til{\Sigma}{\beta}{\omega}$-semistable with the phase $\phi$. 
If $\pi \phi \leq  \arg (-\beta +\sqrt{-1}\omega)$, then $E$ is in $\til{\mca T}{\beta}{\omega}$. 
In particular, 
$\til{\mca P}{\beta}{\omega}(0, \arg (-\beta+\sqrt{-1}\omega)/\pi]$ is the full subcategory of $\til{\mca T}{\beta}{\omega}$. 
\end{prop}

\begin{proof}
Since the pair $(\til{\mca T}{\beta}{\omega}, \til{\mca F}{\beta}{\omega}[1])$ is a torsion pair on 
$\til{\mca A}{\beta}{\omega}$, there is an short exact sequence 
\begin{equation*}
\xymatrix{
0	\ar[r]	&	F[1]	\ar[r]	&	E	\ar[r]	&	T	\ar[r]	&	0, 
}
\end{equation*}
where $T \in \til{\mca T}{\beta}{\omega}$ and $F \in \til{\mca F}{\beta}{\omega}$. 
By the $\til{\Sigma}{\beta}{\omega}$-semistability of $E$, we have 
\begin{equation}\label{lem:ellie}
\arg \til{Z}{\beta}{\omega}(F[1])	\leq \arg \til{Z}{\beta}{\omega}(E) \leq \arg (-\beta +\sqrt{-1}\omega). 
\end{equation}
Then the inequalities (\ref{lem:ellie}) contradict Lemma \ref{lem:hamigaki}. 
Hence $F[1]$ has to be zero. 
The last assertion follows from the definition of $\til{\mca P}{\beta}{\omega}(0, \arg (-\beta+\sqrt{-1}\omega)/\pi]$. 
\end{proof}

\begin{prop}\label{prop:muzukasi}
Let $\mb D=\sod{\mb D_{1}}{\mb D_{2}}$ be a semiorthogonal decomposition with $\m2$. 
Choose rational stability conditions $\sigma _{i} =(\mca A_{i}, Z_{i}) \in \Stab{\mb D_{i}}$ such that $(\sigma_{1}, \sigma_{2})$ satisfies the condition $\m
5$. 

Then the supremum
\begin{equation}
\sup
\left\{
\frac{|\til{Z}{\beta}{\omega} (\tau_{2}^{R}E)|}{|\til{Z}{\beta}{\omega} (E) |}
\middle|
E \in \til{\mca P}{\beta}{\omega}(\phi) \text{ where } \pi\phi \in \left(\arg (-\beta +\sqrt{-1}\omega), \pi \right]
\right\}
\end{equation}
is finite if $(\beta, \omega) $ is in $\mca H^{+}(\epsilon_{1}) \cap \mca H^{-}(\epsilon_{2})$. 
\end{prop}

\begin{proof}
Let $E$ be in the slicing $\til{\mca P}{\beta}{\omega}(\phi)$ where $\pi \phi $ is in $\left(\arg (-\beta +\sqrt{-1}\omega), \pi \right]$. 
Since $E \in \til{\mca A}{\beta}{\omega}$, there exists a canonical distinguished triangle 
\[
\xymatrix{
E^{-1}[1]	\ar[r]	&	E	\ar[r]	&	E^{0}
}
\]
where $E^{0} \in \til{\mca T}{\beta}{\omega}$ and $E^{-1} \in \til{\mca F}{\beta}{\omega}$. 
Then $E^{-1}$ satisfies $\arg \til{Z}{\beta}{\omega}(E^{-1}[1])	>	\arg (-\beta +\sqrt{-1}\omega)$ by Lemma \ref{lem:hamigaki}. 
Since the morphism $E \to E^{0}$ is an epimorphism in $\til{\mca A}{\beta}{\omega}$, we see 
\begin{equation}
\phi(E^{0})	\geq \phi^{-}(E^{0})	\geq \phi ^{-}(E)=\phi(E)	\geq \frac{1}{\pi}\arg (-\beta +\sqrt{-1}\omega). 
\end{equation}
Hence inequalities $0	\leq |	\arg \til{Z}{\beta}{\omega}(E^{0})	-	\til{Z}{\beta}{\omega}(E^{-1}[1])|  < \pi $ hold and there exists a positive constant $C_{1}$ such that the inequality 
\begin{equation}\label{eq:431}
|\til{Z}{\beta}{\omega}(E)|	\geq C_{1}|\til{Z}{\beta}{\omega}(E^{0})| +C_{1}|\til{Z}{\beta}{\omega}(E^{-1}[1])|
\end{equation}
holds for any $E$. 
Then we see
\begin{align*}
\frac
{|\til{Z}{\beta}{\omega}(\tau_{2}^{R}E)|}
{|\til{Z}{\beta}{\omega}(E)|}
&	\leq 
\frac{1}{C_{1}}
\left(
\frac
{|\til{Z}{\beta}{\omega}(\tau_{2}^{R}E)|}
{|\til{Z}{\beta}{\omega}(E^{0})|+|\til{Z}{\beta}{\omega}(E^{-1}[1])|}
\right)	\\
&	\leq 
\frac{1}{C_{1}}
\left(
\frac
{|\til{Z}{\beta}{\omega}(\tau_{2}^{R}E^{0})|+|\til{Z}{\beta}{\omega}(\tau_{2}^{R}(E^{-1}[1]))|}
{|\til{Z}{\beta}{\omega}(E^{0})|+|\til{Z}{\beta}{\omega}(E^{-1}[1])|}
\right)	\\
&	= 
\frac{1}{C_{1}}
\left(
\frac
{|\til{Z}{\beta}{\omega}(\tau_{2}^{R}E^{0})|}
{|\til{Z}{\beta}{\omega}(E^{0})|+|\til{Z}{\beta}{\omega}(E^{-1}[1])|}
+
\frac
{|\til{Z}{\beta}{\omega}(\tau_{2}^{R}(E^{-1}[1]))|}
{|\til{Z}{\beta}{\omega}(E^{0})|+|\til{Z}{\beta}{\omega}(E^{-1}[1])|}
\right)	\\
&	\leq
\frac{1}{C_{1}}
\left(
\frac
{|\til{Z}{\beta}{\omega}(\tau_{2}^{R}E^{0})|}
{|\til{Z}{\beta}{\omega}(E^{0})|}
+
\frac
{|\til{Z}{\beta}{\omega}(\tau_{2}^{R}(E^{-1}[1]))|}
{|\til{Z}{\beta}{\omega}(E^{-1}[1])|}
\right). 
\end{align*}
Since $E^{-1}$ is in $\til{\mca F}{\beta}{\omega}$,  Proposition \ref{prop:F-bound} implies that 
${|\til{Z}{\beta}{\omega}(\tau_{2}^{R}(E^{-1}[1]))|}/{|\til{Z}{\beta}{\omega}(E^{-1}[1])|}$ is bounded above. 
Moreover, by Proposition \ref{prop:boundofT}, 
${|\til{Z}{\beta}{\omega}(\tau_{2}^{R}(E^{0}))|}/{|\til{Z}{\beta}{\omega}(E^{0})|}$ is bounded above. 
\end{proof}

\begin{prop}\label{prop:fake}
Let $\mb D=\sod{\mb D_{1}}{\mb D_{2}}$ be a semiorthogonal decomposition with $\m2$. 
Choose rational stability conditions $\sigma _{i} =(\mca A_{i}, Z_{i}) \in \Stab{\mb D_{i}}$ such that $(\sigma_{1}, \sigma_{2})$ satisfies the condition $\m
5$. 

If $(\beta, \omega)$ be in $\left( \mca H^{+}(\epsilon_{1}) \cap \mca H^{-}(\epsilon_{2}) \right)_{\bb Q}$ then the stability condition $\til{\Sigma}{\beta}{\omega}$ satisfies 
that the supremum 
\begin{equation}
\sup
\left\{
\frac{|\til{Z}{\beta}{\omega}(\tau_{2}^{R}E)|}{|\til{Z}{\beta}{\omega}(E)|}
\middle|
\text{$E$ is $\til{\Sigma}{\beta}{\omega}$-semistable}
\right\}
\end{equation}
is finite. 
\end{prop}

\begin{proof}
Let $E$ be a $\til{\Sigma}{\beta}{\omega}$-semistable object with the phase $\phi$. 
Without loss of generality, we may assume that $\phi \in (0,1]$. 

If $\pi \phi \leq \arg (-\beta +\sqrt{-1}\omega)$, then $E$ is in $\til{\mca T}{\beta}{\omega}$ by Proposition \ref{prop:keyprop}. 
Then, by Proposition \ref{prop:free-angle}, 
there exists an $M_{1}>0$ such that 
\begin{equation}
\frac{|\til{Z}{\beta}{\omega}(\tau_{2}^{R}E)|}{|\til{Z}{\beta}{\omega}(E)|}	<M_{1}. 
\end{equation}

Suppose $\pi \phi >  \arg (-\beta +\sqrt{-1}\omega)$. 
By Proposition \ref{prop:muzukasi}, we see the supremum
\begin{equation}
\sup
\left\{
\frac{|\til{Z}{\beta}{\omega}(\tau_{2}^{R}E)|}{|\til{Z}{\beta}{\omega}(E)|}
\middle|
\text{$E$ is $\til{\Sigma}{\beta}{\omega}$-semistable with }\pi \phi >  \arg (-\beta +\sqrt{-1}\omega)
\right\}
\end{equation}
is bounded. 
Hence we have finished the desired assertion. 
\end{proof}

\begin{cor}\label{cor:fakesupport}
Let $\mb D=\sod{\mb D_{1}}{\mb D_{2}}$ be a semiorthogonal decomposition with $\m2$. 
Choose rational stability conditions $\sigma _{i} =(\mca A_{i}, Z_{i}) \in \Stab{\mb D_{i}}$ such that $(\sigma_{1}, \sigma_{2})$ satisfies the condition $\m
5$. 

If a rational points $(\beta, \omega)$ is in $\mca H^{+}(\epsilon_{1}) \cap \mca H^{-}(\epsilon_{2})$, then 
the supremum 
\begin{equation}
\sup
\left\{
\frac{|Z_{2}(\tau_{2}^{R}E)|}{|\til{Z}{\beta}{\omega}(E)|}
\middle|
\text{$E$ is $\til{\Sigma}{\beta}{\omega}$-semistable}
\right\}
\end{equation}
is finite. 
\end{cor}

\begin{proof}
By the definition, we see $|\til{Z}{\beta}{\omega}(\tau_{2}^{R}E)|=|\beta-\sqrt{-1}\omega|\cdot |Z_{2}(\tau_{2}^{R}E)|$ for any $E \in \mb D$. 
Thus Proposition \ref{prop:fake} implies the desired assertion. 
\end{proof}

Now we will show that the following map 
\begin{equation}
s_{\bb Q}	\colon \mca H^{+}(\epsilon_{1})_{\bb Q} \cap \mca H^{-}(\epsilon_{2})_{\bb Q} \to \Stab{\mb D}, 
s_{\bb Q}(\beta, \omega) =\til{\Sigma}{\beta}{\omega}
\end{equation}
is continuous.

\begin{lem}\label{lem:heart}
Let $\mb D=\sod{\mb D_{1}}{\mb D_{2}}$ be a semiorthogonal decomposition with $\m2$. 
Suppose rational stability conditions $\sigma_{i}=(\mca A_{i}, Z_{i}) \in \Stab{\mb D_{i}}$ satisfies $\m5$.  
Take $E \in \til{\mca A}{\beta}{\omega}$ and $E'\in \til{\mca A}{\beta'}{\omega'}$ where $(\beta, \omega)$ and $(\beta', \omega ') \in \bb R \times \bb R_{>0}$.  
Then the vanishing $\Hom_{\mb D}(E[p], E')=0$ holds for any $p \geq 2$.  
\end{lem}

\begin{proof}
Let $\mca P^{\sigma_{i}}$ be the slicing of $\sigma_{i}$. 
By the construction of the heart $\til{\mca A}{\beta}{\omega}$, 
if  $E $ is in $\til{\mca A}{\beta}{\omega}$ then 
$\tau_{1}^{L}E \in \mca P^{\sigma_{1}}(0,2]$ and $\tau_{2}^{R}E \in \mca P^{\sigma_{2}}(0,2]$. 
Since $\sigma_{i}$ satisfies $\m3$, we see $\Phi(\tau_{2}^{R}E) \in \mca P^{\sigma_{1}}(0,2]$. 
Thus the vanishings
\begin{equation}
\begin{cases}
\Hom_{\mb D_{1}}(\tau_{1}^{L}E[p], \tau_{1}^{L}E')=0	&	\\
\Hom_{\mb D_{2}}(\tau_{2}^{R}E[p], \tau_{2}^{R}E')=0	&	\\
\Hom_{\mb D_{1}}(\tau_{1}^{L}E[p], \Phi(\tau_{2}^{R}E')[-1])=0
\end{cases}
\end{equation}
hold. 
Then Lemma \ref{lem:morvanishing} implies the desired assertion. 
\end{proof}

\begin{prop}\label{prop:t-str-cont}
Let $\mb D=\sod{\mb D_{1}}{\mb D_{2}}$ be a semiorthogonal decomposition with $\m2$. 
Choose rational stability conditions $\sigma _{i} =(\mca A_{i}, Z_{i}) \in \Stab{\mb D_{i}}$ such that $(\sigma_{1}, \sigma_{2})$ satisfies the condition $\m
5$. 
Take a $(\beta, \omega)  \in (\mca H^{+}(\epsilon_{1}) \cap \mca H^{-}(\epsilon_{2}))_{\bb Q}$ for an $\epsilon_{1}\in (0, 1/2)$ and an $\epsilon _{2} <0$. 

For an arbitrary $\varepsilon \in (0, \varepsilon_{0})$, 
choose $(\beta', \omega ') \in (\mca H^{+}(\epsilon_{1}) \cap \mca H^{-}(\epsilon_{2}) )_{\bb Q} $ so that the following holds: 
\begin{equation}\label{eq:distance}
\sup
\{
|\mu_{\beta', \omega'}^{+}(E)-\mu_{\beta, \omega }^{+}(E)|
, 
|\mu_{\beta', \omega'}^{-}(E)-\mu_{\beta, \omega }^{-}(E)|
\mid	
E \text{ is } (\sigma_{1}, \sigma_{2})\text{-free}
\} < \varepsilon. 
\end{equation}

\begin{enumerate}
\item Let $E \in \til{\mca A}{\beta}{\omega}$ be $\til{\Sigma}{\beta}{\omega}$-semistable with 
\begin{equation*}
\arg \til{Z}{\beta}{\omega}(E) > 
\arg (\beta + 1 +\sqrt{-1}\omega)(\beta -\sqrt{-1}\omega)+\pi
\end{equation*}
Then the vanishing $\Hom_{\mb D}(E[1], E')=0$ holds for any $E' \in \til{\mca A}{\beta'}{\omega'}$. 
\item Let $E \in \til{\mca A}{\beta}{\omega}$ be $\til{\Sigma}{\beta}{\omega}$-semistable with 
\begin{equation*}
\arg \til{Z}{\beta}{\omega}(E[-1]) \leq \arg (\beta+1 + 2\epsilon_{0}+\sqrt{-1}\omega)(\beta - \sqrt{-1}\omega). 
\end{equation*}
Then the vanishing $\Hom_{\mb D}(E', E[-1])=0$ holds for any $E' \in \til{\mca A}{\beta'}{\omega'}$. 
\end{enumerate}
\end{prop}

\begin{proof}
Throughout the proof, 
let $T$ (resp. $T'$) be the free part of $E$ (resp. $E'$) and let $F[1]$ (resp. $F'[1]$) be the torsion part of $E$ (resp. $E'$) with respect to 
the torsion pair $(\til{\mca F}{\beta}{\omega}[1], \til{\mca T}{\beta}{\omega}) $ on the heart 
$\til{\mca A}{\beta}{\omega}$ (resp. on $\til{\mca A}{\beta'}{\omega'}$). 
Then one easily see 
$\Hom_{\mb D}(E[1], E') \cong \Hom_{\mb D}(T, F')$. 

Basically the proof of the assertions are similar. 
Since $F'$ is in $\til{\mca F}{\beta'}{\omega'}$, 
we have $\mu_{\beta, \omega}^{+}(F') \leq \epsilon$ by (\ref{eq:distance}). 
By truncating the Harder-Narasimhan filtration of $F'$ with respect to the $\mu_{\beta, \omega}$-semistability, 
there exists a subobject $F'_{+}$ of $F'$ in $\gl{\mca A_{1}}{\mca A_{2}}$ such that 
\begin{itemize}
\item $F'_{+}$ is $(\sigma_{1}, \sigma_{2})$-free with $0 < \mu_{\beta, \omega}^{-}(F'_{+}) \leq \mu_{\beta, \omega}^{+}(F'_{+}) \leq \epsilon$, 
\item the quotient $F'_{-}:=F'/F'_{+}$ in $\gl{\mca A_{1}}{\mca A_{2}}$ is $(\sigma_{1}, \sigma_{2})$-free with 
$\mu_{\beta, \omega}^{+} (F'_{-})\leq 0$. 
\end{itemize}
Then any morphism $\varphi \colon T \to F' $ factors through $F'_{+}$, that is, 
we have the following diagram: 
\[
\xymatrix{
T	\ar[r]^-{\varphi}	\ar[rd]_-{\varphi'}	& F'. 	\\
		&	F'_{+}\ar[u]
}
\]

Note that $\til{\mca T}{\beta}{\omega}$ is not only the free part of $\til{\mca A}{\beta}{\omega}$ but also the torsion part of $\gl{\mca A_{1}}{\mca A_{2}}$. 
Since the morphism $\varphi' \colon T \to F'_{+}$ above lives in $\til{\mca T}{\beta}{\omega}$ which is closed under subobjects in $\til{\mca A}{\beta}{\omega}$, 
the image $\im \varphi'$ of $\varphi'$ in $\til{\mca A}{\beta}{\omega}$ belongs to $\til{\mca T}{\beta}{\omega}$. 

Again by truncating the Harder-Narasimhan filtration of $\im \varphi'$ with respect to the $\mu_{\beta,\omega}$-semistability, 
we obtain a subobject $G$ of $\im \varphi'$ in the abelian category $\gl{\mca A_{1}}{\mca A_{2}}$ such that 
\begin{itemize}
\item the $(\sigma_{1}, \sigma_{2})$-free part $G_{\mr {fr}}$ of $G_{+}$ satisfies $\mu_{\beta, \omega}^{-}(G_{\mr {fr}}) >\epsilon$, and 
\item the quotient $G_{-}:=\im \varphi'/G_{+}$ is $(\sigma_{1}, \sigma_{2})$-free with 
$0  < \mu_{\beta , \omega}^{-}(G_{-}) \leq \mu_{\beta, \omega}^{+}(G_{-}) \leq \epsilon$. 
\end{itemize}
Thus we have the diagram: 
\[
\xymatrix{
	&	T\ar[d]	&	\\
G_{+}	\ar[r]	&	\im \varphi'	\ar[r]\ar[d]	&	G_{-}. 	\ar@{-->}[ld]	\\
&F'_{+}&
}
\]
The vanishing $\Hom_{\mb D}(G_{+}, F'_{+}) = 0$ gives the lift of $\im \varphi' \to F'_{+}$ to $G^{-}$. 
The horizontal line gives a short exact sequence not only in $\gl{\mca A_{1}}{\mca A_{2}}$ but also in $\til{\mca A}{\beta}{\omega}$ 
since all objects are in $\til{\mca T}{\beta}{\omega}$. 
Hence the composite $E \to T \to \im \varphi' \to G_{-}$ gives an epi morphism in $\til{\mca A}{\beta}{\omega}$, 
in particular $G_{-}$ is a quotient of $E$. 
So the argument $\arg \til{Z}{\beta}{\omega}(G^{-})$ should satisfy
\begin{equation}\label{eq:contradiction}
\arg \til{Z}{\beta}{\omega}(G^{-}) \geq \arg \til{Z}{\beta}{\omega}(E)  \geq \arg (\beta + 1 +\sqrt{-1}\omega)(\beta -\sqrt{-1}\omega)+\pi. 
\end{equation}
Since $G^{-}$ satisfies $0  < \mu_{\beta , \omega}^{-}(G_{-}) \leq \mu_{\beta, \omega}^{+}(G_{-}) \leq \epsilon$, 
the inequality (\ref{eq:contradiction}) contradicts Lemma \ref{lem:slope-mor}. 
Thus we have finished the proof of the assertion (1). 

Similarly to the proof for (1), we easily see 
$\Hom_{\mb D}(E', E[-1]) \cong \Hom_{\mb D}(T', F)$. 
Since the $(\sigma_{1}, \sigma_{2})$-free part of $T' \in \til{\mca T}{\beta'}{\omega'}$ satisfies $\mu_{\beta, \omega}^{-}(T')>-\epsilon$, 
there exists a subobject $T'_{+}$ of $T'$ in $\gl{\mca A_{1}}{\mca A_{2}}$ such that
\begin{itemize}
\item the $(\sigma_{1}, \sigma_{2})$-free part $T'_{\mr {fr}}$ of $T'_{+}$ satisfies $\mu_{\beta, \omega}^{-} (T'_{\mr {fr}})>0$, and 
\item the quotient $T'_{-}:= T'/T'_{+}$ is $(\sigma_{1}, \sigma_{2})$-free with 
$-\epsilon < \mu_{\beta, \omega}^{-}(T'_{-}) \leq \mu_{\beta, \omega}^{+}(T'_{-}) \leq 0$. 
\end{itemize}

Then any morphism $\psi \colon T ' \to F$ lifts to $ \psi ' \colon T'_{-} \to F$ by $\mu_{\beta, \omega}^{+}(F)\leq 0$: 
\[
\xymatrix{
T'	\ar[r]^-{\psi}\ar[d]	&	F	\\
T'_{-}\ar[ur]_-{\psi'}
}. 
\]
Since the morphism $\psi'$ is in the torsion part $\til{\mca F}{\beta}{\omega}$ of the abelian category $\til{\mca A}{\beta}{\omega}[-1]$, 
the image $\im \psi'$ of $\psi '$ in $\til{\mca A}{\beta}{\omega}[-1]$ is also in $\til{\mca F}{\beta}{\omega}$. 
Truncating the Harder-Narasimhan filtration of $\im \psi '$ with respect to the $\mu_{\beta, \omega}$-semistability, 
we find a subobject $H_{+} \subset \im \psi '$ such that 
\begin{itemize}
\item $H_{+}$ is $(\sigma_{1}, \sigma_{2})$-free with $-\epsilon < \mu_{\beta, \omega}^{-} \leq \mu_{\beta, \omega}^{+}(H_{+}) \leq 0$, 
and 
\item the quotient $H_{-}:= \im \psi'/H_{+}$ is $(\sigma_{1}, \sigma_{2})$-free with $\mu_{\beta, \omega}^{+}(H_{-})\leq -\epsilon$. 
\end{itemize}
Thus we have the diagram: 
\[
\xymatrix{
	&	T'_{-}\ar[d]	\ar@{-->}[ld]&	\\
H_{+}	\ar[r]	&	\im \psi'	\ar[r]\ar[d]	&	H_{-}. 		\\
&F&
}
\]
By the assumption for $T'_{-}$, the morphism $T'_{-} \to \im \psi '$ factors through $H_{+}$. 
Moreover the composite $H_{+}\to \im \psi' \to F \to E[-1]$ is a mono morphism in $\til{\mca A}{\beta}{\omega}$ since any object in the horizontal line lives in the torsion part $\til{\mca F}{\beta}{\omega}$ of $\til{\mca A}{\beta}{\omega}[-1]$. 
Hence $H_{+}$ should satisfies 
\begin{equation}
\label{eq:contradiction2}
0	<	\arg \til{Z}{\beta}{\omega}(H_{+}) \leq \arg \til{Z}{\beta}{\omega}(E[-1]) \leq \arg (\beta+1 + 2\epsilon_{0}+\sqrt{-1}\omega)(\beta - \sqrt{-1}\omega). 
\end{equation}
Then the inequality above gives a contradiction applying Lemma \ref{lem:slope-mor} as $\epsilon_{2}=-\epsilon_{0}$ and $\epsilon_{1}=0$. 
\end{proof}

\begin{prop}\label{prop:conti-rational}
Let $\mb D=\sod{\mb D_{1}}{\mb D_{2}}$ be a semiorthogonal decomposition with $\m2$. 
Choose stability conditions $\sigma _{i} =(\mca A_{i}, Z_{i}) \in \Stab{\mb D_{i}}$ such that $(\sigma_{1}, \sigma_{2})$ satisfies the condition $\m
5$. 

Then the map $s_{\bb Q} \colon \mca H^{+}(\epsilon_{1})_{\bb Q} \cap \mca H^{-}(\epsilon_{2})_{\bb Q} \to \Stab{\mb D}$ defined by 
\begin{equation}
s_{\bb Q}(\beta, \omega)=\til{\Sigma}{\beta}{\omega}
\end{equation}
is continuous. 
\end{prop}

\begin{proof}

We first show that 
$\til{\mca P}{\beta'}{\omega'}(0,1] \subset \til{\mca P}{\beta}{\omega}(-1, 2]$ for any $(\beta, \omega)$ and $(\beta', \omega') $ in $\bb R \times \bb R_{>0}$. 
To show this, 
Let $E$ be in $\til{\mca P}{\beta}{\omega}(0,1] $ and $E'$ in $\til{\mca P}{\beta'}{\omega'}(0,1] $. 
We denote by $\mca A_{0}$ the heart $\gl{\mca A_{1}}{\mca A_{2}}$ on $\mb D$. 
Then cohomologies of $E$ and of $E'$ with respect to $\mca A_{0}$ are concentrated in degree $-1$ and $0$. 
Hence the vanishings $\Hom_{\mb D}(E[p], E')=\Hom_{\mb D}(E', E[-p])=0$ hold for any $p\geq 2$ by Lemma \ref{lem:heart}. 
Thus we see $\til{\mca P}{\beta'}{\omega'}(0,1] \subset \til{\mca P}{\beta}{\omega}(-1, 2]$. 

Take $0 < \varepsilon <1/8$ arbitrary. 
By Proposition \ref{prop:t-str-cont}, if $(\beta', \omega')$ is sufficiently close, then the inclusion $\til{\mca A}{\beta'}{\omega'} \subset \til{\mca P}{\beta}{\omega}(-1+\varepsilon, 2-\varepsilon]$ holds. 
By Corollary \ref{cor:fakesupport}, shrinking $(\beta', \omega')$ if necessary, we may assume that the following holds: 
\begin{equation}\label{eq:deformation}
\sup
\left\{
\frac{|\til{Z}{\beta'}{\omega'}(E)-\til{Z}{\beta}{\omega}(E)|}
{|\til{Z}{\beta}{\omega}(E)|}
\middle|
E \text{ is $\til{\Sigma}{\beta}{\omega}$-semistable}
\right\}
< \sin (\pi \varepsilon). 
\end{equation}
By \cite[Proposition 4.2]{MR2721656} we see that $\til{\Sigma}{\beta}{\omega}$ is in the open neighborhood $B_{\varepsilon}(\til{\Sigma}{\beta}{\omega})$ and this gives the proof. 
\end{proof}

\section{Non-rational coefficients stability conditions}

In the previous section, we have constructed stability conditions $\til{\Sigma}{\beta}{\omega}$ for rational points $(\beta, \omega)$ in 
the open set $\mca H^{+}(\epsilon_{1}) \cap \mca H^{-}(\epsilon_{2})$. 
These stability conditions $\til{\Sigma}{\beta}{\omega}$ could be deformed for non-rational points in $\mca H^{+}(\epsilon_{1}) \cap \mca H^{-}(\epsilon_{2})$ by Proposition \ref{prop:fake} and Theorem \ref{thm:Bridgeland}. 
We wish to show that the extended stability condition is given by the same construction as rational stability conditions $\til{\Sigma}{\beta}{\omega}$ and that 
the family $\{\til{\Sigma}{\beta}{\omega} \mid (\beta, \omega) \in \mca H^{+}(\epsilon_{1}) \cap \mca H^{-}(\epsilon_{2}) \}$ is continuous for $(\beta, \omega)$.


\begin{lem}\label{lem:torsion}
Let $\mb D=\sod{\mb D_{1}}{\mb D_{2}}$ be a semiorthogonal decomposition with $\m2$. 
Choose rational stability conditions $\sigma _{i} =(\mca A_{i}, Z_{i}) \in \Stab{\mb D_{i}}$ such that $(\sigma_{1}, \sigma_{2})$ satisfies the condition $\m
5$. 
Suppose $\tau =(\til{Z}{\beta}{\omega}, \mca Q)$ is a stability condition obtained by a deformation of $\til{\Sigma}{\beta_{0}}{\omega_{0}}$ for a rational point $(\beta_{0}, \omega_{0}) \in \mca H^{+}(\epsilon_{1})_{\bb Q} \cap \mca H^{-}(\epsilon_{2})_{\bb Q}$ (see also Definition \ref{dfn:deformation}). 
Any $(\sigma_{1}, \sigma_{2})$-torsion object is in $\mca Q(0,1]$. 
\end{lem}

\begin{proof}
Let $U$ be an open neighborhood of $(\beta, \omega)$ in $\mca H$. 
By the definition of $\til{\Sigma}{\beta'}{\omega'}$ for any rational point $(\beta', \omega') \in U_{\bb Q}$, any $(\sigma_{1}, \sigma_{2})$-torsion object $T$ is in $\til{\mca T}{\beta'}{\omega'}$, in particular, in $\til{\mca P}{\beta'}{\omega'}(0,1]$. 
Taking limit $(\beta', \omega')$ to $(\beta, \omega)$, 
we see that $T$ is in $\mca Q[0,1]$. 

Now we wish to show that $T$ is in $\mca Q(0,1]$. 
To show this, let $\phi^{-}_{\beta', \omega'}(T)$ be the phase of the maximal destabilizing quotient of $T$ with respect to $\til{\Sigma}{\beta'}{\omega'}$ for $(\beta', \omega') \in U_{\bb Q}$.  
We claim 
\begin{equation} \label{eq:torsion}
\phi^{-}_{\beta', \omega'}(T) \geq \frac{1}{\pi}\arg (-\beta' +\sqrt{-1}\omega').
\end{equation} 
In fact, otherwise, there exists a subobject $K$ of $T$ in $\til{\mca A}{\beta'}{\omega'}$ such that 
the quotient $Q= T/K$ is $\til{\Sigma}{\beta'}{\omega'}$-semistable with $\arg \til{Z}{\beta'}{\omega'}(Q) < \arg (-\beta' +\sqrt{-1}\omega')$. 
By Proposition \ref{prop:keyprop}, $Q$ is in $\til{\mca T}{\beta'}{\omega'}$. 
Moreover, $K$ is also in $\til{\mca T}{\beta'}{\omega'}$
since $\til{\mca T}{\beta'}{\omega'}$ is the free part of $\til{\mca A}{\beta'}{\omega'}$.  
Hence the sequence 
\begin{equation}
\xymatrix{
0	\ar[r]	&	K	\ar[r]	&	T	\ar[r]	&	Q	\ar[r]	&	0
}
\end{equation}
is exact not only in $\til{\mca A}{\beta'}{\omega'}$ but also in $\gl{\mca A_{1}}{\mca A_{2}}$. 
Thus $Q$ has to be $(\sigma_{1}, \sigma_{2})$-torsion. 
Then the definition of $\til{Z}{\beta'}{\omega'}$ implies the inequality
\[
	\arg (-\beta' +\sqrt{-1}\omega')	\leq \arg \til{Z}{\beta'}{\omega'}(Q)
\]
which contradicts the assumption for $Q$. 

Hence $\phi^{-}(T)$ has to satisfy (\ref{eq:torsion}). 
Since one can choose arbitrary close $(\beta', \omega') \in U_{\bb Q}$ to $(\beta ,\omega)$, 
$T$ belongs to $\mca Q(0,1]$. 
\end{proof}

\begin{lem}\label{lem:tilt-torsion}
Let $\mb D=\sod{\mb D_{1}}{\mb D_{2}}$ be a semiorthogonal decomposition with $\m2$. 
Choose rational stability conditions $\sigma _{i} =(\mca A_{i}, Z_{i}) \in \Stab{\mb D_{i}}$ such that $(\sigma_{1}, \sigma_{2})$ satisfies the condition $\m
5$. 
Suppose $\tau =(\til{Z}{\beta}{\omega}, \mca Q)$ is a stability condition obtained by a deformation of $\til{\Sigma}{\beta_{0}}{\omega_{0}}$ for a rational point $(\beta_{0}, \omega_{0}) \in \mca H^{+}(\epsilon_{1})_{\bb Q} \cap \mca H^{-}(\epsilon_{2})_{\bb Q}$. 
Then $\til{\mca T}{\beta}{\omega}	\subset 	\mca Q(0,1]$ and $\til{\mca F}{\beta}{\omega} \subset \mca Q(-1,0]$. 
\end{lem}

\begin{proof}
We first show $\til{\mca T}{\beta}{\omega}	\subset 	\mca Q(0,1]$. 
Let $E$ be in $\til{\mca T}{\beta}{\omega}$. 
If $E$ is $(\sigma_{1}, \sigma_{2})$-torsion, then $E$ is in $\mca Q(0,1]$ by Lemma \ref{lem:torsion}. 
Now suppose that the object $E$ is $(\sigma_{1}, \sigma_{2})$-free and $\mu_{\beta, \omega}$-semistable. 
Since $\mca Q(0,1]$ is closed under extension, it is enough to show that 
such an $E$ belongs to $\mca Q(0,1]$. 

Since the $\mu_{\beta, \omega}$-stability satisfies the support property by Proposition \ref{prop:morslope}, 
there exists an open neighborhood $U$ of $(\beta, \omega)$ such that $\mu_{\beta', \omega'}^{-}(E)>0$ for any $(\beta', \omega' ) \in U$. 
Hence $E$ is in $\til{\mca T}{\beta'}{\omega'}$ which is a subcategory of $\til{\mca P}{\beta'}{\omega'}(0,1]$ for any $(\beta', \omega')\in U_{\bb Q}$. 
Taking the limit $(\beta', \omega')$ to $(\beta, \omega)$, we see that  $E$ is in $\mca Q[0,1]$.

Now we claim the following: 
\begin{claim}\label{claim:length}
Notations being as above, any object in $\mca Q(0,1]$ is quasi isomorphic to a $2$-term complex concentrated in degree $0$ and $-1$ with respect to the heart $\gl{\mca A_{1}}{\mca A_{2}}$. 
\end{claim}

Take $F \in \mca Q(\phi)$ where $\phi \in (0,1]$. 
Since the deformation of a stability condition is locally unique, we can find 
a rational point $(\beta'', \omega'') \in \mca H^{+}(\epsilon_{1})_{\bb Q} \cap \mca H^{-}(\epsilon_{2})_{\bb Q}$ such that 
$\arg (-\beta+\sqrt{-1}\omega) < \arg (-\beta ''+\sqrt{-1}\omega'') $ and 
the stability condition $\til{\Sigma}{\beta''}{\omega''}$ satisfies $\tau  \in B_{\theta}(\til{\Sigma}{\beta''}{\omega''})$ for some $\theta>0$. 
Then $F$ is in $\til{\mca P}{\beta''}{\omega''}(\phi-\theta, \phi + \theta)$. 
Shrinking $(\beta'', \omega'')$ if necessary, we may assume $\phi -\theta >0$ and $\phi+\theta <  \arg (-\beta +\sqrt{-1}\omega)/\pi + 1$. 
Then we have the sequence 
\[
[\phi-\theta, \phi+\theta] \subset (0, \arg (-\beta +\sqrt{-1}\omega)/\pi +1 ) \subset (0, \arg (-\beta'' +\sqrt{-1}\omega'')/\pi +1 ).
\] 
Hence the object $F$ is given by an extension of objects in $ \til{\mca P}{\beta''}{\omega''}(0,1]$ and $ \til{\mca P}{\beta''}{\omega''}(1, \arg (-\beta'' +\sqrt{-1}\omega'')/\pi + 1]$. 
By Proposition \ref{prop:keyprop}, any object in $ \til{\mca P}{\beta''}{\omega''}(1, \arg (-\beta'' +\sqrt{-1}\omega'')/\pi +1 ]$ is in $\til{\mca T}{\beta''}{\omega''}[1]$, in particular, is concentrated in degree $-1$. 
Since any object in $\til{\mca P}{\beta''}{\omega''}(0,1]$ is concentrated in degree $-1$ and $0$, so is $F$.

To show $\til{\mca T}{\beta}{\omega}	\subset 	\mca Q(0,1]$, 
consider the distinguished triangle 
\begin{equation}\label{eq:bunkai}
\xymatrix{
E^{+}	\ar[r]	&	E	\ar[r]	&	E^{-}	\ar[r]	&	E^{+}[1], 
}
\end{equation}
where $E^{+} \in \mca Q(0,1]$ and 
$E^{-}\in \mca Q(0)$. 
It is enough to show that $E^{-}$ is zero. 
Taking the cohomology with respect to the heart $\gl{\mca A_{1}}{\mca A_{2}}$, 
we see that both $E^{+}$ and $E^{-}$ are concentrated in degree $0$ part by Claim \ref{claim:length}. 
Hence the triangle (\ref{eq:bunkai}) gives a short exact sequence in $\gl{\mca A_{1}}{\mca A_{2}}$.

Then the complex number $\til{Z}{\beta}{\omega}(E^{-})$ is in the ray $\bb R_{>0}$. 
If $E^{-}$ is $(\sigma_{1}, \sigma_{2})$-torsion with $\I \til{Z}{\beta}{\omega}(E^{-})=0$, then $E^{-}$ satisfies $\tau_{2}^{R}E^{-}=0$ and $\tau_{1}^{L}E^{-} \in \mca P_{\sigma_{1}}(1)$. 
This contradicts the fact $\til{Z}{\beta}{\omega}(E^{-}) \in \bb R_{>0}$. 
Hence $E^{-}$ is not $(\sigma_{1}, \sigma_{2})$-torsion. 
Thus we see 
$\I Z_{1}(\tau_{1}^{L}E^{-})+\I Z_{2}(\tau_{2}^{R}E^{-})\neq 0$, and 
the exact sequence (\ref{eq:bunkai}) implies 
\[
0 <  \mu_{\beta,\omega}(E) \leq \mu_{\beta, \omega }(E^{-}) 
\]
by the $\mu_{\beta, \omega}$-semistability of $E$. 
Hence $\I \til{Z}{\beta}{\omega}(E^{-})$ has to be positive and this contradicts $\til{Z}{\beta}{\omega}(E^{-}) \in \bb R_{>0}$. 
Thus $E^{-}$ has to be zero.

We secondly show  $\til{\mca F}{\beta}{\omega}	\subset 	\mca Q(-1, 0]$. 
Let $F$ be in $\til{\mca F}{\beta}{\omega}$ and let $U$ be an open neighborhood of $(\beta, \omega)$. 
For any rational point $(\beta', \omega ') \in U_{\bb Q}$, we have 
$\til{\mca F}{\beta'}{\omega'} \subset \til{\mca P}{\beta'}{\omega'}(-1,0]$ and 
$\til{\mca T}{\beta'}{\omega'} \subset \til{\mca P}{\beta'}{\omega'}(0,1]$. 
Since $F$ is in the heart $\gl{\mca A_{1}}{\mca A_{2}}$ which is the extension closure of $\til{\mca F}{\beta'}{\omega'} $ and 
$\til{\mca T}{\beta'}{\omega'} $, 
$F$ is in $\til{\mca P}{\beta'}{\omega'}(-1,1]$.

Taking the limit $(\beta', \omega')$ to $(\beta, \omega)$, we see $F \in \mca Q[-1,1]$.  
On the other hand any object in $\mca Q(-1) \subset \mca Q(-2, -1] $ is in the extension closure of $\gl{\mca A_{1}}{\mca A_{2}}[-2]$ and $\gl{\mca A_{1}}{\mca A_{2}}[-1]$ by Claim \ref{claim:length}. 
Since $F$ is in $\gl{\mca A_{1}}{\mca A_{2}}$, 
$F$ is actually in $\mca Q(-1,1]$. 
Thus we obtain the distinguished triangle 
\begin{equation}\label{eq:F}
\xymatrix{
F^{0}	\ar[r]	&	F	\ar[r]	&	F^{1}	\ar[r]	&	F^{0}[1], 
}
\end{equation}
where $F^{0} \in \mca Q(0,1]$ and $\mca F^{1} \in \mca Q(-1,0]$. 
Again, by Claim \ref{claim:length}, we see that 
both $F^{0}$ and $F^{1}$ are in $\gl{\mca A_{1}}{\mca A_{2}}$ by taking the cohomology of the sequence (\ref{eq:F}). 
Hence the sequence (\ref{eq:F}) gives a short exact sequence in $\gl{\mca A_{1}}{\mca A_{2}}$. 
Since $F^{0}$ is a subobject of $F$ in $\gl{\mca A_{1}}{\mca A_{2}}$, $F^{0}$ is also in $\til{\mca F}{\beta}{\omega}$. 
Hence $F^{0}$ satisfies $\mu_{\beta, \omega}^{+}(F^{0}) \leq 0$ which implies 
\begin{equation}
\I \til{Z}{\beta}{\omega}(F^{0})	\leq 0. 
\end{equation}
On the other hand, we have 
\begin{equation}
\I \til{Z}{\beta}{\omega}(F^{0})	\geq 0
\end{equation}
since $F \in \mca Q(0,1]$. 
Hence $F$ must satisfy $\I \til{Z}{\beta}{\omega}(F^{0})	=	0$ and $\R \til{Z}{\beta}{\omega}(F^{0})	\leq 0$. 
This contradicts the proof of Proposition \ref{prop:centralcharge}. 
Thus $F^{0}$ has to be zero and hence $F$ is in $\mca Q(-1,0]$. 
\end{proof}

\begin{prop}\label{prop:t-st_underderomation}
Let $\mb D=\sod{\mb D_{1}}{\mb D_{2}}$ be a semiorthogonal decomposition with $\m2$. 
Choose rational stability conditions $\sigma _{i} =(\mca A_{i}, Z_{i}) \in \Stab{\mb D_{i}}$ such that $(\sigma_{1}, \sigma_{2})$ satisfies the condition $\m
5$. 
Suppose $\tau =(\til{Z}{\beta}{\omega}, \mca Q)$ is a stability condition is a deformation of $\til{\Sigma}{\beta_{0}}{\omega_{0}}$ for a rational point $(\beta_{0}, \omega_{0}) \in \mca H_{\bb Q}$. 
Then the heart $\mca Q(0,1]$ is the tilting heart $\sod{\til{\mca T}{\beta}{\omega}}{\til{\mca F}{\beta}{\omega}[1]}$ defined in Definition \ref{dfn:torsionpairs}. 
\end{prop}

\begin{proof}
Let us fix the pair $(\beta, \omega)$ through the proof. 
Both $\mca Q(0,1]$ and $\til{\mca A}{\beta}{\omega}=\sod{\til{\mca T}{\beta}{\omega}}{\til{\mca F}{\beta}{\omega}[1]}$ are hearts of bounded $t$-structures on $\mb D$, it is enough to show 
\begin{equation}\label{eq:heartinclusion}
\sod{\til{\mca T}{\beta}{\omega}}{\til{\mca F}{\beta}{\omega}[1]}
\subset \mca Q(0,1]. 
\end{equation}
By Proposition \ref{lem:tilt-torsion}, 
we see both $\til{\mca F}{\beta}{\omega}[1]$ and $\til{\mca T}{\beta}{\omega}$ are contained in $\mca Q(0,1]$. 
Since $\mca Q(0,1]$ is a extension closure, (\ref{eq:heartinclusion}) holds. 
\end{proof}

\begin{thm}\label{thm:cotinuous}
Let $\mb D=\sod{\mb D_{1}}{\mb D_{2}}$ be a semiorthogonal decomposition with $\m2$. 
Choose rational stability conditions $\sigma _{i} =(\mca A_{i}, Z_{i}) \in \Stab{\mb D_{i}}$ such that $(\sigma_{1}, \sigma_{2})$ satisfies the condition $\m
5$.

Then the map 
\begin{equation}
 s\colon	\mca H^{+}(\epsilon_{1} ) \cap  \mca H^{-}(\epsilon_{2}) \to \Stab{\mb D};  (\beta, \omega) \mapsto \til{\Sigma}{\beta}{\omega}
\end{equation}
is continuous, where $\til{\Sigma}{\beta}{\omega}$ is the pair defined in Definition \ref{dfn:tiltingdeformation}. 
\end{thm}

\begin{proof}
Recall that the map is continuous on rational points in $\mca H^{+}(\epsilon_{1} ) \cap  \mca H^{-}(\epsilon_{2})$ by Proposition \ref{prop:conti-rational}. 
By Corollary \ref{cor:fakesupport} and \cite[Theorem 1.2]{MR2373143}, there exists a stability condition $\tau$ given by a deformation of $\til{\Sigma}{\beta_{0}}{\omega_{0}}$ and whose central charge is $\til{Z}{\beta}{\omega}$ where $(\beta_{0}, \omega_{0})\in \mca (H^{+}(\epsilon_{1})\cap \mca H^{-}(\epsilon_{2}))_{\bb Q}$. 
Moreover the heart of $\tau$ is just $\til{\mca A}{\beta}{\omega}$ by Proposition \ref{prop:t-st_underderomation}. 
Hence $\tau$ is the pair $\til{\Sigma}{\beta}{\omega}$ and the map $s$ is continuous. 
\end{proof}

\section{Support Property}\label{sc:support}

We have obtained the continuous family 
$ \mca S(\epsilon_{1}, \epsilon_{2}) = \{	\til{\Sigma}{\beta}{\omega}	\mid	(\beta, \omega)\in \mca H^{+}(\epsilon_{1} ) \cap  \mca H^{-}(\epsilon_{2})  	\}$ of stability conditions on $\mb D=\sod{\mb D_{1}}{\mb D_{2}}$. 
Unfortunately it was not possible to prove the support property for $\til{\Sigma}{\beta}{\omega}$ directly. 
In this section, we show that any $\til{\Sigma}{\beta}{\omega}$ satisfies the support property via ``specialization''. 
More precisely one can find a stability conditions $\til{\Sigma}{1}{0}$ which satisfies the support property in the boundary of the family.

The following proposition is a generalization of our previous work \cite[Proposition 4.8]{morphismstability}. 

\begin{prop}
\label{prop:support}
Let $\mb D=\sod{\mb D_{1}}{\mb D_{2}}$ be a semiorthogonal decomposition with $\m2$. 
Choose stability conditions $\sigma _{i} =(\mca A_{i}, Z_{i}) \in \Stab{\mb D_{i}}$ such that $(\sigma_{1}, \sigma_{2})$ satisfies the condition $\m
5$.

\begin{enumerate}
\item If $E \in \gl{\mca A_{1}}{\mca A_{2}}$ is $\gl{\sigma_{1}}{\sigma_{2}}$-semistable then $\tau_{1}^{L}(E)$ is $\sigma_{1}$-semistable, and $\tau_{2}^{R}E$ is $\sigma_{2}$-semistable. 
\item The equality $\arg Z_{1}(\tau_{1}^{L}E)=\arg Z_{1}(\Phi (\tau_{2}^{R}E))$ holds. 
\end{enumerate}
\end{prop}

\begin{proof}
Any $E \in \gl{\mca A_{1}}{\mca A_{2}}$ has the canonical decomposition: 
\[
\xymatrix{
0	\ar[r]	&	i_{2}\tau_{2}^{R}E	\ar[r]	&	E	\ar[r]	&	i_{1}\tau_{1}^{L}E	\ar[r]	&	0. 
}
\]
Suppose that $E$ is $\gl{\sigma_{1}}{\sigma_{2}}$-semistable. 
Then, by the condition $\m5$, we obtain the following: 
\begin{equation}\label{eq:602}
\arg Z_{1}(\Phi (\tau_{2}^{R}E))=\arg Z_{2}(\tau_{2}^{R}E)	\leq \arg \gl{Z_{1}}{Z_{2}}(E)	\leq Z_{1}(\tau_{1}^{L}E). 
\end{equation}

Suppose to the contrary that $\tau_{1}^{L}E$ is not $\sigma_{1}$-semistable.  
Then there exists a subobject $A$ of $\tau_{1}^{L}E$ which is $\sigma_{1}$-semistable with 
\begin{equation}\label{eq:603}
 \arg Z_{1}(\tau_{1}^{L}E) <\arg Z_{1}(A). 
\end{equation}

By Lemma \ref{lem:morsub}, there exists a subobject $F$ of $E$ such that $\tau_{1}^{L}F \cong A$ and $\tau_{2}^{R}F \cong \im g[-1]$ where $g$ is the composite $A \to \tau_{1}^{L}E \to \Phi(\tau_{2}^{R}E)$. 
Since $E$ is $\gl{\sigma_{1}}{\sigma_{2}}$-semistable,  we have 
\begin{equation}\label{eq:601}
\arg \gl{Z_{1}}{Z_{2}}(F)	\leq \arg \gl{Z_{1}}{Z_{2}}(E). 
\end{equation}
The $\sigma_{1}$-semistability of $A$ and the condition $\m2$ imply 
\begin{equation*}\label{eq:604}
\arg Z_{1}(A)	\leq Z_{1}(\im g) = \arg Z_{2}(\im g[-1]). 
\end{equation*}
By the construction of $F$, we have 
\begin{equation}\label{eq:605}
\arg Z_{1}(A)	
\leq 
\arg \gl{Z_{1}}{Z_{2}}(F)	\leq 
\arg Z_{2}(\im g[-1])	. 
\end{equation}
Then (\ref{eq:602}), (\ref{eq:603}) and (\ref{eq:605}) imply the following inequalities 
\begin{equation}
\arg \gl{Z_{1}}{Z_{2}}(E)	\leq \arg Z_{1}(\tau_{1}^{L}E)	<	\arg Z_{1}(A)	\leq 
\arg \gl{Z_{1}}{Z_{2}}(F)
\end{equation}
which contradict (\ref{eq:601}). 
Hence $\tau_{1}^{L}E$ is $\sigma_{1}$-semistable.

Now we wish to show that $\tau_{2}^{R} E$ is $\sigma_{2}$-semistable. 
Suppose to the contrary that $\tau_{2}^{R}E$ is not $\sigma_{2}$-semistable. 
Then there exists a quotient $B$ of $\tau_{2}^{R}E$ in $\mca A_{2}$ which is $\sigma_{2}$-semistable with 
\begin{equation}\label{eq:606}
\arg Z_{2}(B)	<	\arg Z_{2}(\tau_{2}^{R}E). 
\end{equation}
Then (\ref{eq:602}) implies 
 \begin{equation*}
 \arg Z_{2}(B) =\arg Z_{1}(\Phi B) < \arg Z_{1}(\tau_{1}^{L}E). 
 \end{equation*}
Since $\tau_{1}^{L}E$ is $\sigma_{1}$-semistable, we have 
$\Hom_{\mb D}(i_{1}\tau_{1}^{L}E[-1], i_{2}B)=\Hom_{\mb D_{1}}(\tau_{1}^{L}E, \Phi (B))=0$. 
Thus the morphism $q \colon   i_{2} \tau_{2}^{R}E \to i_{2}B$ lifts to $E$, that is, we have the following commutative diagram in $\mb D$: 
\[
\xymatrix{
i_{2}\tau_{2}^{R}E	\ar[r]^{q}	\ar[d]	&	i_{2}B\ar[d]_{\1}	\\
E	\ar[r]_{\bar q}		\ar[r]		&	i_{2}B. 	\\	
}
\]
The morphism $\bar q $ is an epi morphism in $\gl{\mca A_{1}}{\mca A_{2}}$ since $\tau_{1}^{L}q $ and $\tau_{2}^{R}q$ are respectively epi morphisms in $\mca A_{1}$ and in $\mca A_{2}$. 
Since $E$ is $\gl{\sigma_{1}}{\sigma_{2}}$-semistable, we have the inequalities 
\begin{equation}
\arg \gl{Z_{1}}{Z_{2}}(E)	\leq \arg \gl{Z_{1}}{Z_{2}}(i_{2}B)=\arg Z_{2}(B) 
\end{equation}
which contradicts (\ref{eq:602}) and (\ref{eq:606}). 
Hence $\tau_{2}^{R}E$ is $\sigma_{2}$-semistable.

For the proof of (2), suppose to contrary 
$\arg Z_{2}(\tau_{2}^{R}E) \neq \arg Z_{1}(\tau_{1}^{L}E)$. 
Then (\ref{eq:602}) implies 
\begin{equation}
\arg Z_{1}(\Phi (\tau_{2}^{R}E)) = \arg Z_{2}(\tau_{2}^{R}E)	< \arg Z_{1}(\tau_{1}^{L}E). 
\end{equation}
Since $\Phi(\tau_{2}^{R}E)$ is $\sigma_{1}$-semistable, 
Then we have 
\[
0=\Hom (\tau_{1}^{L}E, \Phi (\tau_{2}^{R}E) ) \cong \Hom (i_{1}\tau_{1}^{L}E, i_{2}\tau_{2}^{R}E[1])
\]
which implies $E \cong i_{1}\tau_{1}^{L}E \oplus i_{2}\tau_{2}^{R}E $. 
Since $E$ is $\gl{\sigma_{1}}{\sigma_{2}}$-semistable, either $\tau_{1}^{L}E$ or $ \tau_{2}^{R}E$ has to be zero. 
\end{proof}

\begin{cor}\label{cor:tikakuwocontrol}
Let $\mb D=\sod{\mb D_{1}}{\mb D_{2}}$ be a semiorthogonal decomposition with $\m2$. 
Choose stability conditions $\sigma _{i} =(\mca A_{i}, Z_{i}) \in \Stab{\mb D_{i}}$ such that $(\sigma_{1}, \sigma_{2})$ satisfies the condition $\m
5$. 

Then the supremum 
\begin{equation}
\sup
\left\{
\frac
{|Z_{2}(\tau_{2}^{R}E)|}
{|\gl{Z_{1}}{Z_{2}}(E)|}
\middle|
E \text{ is $\gl{\sigma_{1}}{\sigma_{2}}$-semistable}
\right\}
\end{equation}
is smaller than or equal to $1$. 
\end{cor}
\begin{proof}
Suppose that $E$ is $\gl{\sigma_{1}}{\sigma_{2}}$-semistable. 
By Proposition \ref{prop:support}, we have 
\begin{align*}
|\gl{Z_{1}}{Z_{2}}(E) |	&=	|Z_{1}(\tau_{1}^{L}E)| + |Z_{2}(\tau_{2}^{R}E)|		
\end{align*}
which directly implies the desired assertion. 
\end{proof}

\begin{prop}\label{prop:mitiwotunagu}
Let $\mb D=\sod{\mb D_{1}}{\mb D_{2}}$ be a semiorthogonal decomposition with $\m2$. 
Choose rational stability conditions $\sigma _{i} =(\mca A_{i}, Z_{i}) \in \Stab{\mb D_{i}}$ such that $(\sigma_{1}, \sigma_{2})$ satisfies the condition $\m
5$. 

If $(\beta, \omega)\in \bb R \times \bb R_{>0}$ satisfies $|\beta -1 +\sqrt{-1}\omega | < \sin (\pi \epsilon)$ for a sufficiently small $\epsilon > 0$ , then
the stability condition $\til{\Sigma}{\beta}{\omega}$ is in $B_{\epsilon}(\gl{\sigma_{1}}{\sigma_{2}})$
\end{prop}

\begin{proof}
By Corollary \ref{cor:tikakuwocontrol}, we have 
\[
|| \til{Z}{\beta}{\omega}-\gl{Z_{1}}{Z_{2}}	||_{\gl{\sigma_{1}}{\sigma_{2}}}	=	|| (\beta -1+\sqrt{-1}\omega)Z\circ \tau_{2}^{R}	||_{\gl{\sigma_{1}}{\sigma_{2}}} <\sin (\pi \epsilon). 
\]
Let $\mca Q$ be the slicing of the stability condition $\gl{\sigma_{1}}{\sigma_{2}}$. 
To complete the proof, it is enough to show the tilting heart $\til{\mca A}{\beta}{\omega}$ is contained in $\mca Q(-1+\epsilon, 2-\epsilon]$ by \cite[Proposition 4.2]{MR2721656}.

Recall $\til{\mca A}{\beta}{\omega}$ is the extension closure of $\til{\mca T}{\beta}{\omega}$ and $\til{\mca F}{\beta}{\omega}[1]$. 
The subcategory $\til{\mca T}{\beta}{\omega}$ is clearly contained in $\gl{\mca A_{1}}{\mca A_{2}} = \mca Q(0,1]$. 
The inclusion 
$\til{\mca F}{\beta}{\omega} \subset \gl{\mca A_{1}}{\mca A_{2}}$ 
implies $\til{\mca F}{\beta}{\omega}[1] \subset \mca Q(1,2]$. 
Hence we wish to show $\til{\mca F}{\beta}{\omega}[1] \subset \mca Q(1, 2-\epsilon]$.

Now take $E \in \til{\mca F}{\beta}{\omega}$. 
let $F$ be the maximal destabilizing quotient of $E$ with respect to the stability condition $\gl{\sigma_{1}}{\sigma_{2}}$. 
Since $\til{\mca F}{\beta}{\omega}$ is the free part of $\gl{\mca A_{1}}{\mca A_{2}}$, the subobject $F$ of $E$ in $\gl{\mca A_{1}}{\mca A_{2}}$ is also in $\til{\mca F}{\beta}{\omega}$. 
Lemma \ref{lem:comp-slope} implies the inequalities 
$0	<	\arg Z_{1}(\tau_{1}^{L}F)  \leq \arg (\beta +1+\sqrt{-1}\omega)$ and $0	<	\arg Z_{2}(\tau_{2}^{R}F)  \leq \arg (\beta +\sqrt{-1}\omega)$. 
Since $\gl{Z_{1}}{Z_{2}}(F)=Z_{1}(\tau_{1}^{L}F) +Z_{2}(\tau_{2}^{R}F)$, we see 
\begin{equation}
\til{\mca F}{\beta}{\omega } [1]\subset \mca Q(1, 1+ \theta], 
\end{equation}
where $\theta = \arg (\beta+ \sqrt{-1}\omega)/\pi$. 
By the assumption $|\beta -1 +\sqrt{-1}\omega | < \sin (\pi \epsilon)$, 
if $\epsilon$ is sufficiently small, then $\theta $ is smaller than $1/2$. 
Thus we may assume the inequality $1+\epsilon < 2-\epsilon$ holds for a sufficiently small $\epsilon$. 
Hence we have 
\[
\til{\mca F}{\beta}{\omega}[1]	\subset \mca Q(1, 1+\epsilon] \subset \mca Q(1,2-\epsilon]. 
\]
Thus, if $\epsilon $ is sufficiently small, then the tilting heart $\til{\mca A}{\beta}{\omega}$ is contained in 
$\mca Q(0, 2-\epsilon]$. 
\end{proof}

\begin{rmk}
If $(\beta, \omega)$ is sufficiently close, the stability condition $\til{\Sigma}{\beta}{\omega}$ is close to $\gl{\sigma_{1}}{\sigma_{2}}$. 
Hence $\gl{\sigma_{1}}{\sigma_{2}}$ can be regarded as $\til{\Sigma}{1}{0}$. 
\end{rmk}

\begin{cor}
Let $\mb D=\sod{\mb D_{1}}{\mb D_{2}}$ be a semiorthogonal decomposition with $\m2$. 
Choose rational stability conditions $\sigma _{i} =(\mca A_{i}, Z_{i}) \in \Stab{\mb D_{i}}$ such that $(\sigma_{1}, \sigma_{2})$ satisfies the condition $\m
5$. 

If $\sigma_{i}$ satisfy the support property, then any stability condition in $\mca H^{+}(\epsilon_{1}) \cap \mca H^{-}(\epsilon_{2})$ satisfies the support property. 
\end{cor}

\begin{proof}
Recall that the support property is open and closed by Remark \ref{rmk:openclosed}. 
By Proposition \ref{prop:mitiwotunagu}, 
if $(\beta, \omega)$ is sufficiently closed to $(1,0)$, then $\til{\Sigma}{\beta}{\omega}$ satisfies the support property. 
Since the set $\mca H^{+}(\epsilon_{1}) \cap \mca H^{-}(\epsilon_{2})$ is connected, any $\til{\Sigma}{\beta}{\omega}$ for $(\beta, \omega) \in \mca H^{+}(\epsilon_{1}) \cap \mca H^{-}(\epsilon_{2})$ satisfies the support property. 
\end{proof}

\section{Construction of a path}

The aim of this section is a construction of a path in the space of stability conditions on the category of morphisms. 
Before the construction, let us introduce notation. 
Let $(\mb D^{\leq 0}, \mb D^{\geq 0})$ be a bounded $t$-structure of a triangulated category $\mb D$. 
To simplify notation, we denote by $\mca A_{[-q, -p]}$ the subcategory $\mb D^{\leq q} \cap \mb D^{\geq p}$ if $p \leq q$. 
The subcategory $\mca A_{[p,p]}$ is nothing but the shift $\mca A[p]$ of $\mca A$. 
Moreover the subcategory $\mca A_{[p, p+1]}$ is the extension closure of $\mca A[p]$ and $\mca A[p+1]$.

Now let us specialize $\mb D$ as the category of morphisms. 
Let $\ms C^{\Delta^{1}}$ be the infinity category of functors from $\Delta^{1}$ to a stable infinity category $\ms C$. 
By \cite{higheralgebra}, $\ms C^{\Delta^{1}}$ is stable and hence the homotopy category $\ho{\ms C^{\Delta^{1}}}$ is triangulated.  
We refer to $\ho{\ms C^{\Delta^{1}}}$ as the category of morphisms in $\ho{\ms C}$. 

Recall three functors between $\ms C$ and $\ms C^{\Delta^{1}}$ introduced in (\ref{eq:threefunctors}). 
Then the adjoint pairs $d_{0}\dashv s$ and $s\dashv d_{1}$ determine semiorthogonal decompositions on 
$\ho{\ms C^{\Delta^{1}}}=\sod{\mb D_{1}^{i}}{\mb D_{2}^{i}}$ ($i\in \{0,1\}$) respectively. 
Both subcategories $\mb D_{1}^{0}$ and $\mb D_{2}^{1}$ are given by full sub category of $\ho {\ms C^{\Delta^{1}}}$ consisting of identity morphisms in $\ho{\ms C}$. 
Moreover $\mb D_{1}^{2}$ is the full subcategory of morphisms to zero objects, 
and $\mb D_{1}^{1}$ is the full subcategory of morphisms from zero objects: 
\begin{align}
\mb D_{1}^{0}	=	\mb D_{2}^{1}	&=	\{	[\1 \colon x\to x]	\mid x \in \ho{\ms C}	\},  \notag\\
\mb D_{2}^{0}	&=	\{	[y\to 0] \mid y \in \ho{\ms C}	\} = : \ho{\ms C_{/0}}, \text{ and} 	\label{eq:0-}\\
\mb D_{1}^{1}	&=	\{	[0 \to z ] \mid z \in \ho{\ms C}	\} =: \ho{\ms C_{0/}} \label{eq:-0}. 
\end{align}
These subcategories are canonically equivalent to $\ho{\ms C}$. 
Under the equivalence, the inclusion functor $\mb D_{2}^{0} \to \ho{\ms C^{\Delta^{1}}}$ (resp. $\mb D_{1}^{1} \to \ho{\ms C^{\Delta^{1}}}$) is denoted by $j_{!}$ (resp. $j_{*}$).

For a stability condition $\sigma =(\mca A, Z)\in \Stab{\mb D_{1}^{0}}$, 
set $(\sigma_{1}^{0}, \sigma_{2}^{0})\in \Stab{\mb D_{1}^{0}}\times \Stab{\mb D_{2}^{0}}$ by $(\sigma, \sigma[-1])$ 
and set $(\sigma_{1}^{1}, \sigma_{2}^{1})\in \Stab{\mb D_{1}^{1}}\times \Stab{\mb D_{2}^{1}}$ by $(\sigma[1], \sigma)$. 

The gluing hearts derived from the semiorthogonal decompositions $\ho{\ms C^{\Delta^{1}}}=\sod{\mb D_{1}^{i}}{\mb D_{2}^{i}}$ are denoted by $d_{i}^{*}\mca A$ respectively. 
Then we obtain the torsion pairs $(d_{i}^{*}\mca T_{\beta, \omega}, d_{i}^{*}\mca F_{\beta, \omega})$ on $d_{i}^{*}\mca A$ via Definition \ref{dfn:torsionpairs} and 
 the stability conditions $\Sigma_{\beta, \omega}^{\sigma_{1}^{i}, \sigma_{2}^{i}}$ which will be denoted by $d_{i}^{*}\sigma_{\beta, \omega}=(d_{i}^{*}\mca A_{\beta, \omega}, d_{i}^{*}Z_{\beta, \omega})$.  
By the construction of the continuous map $d_{i}^{*}$, 
the stability condition $d_{i}^{*}\sigma$ is nothing but $\Sigma_{1, 0}^{\sigma_{1}^{i}, \sigma_{2}^{i}}$ (see also \cite{morphismstability}).

\begin{lem}\label{lem:TF1}
Let $\ms C$ be a stable infinity category and let $\sigma  =(\mca A, Z)$ be a discrete stability condition on $\ho{\ms C}$. 
Take $f \in d_{0}^{*}\mca T_{\beta, \omega}$ and $g \in d_{1}^{*}\mca F_{\beta, \omega}$ arbitrary. 
For any integer $p \leq -1$, the vanishing $\Hom_{\ho{\ms C^{\Delta^{1}}}}(f, g[p])=0$ holds. 
%
\end{lem}

\begin{proof}
Consider the semiorthogonal decomposition $\ho{\ms C^{\Delta^{1}}} = \sod{\mb D_{1}^{1}}{\mb D_{2}^{1}}$ associated 
with $s \dashv d_{1}$. 
For $[g \colon z \to w] \in d_{1}^{*}\mca F_{\beta, \omega} $, the gluing morphism is nothing but 
the canonical morphism $v \colon \cof g \to z[1]$. 
Since $g \in d_{1}^{*}\mca F_{\beta, \omega}$, Lemma \ref{lem:mormono} implies that $v$ is a monomorphism in the abelian category $\mca A[1]$. 
Hence we see that both $d_{1}g=z$ and $d_{0}g=w$ are in $\mca A$ and that the morphism $g \colon z\to w$ is an epi morphism in $\mca A$. 

Moreover $f \in d_{0}^{*}\mca T_{\beta, \omega}$ satisfies $d_{1}f \in \mca A_{[-1, 0]}$, $d_{0}f \in \mca A$ and $\cof f \in \mca A$. 
Hence we easily see 
$\Hom(d_{1}f, d_{1}g[p])=\Hom(d_{0}f, d_{0}g[p])=\Hom (d_{1}f, d_{0}g[p-1])=0$ for $p \leq -2$. 
Then Corollary \ref{cor:keyvanishing} implies $\Hom(f, g[p])=0$ for $p \leq -2$. 

To discuss the case $p=-1$, suppose that $f$ is $(\sigma_{1}, \sigma_{2})$-torsion.  
Then $d_{0}f$ and $\cof f$ are both $\sigma$-torsion. 
Moreover the first cohomology $H^{1}(x)$ of $x=d_{1}f$ is also $\sigma$-torsion since $H^{1}(x)$ is a quotient of $\cof f$ in $\mca A$. 
Since $d_{1}g$ is $\sigma$-free, we see $\Hom(H^{1}(x)[-1], d_{1}g[-1])=0$. 
The trivial vanishing $\Hom(H^{0}(x), d_{1}g[-1])=0$ implies $\Hom(d_{1}f, d_{1}g[-1])=0$. 
Moreover the vanishings $\Hom(d_{0}f , d_{0}g[-1]) = \Hom (d_{1}f, d_{0}g[-2])=0$ hold by the cohomological degree reason. 
Again Corollary \ref{cor:keyvanishing} implies $\Hom(f, g[-1])=0$. 

Now let us suppose $f$ is not $(\sigma_{1}, \sigma_{2})$-torsion. 
Let $u$ be the gluing morphism $u \colon d_{0}f \to \cof f$ of $f$. 
Since $H^{0}(x)$ is the kernel $\ker u$ of $u$ in $\mca A$, we obtain the following distinguished triangle $\1 _{\ker u }\to f \to \bar f$ in $\ho{\ms C^{\Delta^{1}}}$ denoted by 
\begin{equation}
\xymatrix{
\ker u	\ar[r]	\ar[d]_{\1}	&	d_{1}f	\ar[r]	\ar[d]	&	\cok u[-1]	\ar[d]_{\bar f}\\
\ker u \ar[r]		 		&	d_{0}f	\ar[r] 		&	\im u	.  	\\
}
\end{equation}
By the adjunction $s \dashv d_{1}$, we see 
\[
\Hom_{\ho{\ms C^{\Delta^{1}}}}(\1 _{\ker u}, g[-1])\cong \Hom_{\ho{\ms C}}(\ker u, d_{1}g[-1])=0
\]
by degree reason. 
Hence we may assume that $u \colon d_{0}f \to \cof f$ is a mono morphism in $\mca A$ without loss of generality. 
Then $x$ is quasi-isomorphic to $(\cof f/d_{0}f)[-1]$. 

Take a morphism $\tau \in \Hom_{\ho{\ms C^{\Delta^{1}}}}(f, g[-1])$ arbitrary. 
We wish to show $\tau=0$. 
Since $d_{0}g[-1]$ is in $\mca A[-1]$, it is easily see $\Hom (d_{0}f, d_{0}g[-1])=\Hom(d_{1}f, d_{0}g[-2])=0$. 
Thus it is enough to show $d_{1}\tau=0$. 

Put the morphism $d_{1}\tau[1] \colon x[1]=\cof f/y \to z$ in $\mca A$ by $\tau_{1}$. 
We obtain the following commutative diagram in $\ho{\ms C}$: 
\begin{equation}
\xymatrix{
x	\ar[d]_-{f}\ar[r]^-{\tau_{1}[-1]}	&	z[-1]	\ar[d]^-{g[-1]}\\	
y	\ar[r]_-{d_{0}\tau}		&	w[-1]	. 
}
\end{equation}
The commutativity implies the composite $g \circ \tau_{1}$ is zero. 
Hence the morphism $\tau_{1}$ factors through the kernel $\ker g$, which means $\im \tau_{1} \subset \ker g$.  
Thus we obtain the following commutative diagram in $\ho{\ms C}$: 
\begin{equation}
\xymatrix{
x	\ar[r]\ar[d]_-{f}	&	\im \tau_{1}[-1]	\ar[d]\ar[r]	&	z[-1]	\ar[d]^-{g[-1]}	\\
y	\ar[r]		& 	0				\ar[r]	&	w[-1]
}
\end{equation}
By \cite[Lemma 2.11]{morphismstability}, 
the above diagram gives a sequence $f \to j_{!}(\im \tau_{1}[-1]) \to g[-1]$. 
We claim that the morphism $\rho \colon f \to j_{!}(\im \tau_{1}[-1])$ is an epimorphism in $d_{0}^{*}\mca A$ and 
the morphism $\rho ' \colon j_{!}(\im \tau_{1}) \to g$ is a monomorphism in $d_{1}^{*}\mca A$.

Let $h$ be the fiber of $\rho$ in $\ho{\ms C^{\Delta^{1}}}$. 
Then clearly $d_{0}h =y$, and $\cof h $ is the kernel of the morphism $\cof f \to \cof f/y \to \im \delta$ in $\mca A$. 
Hence $h$ is in $d_{0}^{*}\mca A$ which implies that $\rho$ is an epimorphism in $d_{0}^{*}\mca A$. 

To discuss $\rho'$, note that $j_{!}(\im \tau_{1})$ is a subobject $j_{!}(\ker g)$ in $d_{1}^{*}\mca A$ by $\im \tau_{1} \subset \ker g$. 
Since $g$ is epi, there exists a morphism $\iota \colon  j_{!}(\ker g) \to g$ in $\ho{\ms C^{\Delta^{1}}}$. 
Then one can easily see that  
the morphism $\iota $ is a monomorphism in $d_{1}^{*}\mca A$. 
Hence $j_{!}(\im \tau_{1})$ is a subobject of $g$ in $d_{1}^{*}\mca A$.

Since $j_{!}(\im \tau_{1}[-1])$ is a quotient of $f$ in $d_{0}^{*}\mca A$, 
$j_{!}(\im \tau_{1}[-1])$ is in $d_{0}^{*}\mca T_{\beta, \omega}$. 
Hence $j_{!}(\im \tau_{1}[-1])$ satisfies $\I d_{0}^{*}Z_{\beta, \omega}(j_{!}(\im \tau_{1}[-1]))$ is positive which implies
\begin{equation}\label{eq:801}
\beta \I Z(\im \tau_{1})-\omega \R Z(\im \tau_{1}) >0. 
\end{equation}
On the other hand, $j_{!}(\im \tau_{1})$ is in $d_{1}^{*}\mca F_{\beta, \omega}$ since 
it is a subobject of $g \in d_{1}^{*}\mca F_{\beta, \omega}$. 
Thus the non-positivity $d_{1}^{*}Z_{\beta, \omega}(j_{!}(\im \tau_{1})) \leq 0$ implies 
\begin{equation}\label{eq:802}
(\beta +1)\I Z(\im \tau_{1})-\omega \R Z(\im \tau_{1}) \leq 0. 
\end{equation}
Then (\ref{eq:802}) implies $\beta \I Z(\im \tau_{1})-\omega \R Z(\im \tau_{1}) \leq 0$ since $\Im Z(\im \tau_{1})$ is non-negative. 
This clearly contradicts (\ref{eq:801}). 
Hence $\im \tau_{1}$ is zero which implies $d_{1}\tau=0$. 
\end{proof}

\begin{lem}\label{lem:TF2}
Let $\ms C$ be a stable infinity category and let $\sigma  =(\mca A, Z)$ be a discrete stability condition on $\ho{\ms C}$. 
Take $f \in d_{0}^{*}\mca F_{\beta, \omega}$ and $g \in d_{1}^{*}\mca T_{\beta, \omega}$ arbitrary. 
For any integer $p \leq 0$, the vanishing $\Hom_{\ho{\ms C^{\Delta^{1}}}}(g, f[p])=0$ holds. 
%
\end{lem}

\begin{proof}
Since $f$ is in $d_{0}^{*}\mca F_{\beta, \omega}$, the gluing morphism $d_{0}f \to \cof f$ is mono by Lemma \ref{lem:morvanishing}. 
Hence $d_{1}f$ is in $\mca A[-1]$. 
Since $g $ is in $d_{1}^{*}\mca T_{\beta, \omega}$, in particular in $d_{1}^{*}\mca A$, 
$d_{1}g$ is in $\mca A$ and $d_{0}g$ is in $\mca A_{[0,1]}$. 
Hence Corollary \ref{cor:keyvanishing} implies the vanishing $\Hom(g, f[p])=0$ for $p \leq -1$. 

To discuss the case $p=0$, suppose that $g$ is $(\sigma_{1}, \sigma_{2})$-torsion. 
Then $d_{1}g$ and $\fib g$ are $\sigma$-torsion objects in $\mca A$. 
Moreover the $0$-th cohomology $H^{0}(d_{0}g)$ of $d_{0}g$ is $\sigma $-torsion since it is a quotient of the $\sigma$-torsion object $d_{1}g$. 
Then we see $\Hom(d_{0}g, d_{0}f)=0$ by the $\sigma$-freeness of $d_{0}f$. 
In addition the vanishings $\Hom(d_{1}g, d_{1}f)=\Hom(d_{1}g, d_{0}f[-1])=0$ hold by the cohomological degree reason. 
Thus Corollary \ref{cor:keyvanishing} implies $\Hom(g, f)=0$. 

Now suppose that $g$ is not $(\sigma_{1}, \sigma_{2})$-torsion. 
Let $u$ be the universal morphism $\fib g \to d_{1}g$. 
Then there exists a diagram of the distinguished triangle in $\ho{\ms C}$:
\begin{equation}
\xymatrix{
0	\ar[r]	\ar[d]	&	z	\ar[d]^-{g}	\ar[r]	&	z \ar[d]^-{\bar g}\\
\ker u[1]	\ar[r]	&	w		\ar[r]	&	\cok u. 
}
\end{equation}
Put the right horizontal arrow by $[\bar g \colon z \to \cok u]$ which gives an object in $\ho{\ms C^{\Delta^{1}}}$. 
By \cite[Lemma 2.11]{morphismstability}, the above digram gives a distinguished triangle in $\ho{\ms C^{\Delta^{1}}}$, in particular, a short exact sequence in $d_{1}^{*}\mca A$. 
Since the vanishings $\Hom (j_{*}(\ker u[1]), f) \cong \Hom(\ker u[1], d_{0}f)=0$ hold by the fact $d_{0}f \in \mca A$, 
it is enough to show $\Hom(\bar g , f)=0$. 

Take $\bar \varphi  \in \Hom(\bar g, f)$. 
Since $\Hom(d_{1}\bar g, d_{1}f)=\Hom(d_{1}\bar g, d_{0}f[-1])=0$ hold by the cohomological degree reason, 
it is enough to show $d_{0}\bar \varphi=0$.  
Since $(d_{0}\bar \varphi) \circ  \bar g = f \circ (d_{1}\bar \varphi)=0$, 
$\bar g$ factors through the kernel $\ker d_{0}\bar \varphi$. 
Since $\bar g$ is an epimorphism in $\mca A$, the canonical morphism $\ker d_{0}\bar \varphi \to \cok u$ has to be an epimorphism. 
Thus $d_{0}\bar \varphi$ is zero morphism.  
\end{proof}

\begin{prop}\label{prop:pathnoheart}
Let $\ms C$ be a stable infinity category and let $\sigma  =(\mca A, Z)$ be a locally finite stability condition on $\ho{\ms C}$. 
\begin{enumerate}
\item $d_{0}^{*}\mca T_{\beta, \omega} \subset (d_{1}^{*}\mca A_{\beta, \omega})_{[-1, 0]}$. 
\item $d_{0}^{*}\mca F_{\beta, \omega} \subset (d_{1}^{*}\mca A_{\beta, \omega})_{[ -2, -1]}$. 
\end{enumerate}
\end{prop}

\begin{proof}
Any object $g \in d_{1}^{*}\mca A_{\beta, \omega}$ has a decomposition 
\[
\xymatrix{
H^{-1}(g)[1] \ar[r]	&	g \ar[r]	&	H^{0}(g)
}
\]
with respect to the heart $d_{1}^{*}\mca A$ where 
$H^{-1}(g) \in d_{1}^{*}\mca F_{\beta, \omega}$ and $H^{0}(g)	 \in d_{1}^{*}\mca T_{\beta ,\omega}$. 
Clearly $d_{1}(H^{0}(g))$ is in $\mca A$, and $d_{0}(H^{0}(g)) $ is in $\mca A_{[0,1]}$. 
By Lemma \ref{lem:mormono}, both $d_{0}(H^{-1}(g))$ and $d_{1}(H^{-1}(g))$ are in $\mca A$. 
Hence both $d_{0}g$ and $d_{1}g$ are in $\mca A_{[0,1]}$. 

Take $f \in d_{0}^{*}\mca T_{\beta, \omega}$ and $g \in d_{1}^{*}\mca A_{\beta, \omega}$. 
To prove the assertion (1), it is enough to show 
\begin{enumerate}
\item[(1a)] $\Hom(g[p], f)=0$ for $p \geq 1$ and 
\item[(1b)] $\Hom(f, g[q])=0$ for $q \leq -2$. 
\end{enumerate}
Since $d_{1}g[p]$ and $d_{0}g[p]$ are in $\mca A_{[p,p+1]}$, 
$\Hom(d_{1}g[p], d_{1}f)$, $\Hom(d_{0}g[p], d_{0}f)$, and $\Hom(d_{1}g[p], d_{0}f[-1])$ vanish. 
Then Corollary \ref{cor:keyvanishing} implies (1a). 

Now suppose $g \in d_{1}^{*}\mca T_{\beta, \omega}$. 
Then $d_{1}g[q] \in \mca A[q]$ and $d_{0}g \in \mca A_{[q,q+1]}$.  
The assumption $q \leq -2$ implies the vanishings
\[
\Hom(d_{1}f, d_{1}g[q])=\Hom(d_{0}f, d_{0}g[q])=\Hom(d_{1}f, d_{0}g[q-1])=0. 
\]
Corollary \ref{cor:keyvanishing} implies $\Hom(f, g[q])=0$ for $g \in d_{1}^{*}\mca T_{\beta, \omega}$. 
Now Lemma \ref{lem:TF1} implies $\Hom(f, g[q])=0$ for $g \in d_{1}^{*}\mca F_{\beta, \omega}[1]$. 
Thus we have proven (1b).

Take $f \in d_{0}^{*}\mca F_{\beta, \omega}$ and $g \in d_{1}^{*}\mca A_{\beta, \omega}$. 
To prove the assertion (2), it is enough to show 
\begin{enumerate}
\item[(2a)] $\Hom(g[p], f)=0$ for $p\geq 0$ and 
\item[(2b)] $\Hom(f, g[q])=0$ for $q \leq -3$. 
\end{enumerate}
Lemma \ref{lem:TF2} implies $\Hom(g[p], f)=0$ for $g \in d_{1}^{*}\mca T_{\beta, \omega}$. 
Suppose $g \in d_{1}^{*}\mca F_{\beta, \omega}[1]$. 
Then $d_{1}g[p]$ and $d_{0}g[p]$ are in $\mca A[p+1]$. 
Since $f$ is in $d_{0}^{*}\mca F$, $d_{1}f$ is in $\mca A[-1]$ and $d_{0}f $ is in $\mca A$. 
Corollary \ref{cor:keyvanishing} implies $\Hom(g[p], f)=0$ for $g \in d_{1}^{*}\mca F_{\beta, \omega}[1]$. 
This gives the proof of claim (2a). 

To prove (2b), recall $g[q]$ satisfies 
that both $d_{0}^{*}g[q]$ and $d_{1}g[q]$ are in $\mca A_{[q, q+1]}$. 
Since $f$ satisfies $d_{0}f \in \mca A$ and $d_{1}f \in \mca A[-1]$, 
Corollary \ref{cor:keyvanishing} implies the claim (2b). 
\end{proof}

\begin{cor}\label{cor:heartnohenkei}
Let $\ms C$ be a stable infinity category and let $\sigma=(\mca A, Z)$ be a locally finite stability condition on $\ho{\ms C}$. 
Then $d_{0}^{*}\mca A_{\beta, \omega} \subset (d_{1}^{*}\mca A_{\beta, \omega})_{[-1,0]}$. 
\end{cor}

\begin{proof}
Since the heart $d_{0}^{*}\mca A_{\beta, \omega} $ is the extension closure of $d_{0}^{*}\mca T_{\beta, \omega}$ and 
$d_{0}^{*}\mca F_{\beta, \omega}[1]$, the assertion follows from Proposition \ref{prop:pathnoheart}. 
\end{proof}

\begin{thm}\label{thm:pathconnection}
Let $\ms C$ be a stable infinity category and let $\sigma=(\mca A, Z)$ be a locally finite stability condition on $\ho{\ms C}$. 
If $\sigma$ is rational, 
then $d_{0}^{*}\sigma$ and $d_{1}^{*}\sigma$ are path connected in $\Stab{\ho{\ms C^{\Delta^{1}}}}$. 
\end{thm}

\begin{proof}
Let $p$ be a path in $\bar{\mca H}$ given by 
\[
p	\colon [0,1] \to\bar{\mca H} ; p(t)=
\left(
\cos(2\pi t/3), \sin (2 \pi t /3)
\right). 
\]
Set $\epsilon_{1}$ as $1/3$ and $\epsilon_{2}$ as $-1/2$ in $\mca H^{+}(\epsilon_{1}) \cap \mca H^{-}(\epsilon_{2})$. 
Then the domain $\mca H^{+}(1/3) \cap \mca H^{-}(-1/2)$ contains the image $\Im p$ of the path $p$. 
Hence the set 
$\{	d_{0}^{*}\sigma_{\beta, \omega}	\mid	(\beta, \omega) \in \Im p	\}$ 
of stability conditions 
gives a continuous family. 

Now, when $(\beta, \omega)\in \Im p$, we denote by $d_{i}^{*} \sigma_{p(t)} =(d_{i}^{*}Z_{p(t)}, d_{i}^{*}\mca A_{p(t)}) $ instead of $d_{i}^{*}\sigma_{\beta, \omega}$. 
By direct calculation, one can check that 
\[
d_{0}^{*}Z_{p(2\pi/3)}(f)
=\exp(2\pi \sqrt{-1}/3)
d_{1}^{*}Z_{p(2\pi/3)}(f) 
\]
for any $f \in \ho{\ms C^{\Delta^{1}}}$. 
Corollary \ref{cor:heartnohenkei} implies that the heart of the stability condition $d_{0}^{*}\sigma_{p(2\pi/3)}$ is contained in 
$ (d_{1}^{*}\mca A_{p(2\pi/3)})_{[-1,1]}$. 
Then \cite[Proposition 4.1]{MR2721656} implies the desired assertion. 
\end{proof}

\begin{rmk}
If $\sigma$ is full then one can assume that $\sigma$ is rational. 
Hence for any full stability $\sigma$, $d_{0}^{*}\sigma$ is path connected to $d_{1}^{*}\sigma$. 
Theorem \ref{thm:pathconnection} is a generalization of our previous result \cite[Theorem 1.2]{morphismstability}. 
Hence we obtain the following: 
\end{rmk}

\begin{cor}
Let $\ms C$ be a stable infinity category and let $\Stabf{\ho{\ms C}}$ be the space of full stability conditions. 
Then the images of maps 
\[
d_{0}^{*}, d_{1}^{*} \colon \Stabf{\ho{\ms  C}} \rightrightarrows \Stabf{\ho{\ms C^{\Delta^{1}}}}
\]
are path connected to each other. 
In other word, 
the restricted map $d_{0}^{*}$ and $d_{1}^{*}$ gives the same map from the $0$-th homotopy of $\Stab{\ho{\ms C}}$ to that of $\Stab{\ho{\ms C^{\Delta^{1}}}}$
\end{cor}

\begin{proof}
Recall that the $0$-th homotopy $\pi_{0}(X)$ of a topological space $X$ is nothing but the set of connected components of $X$. 
Let $[x]$ be the equivalence class of a point $x \in X$ with respect to path connectedness. 
By Theorem \ref{thm:pathconnection}, $[d_{0}^{*}\sigma]$ is the same as $[d_{1}^{*}\sigma]$ for $\sigma \in \Stabf{\ho{\ms C}}$. 
\end{proof}

\subsection*{Acknoledgement}
The author would like to thank his family for their great support and thank Atsushi Kanazawa for his encouragement. 
This work is partially supported by JSPS KAKENHI Grant Number JP21K03212.

%

%
\newcommand{\etalchar}[1]{$^{#1}$}

\end{document}